\documentclass[a4paper,twoside,english]{amsart}
\usepackage{amsmath,amssymb,amsthm,amscd,amsfonts} 
\usepackage[utf8]{inputenc}
\usepackage{amssymb,fge} 
\usepackage{array}
\usepackage{booktabs}
\usepackage[justification=centering]{caption}
\usepackage{tikz}
\usepackage{tabularx}
\usepackage{subcaption}
\usepackage{todonotes}
\usepackage{mathtools}
\usepackage{mathrsfs}
\usepackage{comment}
\usepackage{graphicx}
\usepackage[left=3.2cm,right=3.5cm,top=3cm,bottom=3cm]{geometry}
\usepackage{indentfirst}           
\usepackage{underscore}
\usepackage{enumitem}
\usepackage{relsize}
\usepackage{varwidth}
\usepackage{mathtools} 
\usepackage{hyperref}
\hypersetup{
    colorlinks=true,
    linkcolor=blue,
    filecolor=magenta,      
    urlcolor=cyan,
}
\usepackage{amssymb}
\usepackage[all]{xy}

\usetikzlibrary{calc}
\usetikzlibrary{arrows.meta,calc,decorations.markings,fit,math}

\newtheorem{thm}{Theorem}
\newtheorem{cor}[thm]{Corollary}
\newtheorem{prop}[thm]{Proposition}
\newtheorem{lem}[thm]{Lemma}

\newtheorem{conjecture}[thm]{Conjecture}

\newcommand\scalemath[2]{\scalebox{#1}{\mbox{\ensuremath{\displaystyle #2}}}}

\theoremstyle{definition} 
\newtheorem{defin}[thm]{Definition}

\newtheorem{exa}[thm]{Example}
\newtheorem*{notation}{Notational convention}

\theoremstyle{remark}
\newtheorem{remark}[thm]{Remark}

\numberwithin{thm}{section}
\numberwithin{equation}{thm}
\newcommand{\be}{\begin{equation}}
\newcommand{\ee}{\end{equation}}

\newcommand{\bP}{\ensuremath{{\mathbb{P}}}}
\newcommand{\C}{\ensuremath{{\mathbb{C}}}}

\renewcommand{\P}{\ensuremath{{\mathbb{P}}}}
\newcommand{\Q}{\ensuremath{{\mathbb{Q}}}}

\newcommand{\R}{\ensuremath{{\mathbb{R}}}}

\newcommand{\p}{\ensuremath{{\mathfrak{p}}}}
\newcommand{\q}{\ensuremath{{\mathfrak{q}}}}

\newcommand{\cO}{\mathcal O}

\newcommand{\Kbar}{\ensuremath {\overline K}}

\newcommand{\scrP}{{\mathscr{P}}}
\newcommand{\scrS}{{\mathscr{S}}}

\DeclareMathOperator{\PrePer}{PrePer}

\DeclareMathOperator{\Gal}{Gal}

\DeclareMathOperator{\Spec}{Spec}

\DeclareMathOperator{\rad}{rad}

\usepackage[utf8]{inputenc}

\newcommand\abs[1]{\left|#1\right|}

\newcolumntype{L}[1]{>{\raggedright\let\newline\\\arraybackslash\hspace{0pt}}m{#1}}
\newcolumntype{C}[1]{>{\centering\let\newline\\\arraybackslash\hspace{0pt}}m{#1}}
\newcolumntype{R}[1]{>{\raggedleft\let\newline\\\arraybackslash\hspace{0pt}}m{#1}}

\title[Unicritical polynomials over $\scalemath{1.3}{abc}$-fields]{Unicritical polynomials over $\scalemath{1.3}{abc}$-fields: from uniform boundedness to dynamical Galois groups}
\author{John R. Doyle}
\email{john.r.doyle@okstate.edu}
\address{Department of Mathematics, Oklahoma State University, Stillwater, OK 74075 USA}

\author{Wade Hindes}
\email{wmh33@txstate.edu}
\address{Department of Mathematics, Texas State University, San Marcos, TX 78666 USA}
\date{August 2024}
\subjclass[2020]{Primary: 37P15; Secondary: 37P35, 37P05, 11G50, 11S20.}
\keywords{preperiodic points, uniform boundedness, abc conjecture, semigroup dynamics.}

\begin{document}
\begin{abstract} Let $K$ be a function field of characteristic $p\geq0$ or a number field over which the $abc$ conjecture holds, and let $\phi(x)=x^d+c \in K[x]$ be a unicritical polynomial of degree $d\geq2$ with $d \not\equiv 0,1\pmod{p}$. We completely classify all portraits of $K$-rational preperiodic points for such $\phi$ for all sufficiently large degrees $d$. More precisely, we prove that, up to accounting for the natural action of $d$th roots of unity on the preperiodic points for $\phi$, there are exactly thirteen such portraits up to isomorphism.
In particular, for all such global fields $K$, it follows from our results together with earlier work of Doyle-Poonen and Looper that the number of $K$-rational preperiodic points for $\phi$ is uniformly bounded---independent of $d$. That is, there is a constant $B(K)$ depending only on $K$ such that
\[\big|\PrePer(x^d+c,K)\big|\leq B(K)\]
for all $d\geq2$ and all $c\in K$. Moreover, we apply this work to construct many irreducible polynomials with large dynamical Galois groups 
in semigroups generated by sets of unicritical polynomials under composition.       
\end{abstract}
\maketitle
\section{Introduction} 

\subsection{Uniform boundedness of preperiodic points}
One of the major motivating problems in arithmetic dynamics is the dynamical uniform boundedness conjecture of Morton and Silverman \cite[p. 100]{MR1264933}. Here and throughout, for a field $K$ and a rational function $\phi(x)\in K(x)$, we denote by $\PrePer(\phi,K) \subseteq \P^1(K)$ the set of $K$-rational preperiodic points for $\phi$.

\begin{conjecture}[Morton-Silverman]\label{conj:ubc}
    Let $D \ge 1$ and $d \ge 2$ be integers. Then there exists a constant $B := B(D,d)$ such that if $K$ is a number field of degree $[K:\Q] = D$ and $\phi\in K(x)$ is a rational function of degree $d$, then
        \[
            \big|\PrePer(\phi,K)\big| \le B.
        \]
\end{conjecture}

Conjecture~\ref{conj:ubc} is still open for all $D \ge 1$ and $d \ge 2$. To illustrate the difficulty of proving the conjecture, we note that the $d = 4$ case alone would imply Merel's theorem on torsion points on elliptic curves.

A great deal of work has been done on Conjecture~\ref{conj:ubc}, especially in the case of polynomial maps. We recommend \cite[\textsection 4]{MR4007163} and \cite[\textsection 2]{MR4693952} for recent surveys of progress toward Conjecture~\ref{conj:ubc}; here we highlight the two results most relevant to the present article:
    \begin{itemize}
        \item In \cite[Theorem 1.7]{MR4065068}, Poonen and the first author proved a version of Conjecture~\ref{conj:ubc} for unicritical polynomials over function fields: If $k$ is a field of characteristic $p \ge 0$ and $K$ is the function field of an integral curve over $k$, then for all $d$ not divisible by $p$ and all $c \in K \setminus \bar k$ there is an upper bound for $\big|\PrePer(x^d + c, K)\big|$ depending only on $p$, $d$, and the \emph{gonality} of $K$. We recall that the gonality of $K$ may be defined as
            \[
                \min_{t\in K\setminus \bar k} \ [K : k(t)],
            \]
so gonality is the appropriate function field analogue of the absolute degree $[K:\Q]$ appearing in Conjecture~\ref{conj:ubc}.
        \item Looper showed in \cite[Theorem 1.2]{looper2021dynamical} that if one assumes the $abc$ conjecture (resp., the $abcd$ conjecture), then the following holds: For all number fields $K$, all $c \in K$, and all $d \ge 5$ (resp., $d \ge 2$), there is an upper bound for $\big|\PrePer(x^d + c, K)\big|$ depending only on $K$ and $d$.
    \end{itemize}

One of our main results (Corollary~\ref{cor:strong-ubc}) is to show that the bounds obtained in \cite{MR4065068,looper2021dynamical} do not actually depend on $d$. We do so by connecting preperiodic points for unicritical polynomials to solutions to a certain Fermat-Catalan equation, then using height bounds on solutions to such equations over \emph{$abc$-fields}---that is, fields over which the appropriate form of the $abc$ conjecture holds:

\begin{defin}\label{defn:abc-field}
    We call a field $K$ an {\bf $\boldsymbol{abc}$-field} if $K$ is the function field of a curve over an arbitrary constant field $k$ or a number field for which the $abc$-conjecture (Conjecture~\ref{conj:abc}) holds.
\end{defin}

\begin{remark}
    As all of our main results are valid over $abc$-fields, we note, in particular, that every result is unconditionally true over function fields (of arbitrary characteristic).
\end{remark}

A similar approach was used in \cite{MR4432520} to classify the preperiodic portraits of large-degree unicritical polynomials over the rational numbers; see also \cite{zhang2024abcd} for a summary of recent results in arithmetic dynamics using some standard conjectures in arithmetic geometry (e.g., the $abc$ conjecture).

We now set some notation: We let $\mu_K$ (resp., $\mu_{K,d}$) denote the set of roots of unity (resp., $d$th roots of unity) in the field $K$. We denote by $h$ the naive (logarithmic) height on $K$; see Section~\ref{sec:abc} for a precise definition. If $K$ is the function field of a curve, we denote by $g_K$ the genus of $K$. Finally, for a prime $p$ and a positive integer $d$, we denote by $d_p$ the prime-to-$p$ part of $d$; that is, one can write $d = d_p p^n$ with $n \ge 0$ and $p\nmid d_p$. In the case that $p = 0$ (in the context of a field having characteristic $p = 0$), we set $d_0 := d$ and interpret $p^n$ as $1$.

We first prove the following result for unicritical polynomial maps of large degree: 
\begin{thm}\label{thm:uniform-bound}
    Let $K$ be an $abc$-field of characteristic $p\geq0$. There are constants $C_1 = C_1(K)$ and $D_1 = D_1(K)$ such that for all $d \ge 2$ with $d\not\equiv 1\pmod p$ and $d_p > D_1$, the following statements hold:\vspace{.1cm} 
        \begin{enumerate}
            \item[\textup{(1)}] For all nonzero (and non-constant, in the function field case) $c \in K$, the map $x^d + c$ has no $K$-rational points of period greater than $3$.\vspace{.2cm} 
            \item[\textup{(2)}] If $h(c) > C_1$ and $\PrePer(x^d + c, K)$ is non-empty, then 
		\[
			c = y - y^d \quad\text{and}\quad \PrePer(x^d + c, K) = \{\zeta y : \zeta \in \mu_{K,d}\}  
		\]
            for some unique $y \in K$. \vspace{.2cm}  
	\item[\textup{(3)}] If $h(c) \le C_1$, then all preperiodic points for $x^d + c$ have height $0$. In particular, the only preperiodic points are $0$ and roots of unity in the number field case and constant points in the function field case. \vspace{.1cm} 
        \end{enumerate}
Moreover, one can take $C_1 = 0$ and $D_1 = \max\{40, 20g_K + 20\}$ when $K$ is a function field. 
\end{thm}
\begin{remark}
In part (1) of the theorem, the restrictions on $c$ really are necessary. If we let $c=0$, then for any integers $N, d_0 \ge 2$, the points of period $N$ for the map $x^{d_0}$ lie in the set $\mu_{d_0^N-1}$ of $(d_0^N -1)$th roots of unity. Thus, if $K$ contains $\mu_{d_0^N - 1}$, then for any degree $d \equiv d_0 \pmod{d_0^N - 1}$, the map $x^d$ acts on $\mu_{d_0^N - 1}$ in exact the same way as $x^{d_0}$, so we have an infinite set of $d \ge 2$ for which $x^d$ has a $K$-rational $N$-cycle.

Similarly, we must require $c$ non-constant in the function field case: for example, if the constant subfield $k \subset K$ is algebraically closed and $c \in k$, then $x^d + c$ has infinitely many preperiodic points, all in $k$.
\end{remark}

\begin{remark}
    The only primes $p$ for which $d_p \ne d$ are the primes $p > 0$ which divide $d$. However, we note that in the hypothesis of Theorem~\ref{thm:uniform-bound} that $d_p$ be sufficiently large, $d_p$ cannot be replaced by $d$. For example, if $d = p^n$ (so that $d_p = 1$), then a map of the form $x^d + c$ can actually have infinitely many $K$-rational preperiodic points. Indeed, suppose $K$ is a function field over an algebraically closed field $k$ of characteristic $p > 0$, and let $d = p^n$. For any $t \in K$, set $\phi(x) = x^d + t - t^d$ and $f(x) = x + t$. Then $f^{-1}\circ\phi\circ f$ is the power map $x^d$ (in particular, $\phi$ is \emph{isotrivial} over $K$), which has as its preperiodic points $0$, $\infty$, and the (infinitely many) roots of unity in $k$. This implies that $\phi$ has infinitely many $K$-rational preperiodic points, namely $\infty$ and all points of the form $t + \zeta$ with $\zeta = 0$ or $\zeta \in \mu_k$.
\end{remark}

Combining Theorem~\ref{thm:uniform-bound} with prior results in \cite{MR4065068,looper2021dynamical}, we obtain the following version of uniform boundedness, independent of degrees, for $abc$-fields. 

\begin{cor}\label{cor:strong-ubc}
Let $K$ be an $abc$-field of characteristic $p \ge 0$. Define a constant
    \[
        D_2(K) :=
            \begin{cases}
                2, &\text{ if the $K$ is a function field or a number field for}\\
                & \qquad\qquad \text{which the $abcd$ conjecture from \cite{looper2021dynamical} holds;}\\
                5, &\text{ otherwise}.
            \end{cases}
    \]
Then there is a bound $B(K)$ depending only on $K$ such that, for all $c\in K$ (non-constant, when $K$ is a function field) and all $d\geq D_2(K)$ with $d \not\equiv 0,1\pmod{p}$, we have
    \[
        \big|\PrePer(x^d + c, K)\big| \le \max\big\{B(K),|\mu_{K,d}|\big\}.
    \]
In particular, when $\mu_K$ is finite (e.g., when $K$ is a global field), this upper bound is independent of $d$.
\end{cor}

\begin{remark}
    The $K = \Q$ case of Corollary~\ref{cor:strong-ubc} was shown by Panraksa in \cite[Theorem 4]{MR4432520} (assuming the $abc$ conjecture).
\end{remark}

Corollary~\ref{cor:strong-ubc} says that for an $abc$-field $K$, the number of $K$-rational preperiodic points for maps of the form $x^d + c \in K[x]$ can be uniformly bounded independent of $d\ge2$ (or $d\ge5$ if the $abcd$ conjecture fails for $K$). By carefully analyzing the possible actions of unicritical polynomial maps on points of height zero (i.e., roots of unity and $0$ in the number field case, and constant points in the function field case), we are able to prove something stronger: we completely classify, up to isomorphism, all possible \textit{preperiodic portraits} for maps of the form $x^d+c$ over $K$ with $d \gg_K 0$. Specifically, we prove that there are thirteen graphs that illustrate the dynamics of $x^d + c$ on $\PrePer(x^d + c, K)$, up to the natural action of $\mu_{K,d}$, for all $d\gg_K 0$ and all $c\in K$; see Section~\ref{sec:portraits}---specifically Theorem~\ref{thm:prep-classification}---for a more precise statement.     

\subsection{Irreducibility and Galois groups in unicritically generated semigroups}
We next apply this work on preperiodic points to the construction of irreducible polynomials in semigroups generated by unicritical polynomials. Before we state our results along these lines, we briefly discuss the problem of constructing irreducible polynomials dynamically in a more general framework. 

For a given field $K$, it can often be an interesting and subtle problem to explicitly construct irreducible polynomials over $K$ of arbitrarily large degree. This is especially the case when one makes minimal assumptions about the coefficients of the defining polynomials, rendering the straightforward application of well-known tests (e.g., Eisenstein's criterion) unlikely. In this paper, we study the problem of constructing irreducible polynomials of arbitrarily large degree via composition over function fields and over number fields, assuming the $abc$ conjecture. In particular, we aim to construct a \emph{large} set of irreducible polynomials in the following sense. 
\begin{defin}\label{def:proportion} Let $K$ be a field, let $S=\{f_1,\dots,f_r\}$ be a finite set of elements of $K[x]$ all of degree at least two, and let $M_S$ be the semigroup generated by $S$ under composition. Then we say that $M_S$ \textbf{contains a positive proportion of irreducible polynomials} over $K$ if,  
\[\liminf_{B\rightarrow\infty}\,\frac{\Big|\big\{f\in M_S\,: \deg(f)\leq B\;\text{and $f$ is irreducible in $K[x]$}\big\}\Big|}{\Big|\big\{f\in M_S\,: \deg(f)\leq B\big\}\Big|}>0.\]
\end{defin}
\begin{remark}\label{rem:length} When $M_S$ is a free semigroup then one may make a similar definition of the proportion above using lengths of polynomials as words instead of degrees. However, we have used degrees since this option makes no assumptions on the structure of the underlying semigroup. Moreover, counting arguments in \cite[\S2]{bell2023counting} can be used to pass information back and forth between these perspectives.      
\end{remark}
With a little thought one can see that the generating set $S$ should contain at least one irreducible polynomial; otherwise, the full semigroup $M_S$ contains no irreducible polynomials, so the proportion above is zero. But is this enough? Perhaps the first non-trivial case to investigate this question is to consider sets of unicritical polynomials. This problem was first considered over $K=\mathbb{Q}$ when $c\in\mathbb{Z}$ (and some special cases of $c\in\mathbb{Q}$) in \cite{Mathworks2,Mathworks1}, and in this case, the unicritical polynomials with \emph{rational periodic points that are simultaneously perfect powers} provide the main source of difficulty. In particular, such polynomials will also play a prominent role in our analysis over number fields and function fields.
With this in mind, for $B\geq 0$ and $D\geq 2$, we define 
\begin{equation}\label{eq:smallhtordeg} 
\mathcal{S}(B,D):=\big\{x^d+c\,: c\in K,\; d\geq 2,\;\text{and}\; h(c)\leq B\; \text{or}\; d\leq D\big\}
\end{equation} 
to be the set of unicritical polynomials over $K$ with \emph{small constant term or small degree}. Furthermore, for $P\in K$, $n\geq1$, and $\mathfrak{d}=\{d_1,\dots,d_s\}$, we define the set 
\begin{equation}\label{eq:oneparameter+irre} 
\scalemath{.95}{\mathcal{I}(P,n,\mathfrak{d}):=\Big\{x^d+(\zeta P)-(\zeta P)^d\;:\;d\in\mathfrak{d}\,\;\text{and}\;\,  \zeta\in\mu_{K,n}\Big\}},  
\end{equation} 
which will contain the exceptional \emph{irreducible} polynomials in our analysis to come. In particular, we note that the maps in $\mathcal{I}(P,n,\mathfrak{d})$ all have a fixed point in $K$, namely $\zeta P$, and that all of the maps in $\mathcal{I}(P,n,\mathfrak{d})$ have the same primes of bad reduction, i.e., the places $v$ of $K$ with $v(P)<0$. Moreover, these sets are fairly small in cardinality:
\[\big|\mathcal{I}(P,n,\mathfrak{d})\big|\leq \big|\mu_{K,n}\big|\cdot |\mathfrak{d}|.\]
(Here we do not assume that the degrees $d_1,\dots,d_s$ are distinct, so even if $s$ is large, $|\mathfrak{d}|$ could still be $1$.). Likewise, for $P,n$, and $\mathfrak{d}$ as above, we define the set    
\begin{equation}\label{eq:oneparameter+red} 
\scalemath{.95}{\mathcal{R}(P,n,\mathfrak{d}):=\Big\{x^d+\zeta P\;:\;d\in\mathfrak{d}\;\;\text{and}\;\;\zeta\in\mu_{K,n}\Big\}},  
\end{equation} 
which will contain the exceptional \emph{reducible} polynomials in our analysis; note also that 
\[\big|\mathcal{R}(P,n,\mathfrak{d})\big|\leq \big|\mu_{K,n}\big|\cdot \big|\mathfrak{d}\big|,\]
as in the previous case. 

With these definitions in place, we have our main irreducibility theorem, which says roughly that a semigroup generated by unicritical polynomials over an $abc$-field contains a positive proportion of irreducible polynomials if and only if it contains at least one such polynomial, unless the generating set has small coefficients, has small degree, or sits inside one of the special sets in \eqref{eq:oneparameter+irre} and \eqref{eq:oneparameter+red} above.            
\begin{thm}\label{thm:main+irred} Let $K$ be an $abc$-field of characteristic zero and let $S=\{x^{d_1}+c_1,\dots,x^{d_s}+c_s\}$ for some $c_1,\dots,c_s\in K^\times$ and some $d_1,\dots,d_s\geq2$. Then there are constants $C_3=C_3(K)$ and $D_3=D_3(K)$ depending only on $K$ such that if $S$ contains an irreducible polynomial $x^d+c$ with $h(c)>C_3$ and $d>D_3$, then at least one of the following statements must hold: \vspace{.15cm} 
\begin{enumerate}
\item[\textup{(1)}] The semigroup $M_S$ contains a positive proportion of irreducibles over $K$, and there is a positive lower bound on this proportion depending only on $K$ and $d_1,\dots,d_s$. \vspace{.25cm}      
\item[\textup{(2)}] The set $S$ is of the special form
\vspace{.075cm} 
\[S\subseteq \mathcal{S}\big(C_3,D_3\big)\cup\mathcal{I}\big(P,d,\mathfrak{d}\big)\cup \mathcal{R}\big(P,d,\mathfrak{d}\big)
\vspace{.075cm} \]
for some $P\in K$ and $\mathfrak{d}:=\{d_1,\dots,d_s\}$. Furthermore, $P=ry^m$ for some $r\in\{\pm{1},\pm{4}\}$, some $y\in K$, and some $m\geq2$. \vspace{.25cm} 
\end{enumerate} 
Moreover, one can take $C_3=0$ and $D_3=\max\{14,4g_K+10\}$ in the function field case.       
\end{thm}
\begin{remark}\label{rem:main+irred} In fact, outside of the special sets $S$ in case (2) of Theorem \ref{thm:main+irred}, we show that there are maps $\phi_i,\phi_j\in S$ and large iterates $N_i=N(K,d_i)$ and $N_j=N(K,d_j)$ such that one of the following subsets 
\[
\big\{\phi_i^{N_i}\circ f\,:\,f\in M_S\big\}
\;\;\;\text{or}\;\;\;
\big\{\phi_i^{N_i}\circ\phi_j\circ f\,:\,f\in M_S\big\}
\;\;\;\text{or}\;\;\;
\big\{\phi_i^{N_i}\circ\phi_j^{N_j}\circ f\,:\,f\in M_S\big\}
\]
is a set of irreducible polynomials in $K[x]$, and we determine when each type is needed in Proposition \ref{prop:nopoweredfixedpoints}, Proposition  \ref{prop:2specialirred,unrelated}, and Proposition \ref{prop:oneirred+onered} respectively.      
\end{remark}
In particular, we note the following consequence of Theorem \ref{thm:main+irred} in the number field case; recall that $\mu_K$ denotes the set of roots of unity in $K$, a finite set when $K$ is a number field.
\begin{cor}\label{cor:numberfield} 
Let $K$ be a number field, let $S$, $C_3$, and $D_3$ be as in Theorem \ref{thm:main+irred},
and assume that the $abc$ conjecture holds over $K$. Moreover, assume that $\min_i\{h(c_i)\}>C_3$, that $\min_i\{d_i\}>D_3$, and that at least one of the following additional conditions is satisfied: \vspace{.15cm}   
\begin{enumerate}
    \item[\textup{(1)}] $\big|S\big|>2\big|\mu_K\big|\cdot \big|\{d_1,\dots, d_s\}\big|$, or \vspace{.2cm}
    \item[\textup{(2)}] there is a place $v$ of $K$ and indices $i$ and $j$ such that $v(c_i)<0$ and $v(c_j)\geq0$. \vspace{.15cm}     
\end{enumerate} 
Then $M_S$ contains a positive proportion of irreducible polynomials over $K$ if and only if it contains at least one irreducible polynomial over $K$.  
\end{cor}

Finally, together with a primitive prime divisor argument (see Proposition \ref{prop:newprime} below), we give a consequence of Theorem \ref{thm:main+irred} for dynamical Galois groups. In what follows, we fix a probability measure $\nu$ on $S$ satisfying $\nu(\phi)>0$ for all $\phi\in S$ and let $\bar{\nu}=\nu^{\mathbb{N}}$ denote the associated product measure on $\Phi_S=S^{\mathbb{N}}$, the set of all infinite sequences of elements of $S$. Then to an infinite sequence $\gamma=(\theta_1,\theta_2\dots)\in \Phi_S$ we let $K_n(\gamma)$ denote the splitting field of the polynomial $\gamma_n=\theta_1\circ\dots\circ\theta_n$ over $K$. Moreover, we say that a sequence $\gamma$ is \emph{stable} over $K$ if $\gamma_n$ is irreducible for all $n\geq1$. Finally, we let $[C_d]^\infty$ denote the infinite iterated wreath product of the cyclic group of order $d$.

We prove the following result on the size of the Galois groups of $\gamma_n$ for many sequences $\gamma$ and many indices $n$. Note that we are assuming here that the degrees of the polynomials in the set $S$ are all equal. 

\begin{thm}\label{cor:BigGalois} 
Let $K$ be an $abc$-field of characteristic zero, let $S=\{x^d+c_1,\dots,x^d+c_s\}$ for some $c_1,\dots,c_s\in K$ and $d\geq2$, and let $C_3$ and $D_3$ be as in Theorem \ref{thm:main+irred}. Moreover, assume that $\min_i\{h(c_i)\}>C_3$, that $d>\max\{12,D_3\}$, that $S$ contains an irreducible polynomial over $K$, and that at least one of the following additional conditions is satisfied: \vspace{.075cm}   
\begin{enumerate}
    \item[\textup{(1)}] $\big|S\big|>2\big|\mu_{K,d}\big|$, or \vspace{.15cm}
    \item[\textup{(2)}] there is a place $v$ of $K$ and indices $i$ and $j$ such that $v(c_i)<0$ and $v(c_j)\geq0$. \vspace{.075cm}     
\end{enumerate}  
Then there exists a sequence $\gamma\in\Phi_S$ whose dynamical Galois group is a finite index subgroup of $[C_d]^{\infty}$. In fact, there are infinitely many such sequences when $|S|\geq2$. Moreover, 
\[\bar{\nu}\Big(\big\{\gamma\in \Phi_S\,: \text{$\gamma$ is stable over $K$ and $\big[K_{n}(\gamma)\,:\,K_{n-1}(\gamma)\big]=d^{d^{n-1}}$  i.o.} \big\}\Big)\]
is positive. 
\end{thm}
\begin{remark}\label{rem:newramifiedprime} On the other hand, even in the case of unequal degrees one can still say something. Namely, there exists a sequence $\gamma\in\Phi_S$ with a new ramified prime for every sufficiently large $n$, i.e.,  a prime in $K$ that ramifies in $K_n(\gamma)$ but not in $K_m(\gamma)$ for any $m<n$. Moreover, the set of sequences $\gamma\in\Phi_S$ with a new ramified prime infinitely often has full measure in $\Phi_S$; see Corollary \ref{cor:newramifiedprime}. 
\end{remark}

\subsection{Organization of the paper}
In Section \ref{sec:abc}, we review some basic material on heights and the $abc$ conjecture and then use this information to give height bounds for rational solutions to certain generalized Fermat-Catalan equations; compare to work of Silverman over function fields in \cite{MR0664038}. Next in Section \ref{sec:uniformbdness}, we use the aforementioned height bounds to prove the uniform boundedness statement in Theorem \ref{thm:uniform-bound}. We then prove in Section \ref{sec:portraits} the complete classification of preperiodic portraits, up to the natural action of roots of unity on preperiodic points, for unicritical polynomial maps of large degree. Finally, we apply our results on preperiodic points to Galois theory and prove Theorem \ref{thm:main+irred} and Theorem \ref{cor:BigGalois} in Sections~\ref{red:irreducibility} and~\ref{sec:Galois}, respectively.

\subsection*{Acknowledgments} We thank Vivian Healey and Robin Zhang for comments on an earlier draft that helped improve the exposition. The first author's research was partially supported by NSF grant DMS-2302394.

\section{The $abc$ conjecture and Fermat-Catalan equations}\label{sec:abc}
\subsection{Absolute values and heights}

Let $K$ be either a number field or the function field $k(X)$ of a curve $X$ defined over some base field $k$. (Throughout this paper, we will refer to the latter simply as a ``function field".) We recall some basic properties of such a field.

Let $M_K$ denote the set of nontrivial places of $K$ which, when $K$ is a function field, are trivial on the constant subfield. We also denote by $M_K^0$ and $M_K^\infty$ the sets of nonarchimedean and archimedean places, respectively, of $K$. Note that if $K$ is a function field, then $M_K^\infty = \emptyset$, hence $M_K = M_K^0$. We typically use $v$ to denote a place of $K$; however, since nonarchimedean places of $K$ correspond to prime ideals in an appropriate subring of $K$, we will often use $\p$ to represent nonarchimedean places and $v_\p$ to denote the corresponding valuation.

For a number field $K$, we denote by $\cO_K$ the ring of integers of $K$; when $K$ is a function field, we choose once and for all a place $\q_\infty \in M_K$, and we define 
    \[
    \cO_K := \big\{\alpha \in K : v_{\p}(\alpha) \ge 0 \text{ for all } \p \in M_K^0 \setminus \{\q_\infty\}\big\}.
    \]
(For example, when $K = k(t)$ is a rational function field and $\mathfrak q_\infty$ is the valuation determined by the order of vanishing at $\infty$, then $\cO_K$ is simply $k[t]$.)

For a place $v \in M_K$, we denote by $K_v$ the completion of $K$ at $v$, and from this we define
    \[
        n_v :=
            \begin{cases}
                [K_v : \Q_v], &\text{ if $K$ is a number field};\\
                \hfill 1, &\text{ if $K/k$ is a function field}.
            \end{cases}
    \]
(If $v$ corresponds to a prime $\p$, we may also write $n_\p$ for $n_v$.)

Now suppose that $\p \in M_K^0$, and denote by $k_\p$ the residue field of $\p$. We define
    \[
        N_\p :=
            \begin{cases}
                \dfrac1{[K:\Q]}\log |k_\p|, &\text{ if $K$ is a number field};\\
                \hfill [k_\p : k], &\text{ if $K/k$ is a function field}.
            \end{cases}
    \]

Finally, for each place $v \in M_K$, we define an absolute value $\abs{\cdot}_v$ on $K$ as follows:

\begin{itemize}
\item If $K$ is a number field and $v \in M_K^\infty$, then the place $v$ corresponds to a complex conjugate pair of embeddings $\sigma : K \hookrightarrow \C$, and we take
	\[
		\abs{x}_v := \abs{\sigma(x)},
	\]
where $\abs{\cdot}$ with no subscript corresponds to the standard absolute value on $\C$.
\item If $K$ is a number field and $v \in M_K^0$ corresponds to a prime $\p$, then
	\[
		\abs{x}_v = \abs{x}_\p := e^{-v_\p(x)N_\p\frac{[K:\Q]}{n_\p}} = |k_\p|^{\frac{-v_\p(x)}{n_\p}}.
	\]
\item If $K$ is a function field and $v \in M_K = M_K^0$ corresponds to a prime $\p$, then 
        \[
        \abs{x}_v = \abs{x}_\p := e^{-v_p(x)N_\p}.
        \]
\end{itemize}

For this normalization, the product formula holds: for all $x \in K^\times$, we have
	\[
		\prod_{v\in M_K} \abs{x}_v^{n_v} = 1; \quad\text{equivalently,}\quad \sum_{v\in M_K} n_v\log\abs{x}_v = 0.
	\]

The (logarithmic) height of a point $P = (x_1 : x_2 : \cdots : x_{n+1}) \in \bP^n(K)$ is then
	\[
		h(P) := \delta_K \sum_{v\in M_K} n_v\log\max\{\abs{x_1}_v, \abs{x_2}_v, \ldots,\abs{x_{n+1}}_v\},
	\]
where we define
    \[
        \delta_K :=
            \begin{cases}
                \dfrac1{[K:\Q]}, &\text{ if $K$ is a number field;}\\
                \hfill 1, &\text{ if $K$ is a function field.}
            \end{cases}
    \]
(We have followed the convention of \cite{MR3138487}, normalizing the height by $[K:\Q]$ in the number field case for convenience.)
Note that for $\alpha \in K$, we have
	\[
		h(\alpha) = h\big((\alpha : 1)\big) =
            \delta_K \sum_{v \in M_K} n_v \log^+\abs{\alpha}_v,
	\]
where we recall that $\log^+(\cdot) = \log\max\{\cdot, 1\}$.

Finally, we define the (logarithmic) {\bf radical} of $P$ to be the quantity
	\[
		\rad(P) := \sum_\p N_\p,
	\]
where the sum is taken over all $\p \in \Spec \cO_K$ for which $v_\p(x_i) \ne v_\p(x_j)$ for some $i,j$ with $x_ix_j \ne 0$. For an element $\alpha \in K$, we will define $\rad(\alpha)$ slightly differently: if $\alpha \ne 0$, we set
	\[
		\rad(\alpha) := \sum_{\substack{\p\in \Spec\cO_K\\v_\p(\alpha) > 0}} N_\p,
	\]
and we set $\rad(0) = 0$. In particular, we 
note that it follows from the product formula that
\begin{equation}\label{bd:productformula}
\rad(\alpha)\leq \sum_{\substack{\p\in \Spec\cO_K\\v_\p(\alpha) > 0}} \hspace{-.4cm}v_\p(\alpha)N_\p\; \le h(\alpha).
\end{equation}

We will use the following version of the $abc$ conjecture over number fields; see, for instance, \cite[p. 84]{vojta:1987} or \cite[p. 100]{elkies:1991}.

\begin{conjecture}[$abc$ conjecture for number fields]\label{conj:abc}
Fix a number field $K$ and $\epsilon > 0$. Then there exists a constant $C(K,\epsilon)$ such that, for all $P = (a : b : c) \in \bP^2(K)$ satisfying $a + b = c$ and $abc \ne 0$, we have
	\[
		\max\{h(a/c), h(b/c)\} < (1+\epsilon)\rad(P) + C(K,\epsilon).
	\]
\end{conjecture}

We note that a typical statement of the $abc$ conjecture has $h(P)$ on the left-hand side, rather than $\max\{h(a/c), h(b/c)\}$, but our version is an immediate consequence of the standard conjecture: Writing $P = (a/c : b/c : 1)$, we have
    \begin{align*}
        h(P) &= \sum_{v \in M_K} n_v\log\max\{\abs{a/c}_v, \abs{b/c}_v, 1\}\\
		  &\ge \max\left\{ \sum_{v \in M_K} n_v\log\max\{\abs{a/c}_v, 1\}, \sum_{v \in M_K} n_v\log\max\{\abs{b/c}_v, 1\}\right\}\\
		  &= \max\{h(a/c), h(b/c)\}.
    \end{align*}
We have chosen to state Conjecture~\ref{conj:abc} as we have in order to more closely mirror the function field version due to Mason \cite[Lemma 10, p. 97]{mason1984diophantine}---see also \cite{silverman1984s}---which we now state. For a field $K$ and a positive integer $n$, we write $K^n$ for the set of $n$th powers in $K$.

\begin{thm}[$abc$ theorem for function fields]\label{thm:mason}
    Fix a function field $K$ of characteristic $p \ge 0$ and genus $g_K$. Let $P = (a : b : c) \in \bP^2(K)$ satisfy $a + b = c$ and $abc \ne 0$; if $p > 0$, further assume that $a/c \notin K^p$. Then
        \[
            \max\{h(a/c), h(b/c)\} \le \rad(P) + 2g_K - 2.
        \]
\end{thm}

Note that, in Theorem~\ref{thm:mason}, the assumption that $a/c \notin K^p$ implies also that $b/c \notin K^p$, since if we had $b/c = \alpha^p$, we would then have $a/c = (1 - \alpha)^p$.

\subsection{The generalized Fermat-Catalan equation}

We begin by providing a height inequality for solutions to generalized Fermat-Catalan equations over $abc$-fields. This result will provide a major tool for proving both the uniform boundedness statements and the irreducibility statements from the introduction. For what follows, when $K$ is a function field, we let $g_K$ denote the genus of $K$.

\begin{prop}\label{prop:fermat-catalan}
Let $K$ be an $abc$-field of characteristic $p \ge 0$, and let $a,b \in K^\times$. Fix integers $m \le n$ such that one of the following holds:
	\begin{itemize}
	\item $m \ge 3$ and $n \ge 4$, or
	\item $m = 2$ and $n \ge 5$.
	\end{itemize} Then there exist an absolute constant $B_1 > 0$ and a constant $B_2 = B_2(K) \ge 0$ depending only on $K$ such that for every solution $(x,y) \in K^2$ to the equation
	\begin{equation}\label{eq:power-equation}
		ax^m + by^n = 1
	\end{equation}
we have either
    \begin{equation}\label{eq:FC-ht-bound}
\max\{mh(x), nh(y)\} \le B_1\max\{h(a),h(b)\} + B_2
    \end{equation}
or
    \[
        p > 0 \quad\text{and}\quad ax^m, by^n \in K^p.
    \]
Moreover, when $K$ is a function field we can take $B_1=40$ and
$B_2(K)=\max\{0, 20g_K-20\}$, and when $K$ is a number field we can take $B_1=41$.
\end{prop}

\begin{remark}
In terms of the restrictions on $m$ and $n$, the result is best possible. If one takes $m = 1$, $m = 2$ and $2\le n \le 4$, or $m = n = 3$, then the affine curve defined by \eqref{eq:power-equation} has genus at most $1$, hence could have infinitely many $K$-rational points (and {\it will} have infinitely many points over some finite extension of $K$).
\end{remark}

\begin{remark}
Silverman proved a version of Proposition~\ref{prop:fermat-catalan} in \cite{MR0664038}, but he omitted the possibility that the height bound \eqref{eq:FC-ht-bound} could fail when either $ax^m$ or $by^n$ (hence both) is a $p$th power in $K$ when $K$ has characteristic $p>0$; this was recently pointed out by Koymans \cite{koymans:2022}.
\end{remark}

\begin{proof}
Set
$P = (ax^m : by^n : 1) \in \bP^2(K)$. If $K$ has characteristic $p > 0$, assume that either $ax^m$ or $by^n$ is not a $p$th power in $K$, in which case neither is. We now show that the inequality from \eqref{eq:FC-ht-bound} is satisfied.

We first derive an upper bound for $\rad(P)$. In order for a prime $\p \in \Spec\cO_K$ to appear in the sum defining $\rad(P)$, we must have either $v_\p(ax^m) > 0$, $v_\p(by^n) > 0$, or $\min\{v_\p(ax^m), v_\p(by^n)\} < 0$, and in the latter case we have $v_\p(ax^m) = v_\p(by^n)$. In other words, the primes appearing in the sum for $\rad(P)$ appear in the sum for at least one of $\rad(a)$, $\rad(b)$, $\rad(x)$, $\rad(y)$, $\rad(1/b)$, and $\rad(1/y)$.
It follows that
	\begin{align*}
		\rad(P)
			&\le \rad(a) + \rad(b) + \rad(x) + \rad(y) + \rad(1/b) + \rad(1/y)\\
			&\le h(a) + h(b) + h(x) + h(y) + h(1/b) + h(1/y)\\
			&= h(a) + 2h(b) + h(x) + 2h(y)\\
			&\le 3\max\{h(a), h(b)\} + h(x) + 2h(y).
	\end{align*}

Define
    \[
        \eta_K :=
            \begin{cases}
                \frac{401}{400}, &\text{ if $K$ is a number field;}\\
                1, &\text{ if $K$ is a function field.}
            \end{cases}
    \]
If $K$ is an $abc$-field, then applying Conjecture~\ref{conj:abc} (with $\epsilon = 1/400$) or Theorem~\ref{thm:mason} as appropriate, we have that there is a constant $C(K)$ depending only on $K$, with $C(K) = 2g_K - 2$ in the function field case, such that
	\begin{align*}
		\max\{h(ax^m), h(by^n)\}
            &\le \eta_K\rad(P) + C(K)\\
            &\le \eta_K\big(3\max\{h(a), h(b)\} + h(x) + 2h(y)\big) + C(K).
	\end{align*}
Using the fact that $h(\alpha) \le h(\alpha\beta) + h(\beta)$ whenever $\beta \ne 0$, we have
	\[
		\max\{h(x^m), h(y^n)\} \le \max\{h(ax^m), h(by^n)\} + \max\{h(a), h(b)\},
	\]
and since $\eta_K \ge 1$, this implies that
	\begin{align*}
		\max\{h(x^m), h(y^n)\}
			&\le \eta_K\big(4\max\{h(a), h(b)\} + h(x) + 2h(y)\big) + C(K)\\
			&= \eta_K\left(4\max\{h(a), h(b)\} + \frac1m h(x^m) + \frac2n h(y^n)\right) + C(K)\\
			&\le \eta_K\left(4\max\{h(a), h(b)\} + \left(\frac1m + \frac2n\right) \max\{h(x^m), h(y^n)\}\right) + C(K)\\
                &\le \eta_K\left(4\max\{h(a), h(b)\} + \frac9{10} \max\{h(x^m), h(y^n)\}\right) + C(K),
	\end{align*}
where we have used the fact that whenever $m\ge3$ and $n \ge 4$ or $m = 2$ and $n \ge 5$, we have $\frac1m + \frac2n \le \frac9{10}$. Now, rewriting the above inequality yields
    \[
        \left(1 - \frac{9\eta_K}{10}\right)\max\{h(x^m), h(y^n)\} \le 4\eta_K \max\{h(a), h(b)\} + C(K),
    \]
hence
    \[
		\max\{h(x^n), h(y^m)\}
		  \le \frac{40\eta_K}{10 - 9\eta_K} \max\{h(a), h(b)\} + \frac{10\max\{C(K), 0\}}{10 - 9\eta_K}.
    \]
Finally, if $K$ is a function field, we have $\eta_K = 1$ and $C(K) = 2g_K - 2$, so we can take
    \[
        B_1 = 40 \text{ and } B_2 = \max\{20g_K - 20, 0\};
    \]
if $K$ is a number field, we have $\eta_K = \frac{401}{400}$, so, assuming $C(K) \ge 0$, we can take
    \[
        B_1 = \left\lceil\frac{40\eta_K}{10-9\eta_K}\right\rceil = 41 \text{ and } B_2(K) = \frac{4000}{391}C(K).\qedhere
    \]
\end{proof}
\section{Uniform boundedness of preperiodic points}\label{sec:uniformbdness}
In this section, we prove the uniform boundedness statements from the introduction. Before we begin, we briefly sketch the basic argument. The first step is to show that the preperiodic points of a fixed unicritical polynomial map have roughly the same height; more specifically,
\begin{equation}\label{eq:preperiodics are close}
\qquad\qquad  
h(\alpha)=\frac{1}{d}h(c)+o(d)\qquad \text{for all $\alpha\in\PrePer(x^d+c,K)$,}
\end{equation}
where the error term depends only on $d$ (and not $K$ or $c$).
From here, we note that if $\phi(\alpha)-\phi(\beta)$ is nonzero for some $\alpha,\beta\in\PrePer(\phi,K)$ with $\phi(x)=x^d+c$, then we obtain a solution $(x,y)=(\alpha,\beta)$ to the generalized Fermat-Catalan equation $ax^d+(-a)y^d=1$, where $a := (\phi(\alpha) - \phi(\beta))^{-1}$. In particular, it follows from Proposition~\ref{prop:fermat-catalan} and \eqref{eq:preperiodics are close} that
\begin{equation}\label{eq:dh(alpha)}
    dh(\alpha)\leq B_1 h(\alpha)+B_2
\end{equation}
for some absolute constant $B_1$ and some $B_2$ depending on $K$, unless $K$ has characteristic $p > 0$ and $a^{-1}\alpha^d$ (hence also $a^{-1}\beta^d$) is a $p$th power in $K$. (The issue with $p$th powers is the reason for the congruence conditions imposed on $d$ in the statement of Theorem~\ref{thm:uniform-bound}.) Excluding instances of the $p$th power obstruction, it follows that for all $d\gg_K0$ and $h(c)\gg_K0$, we must have $\phi(\alpha)=\phi(\beta)$ for all $\alpha,\beta\in\PrePer(x^d+c,K)$, in which case the preperiodic points are contained in the set of preimages of a unique fixed point. On the other hand, for each of the remaining $c\in K$ of small height, we may use \eqref{eq:preperiodics are close} again to deduce that $h(\alpha)=0$ for all $d\gg_{K,c}0$ and all $\alpha\in\PrePer(x^d+c,K)$. Moreover, since height-zero points are either roots of unity or zero---or constants in the function field case---we may classify the action of unicritical polynomial maps on these points. In particular, combining these steps yields the desired form of uniform boundedness for preperiodic points in large degree.       
\subsection{Height inequalities for preperiodic points}

For this section, we fix an $abc$-field $K$, an integer $d \ge 2$, and an element $c \in K$. We begin by recording a standard arithmetic property of preperiodic points.

\begin{lem}\label{lem:nonarch-abs-val}
Let $v \in M_K^0$ and suppose that $\alpha \in K$ is preperiodic for $x^d + c$.  
	\begin{enumerate}
	\item[\textup{(1)}] If $\abs{c}_v \le 1$, then $\abs{\alpha}_v \le 1$.
	\item[\textup{(2)}] If $\abs{c}_v > 1$, then $\abs{\alpha}_v = \abs{c}_v^{1/d}$.
	\end{enumerate}
\end{lem}

To prove a similar statement for the archimedean places in the number field setting, we use the following estimates for the roots of an auxiliary polynomial over the real numbers:  

\begin{lem}\label{lem:gamma-d}
Let $F_d(x) := x^d - 2x - 1$ and let $\rho_d$ be the unique positive real root of $F_d$. Then
	\begin{enumerate}
	\item[\textup{(1)}] $\rho_d > 1$ for all $d \ge 2$;
	\item[\textup{(2)}] the sequence $(\rho_d)_d$ is decreasing, with $\displaystyle\lim_{d\to\infty}\rho_d = 1$; and
	\item[\textup{(3)}] the sequence $(\rho_d^d)_d$ is decreasing, with $\displaystyle\lim_{d\to\infty}\rho_d^d = 3$.
	\end{enumerate}
In particular, we have that 
	\[
		\rho_d \le \rho_2 = 1 + \sqrt 2 \quad\text{and}\quad \rho_d^d \le \rho_2^2 = 3 + 2\sqrt 2
	\]
for all $d \ge 2$.
\end{lem}

\begin{proof}
That $F_d$ has a {\it unique} positive real root follows from Descartes' rule of signs. Moreover, since $F_d(1) < 0$ and $F_d$ is increasing on $[1, \infty)$, we have $\rho_d > 1$ for all $d$. We now show that the sequences $(\rho_d)$ and $(\rho_d^d)$ are decreasing. Since $\rho_d > 1$ for each $d\ge2$, we have that
	\[
		F_{d+1}(\rho_d) = \rho_d^{d+1} - 2\rho_d - 1 > \rho_d^d - 2\rho_d - 1 = 0,
	\]
and since $F_{d+1}$ is increasing on $[1,\infty)$, this implies that $\rho_d > \rho_{d+1}$. Likewise, the sequence $(\rho_d^d)$ is decreasing since $\rho_d^d = 2\rho_d + 1$.

Next, we show that for all $\epsilon > 0$, we have $\rho_d < 1+\epsilon$ for all $d \gg_\epsilon 0$, which then implies that $\rho_d$ tends to $1$. Since $F_d(1) < 0$, it suffices to show that $F_d(1+\epsilon) > 0$ for all sufficiently large $d$ (depending on $\epsilon$). Indeed, we have
	\[
		F_d(1+\epsilon) = (1 + \epsilon)^d - (3 + 2\epsilon),
	\]
which tends to infinity as $d \to \infty$. Thus, $\rho_d \to 1$, and it follows immediately that
	\[
		\lim_{d\to\infty} \rho_d^d = \lim_{d\to\infty} 2\rho_d + 1 = 3. \qedhere
	\]
\end{proof}

\begin{lem}\label{lem:arch-abs-val}
Let $\rho_d$ be as defined in Lemma~\ref{lem:gamma-d}, let $K$ be a number field, and let $v \in M_K^\infty$. If $\alpha \in K$ is preperiodic for $x^d + c$, then
	\[
		\abs{\log^+\abs{\alpha}_v - \frac1d\log^+\abs{c}_v} \le \log\rho_d.
	\]
\end{lem}

\begin{proof}
Let $\alpha$ be preperiodic for $\phi(x) = x^d + c$. Then it is straightforward to verify that  
	\begin{equation}\label{eq:2c^1/d}
            \log^+\abs{\alpha}_v \le \frac1d \log^+\abs{c}_v + \log2.
	\end{equation}
We note that the inequality we are now proving is actually stronger than \eqref{eq:2c^1/d} whenever $d \ge 3$, since for $d \ge 3$ we have that $\rho_d \le \rho_3 = \frac{1 + \sqrt 5}{2} < 2$.

We first prove that
	\[
		\log^+\abs{\alpha}_v \le \frac1d\log^+\abs{c}_v + \log\rho_d.
	\]
Since the right-hand side of the inequality is nonnegative, it suffices to show that $\log\abs{\alpha}_v \le \frac1d \log^+\abs{c}_v + \log \rho_d$. Suppose for contradiction that \vspace{.1cm} 
	\begin{equation}\label{eq:basin-ineq-contra}
	\abs{\alpha}_v > \rho_d\max\{\abs{c}_v, 1\}^{1/d}.
	\end{equation}
Then we have that 
	\begin{equation*}
		\;\abs{\phi(\alpha)}_v = \abs{\alpha^d + c}_v
  \ge \abs{\alpha}_v\left(\abs{\alpha}_v^{d-1} - \frac{\abs{c}_v}{\abs{\alpha}_v}\right)
  > \abs{\alpha}_v\left(\rho_d^{d-1} - \frac1{\rho_d}\right)\max\{\abs{c}_v, 1\}^{(d-1)/d}
	> \abs{\alpha}_v. \vspace{.1cm}  
	\end{equation*}
Here the last inequality follows from the fact that $\rho_d^d = 2\rho_d + 1$ and so $\rho_d^{d-1} - \frac{1}{\rho_d} = 2$. But then $\abs{\phi^n(\alpha)}_v$ is increasing with $n$, contradicting the fact that $\alpha$ is preperiodic.

Thus, it remains to show that
	\[
		\log^+\abs{\alpha}_v \ge \frac1d\log^+\abs{c}_v - \log\rho_d.
	\]
However, note that the inequality above is trivial whenever $\abs{c}_v \le \rho_d^d$, so we may assume that $\abs{c}_v > \rho_d^d > 1$.
Now, suppose for contradiction that the claimed inequality fails, so that
	\[
		\abs{\alpha}_v < \frac{\abs{c}_v^{1/d}}{\rho_d}.
	\]
Then
	\begin{equation*}
	\abs{\phi(\alpha)}_v
		= \abs{\alpha^d + c}_v
		\ge \abs{c}_v - \abs{\alpha}_v^d
		> \abs{c}_v\left(1 - \frac{1}{\rho_d^d}\right).
	\end{equation*}
On the other hand, since $\alpha$ is preperiodic for $\phi$, so is $\phi(\alpha)$. Hence, \eqref{eq:2c^1/d} applied to the point $\phi(\alpha)$ in place of $\alpha$, together with the bound above, imply that  
	\[
		\abs{c}_v\left(1 - \frac{1}{\rho_d^d}\right)<\abs{\phi(\alpha)}_v \le 2\abs{c}_v^{1/d}.
	\]
From here, we see that 
	\[
		\abs{c}_v^{(d-1)/d}\left(1 - \frac{1}{\rho_d^d}\right) < 2.
	\]
Moreover, since $\abs{c}_v > \rho_d^d$ by assumption, it follows that
	\[
		\rho_d^{d-1}\left(1 - \frac1{\rho_d^d}\right) < 2.
	\]
However, we then have that $\rho_d^d - 2\rho_d - 1 < 0$, which contradicts that $\rho_d$ is a root of $F_d$.
\end{proof}
We now have the tools in place to show that the preperiodic points for a fixed unicritical polynomial have roughly the same height:   
\begin{cor}\label{cor:ht-bound}
Fix $d \ge 2$ and $c \in K$, and let $\alpha \in K$ be preperiodic for $x^d + c$.
    \begin{enumerate}
        \item[\textup{(1)}] If $K$ is a function field, then $h(\alpha) = \frac1dh(c)$. \vspace{.1cm} 
        \item[\textup{(2)}] If $K$ is a number field, then
            \[
                \abs{h(\alpha) - \frac1dh(c)} \le \log\rho_d,
            \]
        where $\rho_d$ is as defined in Lemma~\ref{lem:gamma-d}. 
    \end{enumerate}
\end{cor}

\begin{proof}
First, we observe that
	\begin{equation}\label{eq:ht-ineq-log+}
		\abs{h(\alpha) - \frac{1}{d}h(c)}
			\le \delta_K \sum_{v\in M_K} n_v\abs{\log^+\abs{\alpha}_v - \frac1d\log^+\abs{c}_v},
	\end{equation}
where we recall that $\delta_K$ is equal to $1/[K:\Q]$ if $K$ is a number field and $1$ if $K$ is a function field.
Since $\log^+\abs{\alpha}_v = \frac1d\log^+\abs{c}_v$ for all $v \in M_K^0$ by Lemma~\ref{lem:nonarch-abs-val}, the only possible nonzero terms of the sum in \eqref{eq:ht-ineq-log+} are at the archimedean places $v \in M_K^\infty$; in particular, statement (1) now follows.
For statement (2), we assume that $K$ is a number field and apply Lemma~\ref{lem:arch-abs-val} to \eqref{eq:ht-ineq-log+} to get
	\[
		\abs{h(\alpha) - \frac1dh(c)}
			\le \frac{1}{[K:\Q]}\sum_{v\in M_K^\infty} n_v \log\rho_d
			= \log\rho_d. \qedhere
	\]
\end{proof}

As a final preliminary height estimate, we record the following inequality from \cite{looper2021dynamical}; note that only the number field case is stated explicitly in \cite[Lemma 8.2]{looper2021dynamical}, but the same proof also yields the function field case immediately.

\begin{lem}\label{lem:diff-bound}
Let $c \in K$, and let $\alpha,\beta\in K$ be any two preperiodic points for $x^d + c$.
Then
    \[
    h(\alpha-\beta) \le
        \begin{cases}
            \dfrac1d h(c), &\text{ if $K$ is a function field;}\\[10pt]
            \dfrac1d h(c) + \log 4, &\text{ if $K$ is a number field.}
        \end{cases}
    \]
\end{lem}

\subsection{Proof of Theorem~\ref{thm:uniform-bound}}
Having established several useful height estimates, we now begin the proof of uniform boundedness by establishing that, for a fixed $c \in K$, the preperiodic points for $x^d + c$ have height zero for all sufficiently large degrees $d$.

\begin{lem}\label{lem:rou}
Let $K$ be an $abc$-field, and let $c \in K$. Then for all $d \gg_{K,c} 0$, all $K$-rational preperiodic points for $x^d + c$ have height $0$.
\end{lem}

\begin{proof}
Let $\alpha \in K$ be a preperiodic point for $x^d + c$. By Corollary~\ref{cor:ht-bound}, we have
	\[
		h(\alpha) \le \frac1dh(c) + \log \rho_d,
	\]
and by Lemma~\ref{lem:gamma-d} we have
	\[
		\lim_{d\to\infty} \frac1dh(c) + \log \rho_d = 0.
	\]
Now let $h_K^{\min}$ be the minimal nonzero height of an element of $K$,
\begin{equation}\label{eq:heightmin}
h^{\min}_K:=\min\{h(\alpha)\;:\; \alpha\in K\;\text{and}\; h(\alpha)>0\},
\end{equation} 
which is well-defined (and positive) by Northcott's theorem in the number field case and by the fact that heights are nonnegative integers in the function field case. Then for all $d$ sufficiently large we have
	\[
		\frac1dh(c) + \log \rho_d < h_K^{\min}.
	\]
In other words, for all $d \gg_{K,c} 0$, every $K$-rational preperiodic point for $x^d + c$ has height $0$.
\end{proof}

Next, for the number field case, we need the following elementary statement related to points on the unit circle in the complex plane. 

\begin{lem}\label{lem:x+y=c}
Let $c \in \C^\times$. Up to swapping $x$ and $y$, the system
	\[
            \left\{
 		\begin{split}
		x + y &= c\\
		\abs x = \abs y &= 1
		\end{split}
            \right.
	\]
has at most one solution. Moreover, if $\abs c = 1$, then
	\[
		x = \zeta_6 c \quad\text{and}\quad y = \zeta_6^{-1} c
	\]
for some choice $\zeta_6$ of a primitive $6$th root of unity.
\end{lem}

\begin{proof}
Since $c \ne 0$, the equations $\abs x = \abs{c - x} = 1$ describe two distinct circles of radius $1$, which intersect in at most two points. Thus $x$ must be one of at most two points. But $\abs y = \abs{c - y} = 1$ as well, so $y$ must also be one of those two intersection points. Thus, if $x \ne y$, we are done. On the other hand, if $x = y$, then $\abs c = 2\abs x = 2$, so the two circles are tangent, hence $(x,x) = (y,y)$ is the only solution. In either case, there is a unique solution up to swapping $x$ and $y$.

The last statement is immediate: there can only be one solution up to reordering, and it is easy to check that
	\[
		\zeta_6 c + \zeta_6^{-1}c = (\zeta_6 + \zeta_6^{-1})c = c. \qedhere
	\]
\end{proof}
In particular, we can use the result above to restrict the possible actions of a unicritical polynomial map on the roots of unity in characteristic $0$.

\begin{cor}\label{cor:two-images}
Let $K$ be a field of characteristic $0$, let $c \in K$, let $d \ge 2$ and let $\phi(x) = x^d + c$. Then
	\[
		\big|\phi(\mu_K) \cap \mu_K\big| \le 2.
	\]
If $\omega_1$ and $\omega_2$ are distinct elements of $\phi(\mu_K) \cap \mu_K$, then $c = \omega_1 + \omega_2$. Moreover, if $c \in \mu_K$, then
	\[
		\omega_1 = \zeta_6c \quad\text{and}\quad \omega_2 = \zeta_6^{-1}c
	\]
for some choice $\zeta_6$ of a primitive $6$th root of unity.
\end{cor}

\begin{proof}
Let $\omega_1,\omega_2,\omega_3 \in \phi(\mu_K) \cap \mu_K$, and for each $i = 1,2,3$ write $\omega_i = \phi(\xi_i)$ with $\xi_i \in \mu_K$. Then
	\[
		c = \omega_1 - \xi_1^d = \omega_2 - \xi_2^d = \omega_3 - \xi_3^d.
	\]
By choosing an embedding of $\Q(c,\xi_1,\xi_2,\xi_3)$ into $\C$, we may apply Lemma~\ref{lem:x+y=c} to say that the points $\omega_1$, $\omega_2$, and $\omega_3$ cannot all be distinct, so there are at most two elements of $\phi(\mu_K) \cap \mu_K$.

Now suppose $\omega_1, \omega_2 \in \phi(\mu_K) \cap \mu_K$ are distinct. As above, we have
	\[
		c = \omega_1 - \xi_1^d = \omega_2 - \xi_2^d,
	\]
and since $\omega_1 \ne \omega_2$, Lemma~\ref{lem:x+y=c} says that $\omega_2 = -\xi_1^d$. Therefore, we have $c = \omega_1 + \omega_2$. The final statement about the case $c \in \mu_K$ follows immediately from the last statement of Lemma~\ref{lem:x+y=c}.
\end{proof}
We now prove the uniform boundedness theorem from the introduction. We start with the following, which is part (2) of Theorem~\ref{thm:uniform-bound} for function fields. Recall from the introduction that if $p$ is prime and $d \ge 2$ is an integer, we factor $d = d_p p^n$ with $p \nmid d_p$; when $p = 0$, we take $d_0 = d$ and interpret $p^n$ as $1$.

\begin{prop}\label{prop:function-fields-ubc}
    Let $K$ be a function field of characteristic $p \ge 0$. Fix $c \in K \setminus \overline{k}$ and $d \ge 2$. If $d \not\equiv 1 \pmod p$ and $d_p > D_1 = D_1(K) := \max\{20g_K + 20, 40\}$, then either $\PrePer(x^d + c, K) = \emptyset$ or
        \begin{equation}\label{eq:c=y-y^d}
            c = y - y^d \quad\text{and}\quad \PrePer(x^d + c, K) = \{\zeta y : \zeta \in \mu_{K,d}\}
        \end{equation}
    for some unique $y \in K$.
\end{prop}

\begin{proof}
We show that for all $c \in K\setminus \overline{k}$ and all $d$ with $d_p > D_1$ and $d \not\equiv 1 \pmod{p}$, any preperiodic points $\alpha,\beta \in K$ for $\phi(x) = x^d + c$ must satisfy $\phi(\alpha) = \phi(\beta)$. It follows that $y := \phi(\alpha)$ is a fixed point, and \eqref{eq:c=y-y^d} is satisfied.

We henceforth assume that $d \not\equiv 1 \pmod p$. Write $d = d_pp^n$ as above. Let $c \in K\setminus \overline{k}$ be such that $\phi(x) = x^d + c$ has preperiodic points $\alpha,\beta$ with $\phi(\alpha) \ne \phi(\beta)$. Our goal is to show that $d_p \le D_1$.

We first observe that $\left(\phi(\alpha) - \phi(\beta)\right)^{-1}$ is a $(p^n)$th power. Indeed, we have
	\[
		\phi(\alpha) - \phi(\beta) = \alpha^d - \beta^d = \left(\alpha^{d_p} - \beta^{d_p}\right)^{p^n};
	\]
thus, setting $a := \left(\alpha^{d_p} - \beta^{d_p}\right)^{-1}$ we have $a^{p^n} = \left(\phi(\alpha) - \phi(\beta)\right)^{-1}$.
Moreover, $(x,y) = (\alpha,\beta)$ is a solution to the Fermat-Catalan equation
	\[
		ax^{d_p} - ay^{d_p} = 1.
	\]

Suppose either that $p = 0$ or that $p > 0$ and $a\alpha^{d_p} \notin K^p$.
Then we may apply Proposition~\ref{prop:fermat-catalan} to say that
    \[
        d_p\max\{h(\alpha), h(\beta)\} \le 40h(a) + \max\{20g_K - 20, 0\}.
    \]
Since
	\[
		h(a) = \frac{1}{p^n}h\left(\phi(\alpha) - \phi(\beta)\right) \le \frac1{dp^n}h(c) = \frac1{p^n}h(\alpha) = \frac1{p^n}h(\beta)
	\]
by Corollary~\ref{cor:ht-bound} and Lemma~\ref{lem:diff-bound}, we have
    \[
        d_ph(\alpha) \le \frac{40}{p^n}h(\alpha) + \max\{20g_K - 20, 0\} \le 40h(\alpha) + \max\{20g_K - 20, 0\},
    \]
hence
    \[
        (d_p-40)h(\alpha) \le \max\{20g_K - 20, 0\}.
    \]
Thus, we have either that $d_p \le 40$ or, since $h(\alpha)\geq1$ (as $c$ being nonconstant implies $\alpha$ is nonconstant, hence $h(\alpha)$ is a positive integer), that
    \[
        d_p-40 \le (d_p-40)h(\alpha) \le \max\{20g_K - 20, 0\}.
    \]
In either case, we have $d_p \le D_1$.

In summary, we have now shown that if $c \in K \setminus \overline{k}$ and $d_p > D_1$, then for all $\alpha,\beta \in \PrePer(\phi,K)$, either $\phi(\alpha) = \phi(\beta)$ or $p > 0$ and $a\alpha^{d_p} \in K^p$. In particular, having reached the desired conclusion in characteristic zero, we may assume that $p>0$ and that for all $\alpha,\beta\in\PrePer(\phi,K)$, either $\phi(\alpha) = \phi(\beta)$ or $a\alpha^{d_p} \in K^p$. In the latter case, 
    \[
        \left(\frac{\beta}{\alpha}\right)^{d_p} = 1 - \frac{\alpha^{d_p} - \beta^{d_p}}{\alpha^{d_p}} = 1 - \left(a\alpha^{d_p}\right)^{-1} \in K^p,
    \]
and since $p\nmid d_p$, this implies that $\beta/\alpha \in K^p$. Likewise, if $\phi(\alpha) = \phi(\beta)$, then $\alpha = \zeta\beta$ for some $\zeta\in\mu_{K,d} \subset \overline{\mathbb{F}}_p$. Hence, $\beta/\alpha\in K^p$, since $\overline{\mathbb{F}}_p\subseteq K^p$ when $p > 0$. In particular, we may assume that $\beta/\alpha\in K^p$ for all $\alpha,\beta\in\PrePer(\phi,K)$. From here we proceed in cases:\\

\noindent {\bf Case 1.} Suppose there exists $\alpha \in \PrePer(\phi,K)$ with $\alpha \notin K^p$. Let $D$ be a derivation on $K$ with kernel $K^p$; see \cite[\S6]{mason1984diophantine}. In particular, $D(\beta/\alpha)=0=D(\phi(\beta)/\alpha)$ by our assumption above. Hence, $\alpha D(\beta)=\beta D(\alpha)$ and $\alpha D(\phi(\beta))=\phi(\beta) D(\alpha)$ by the quotient rule. But then,  
\begin{equation*}
\begin{split}
0&=\alpha D(\phi(\beta))-\phi(\beta)D(\alpha)\\[5pt] 
&=\alpha(d\beta^{d-1}D(\beta)+D(c))-(\beta^d+c)D(\alpha)\\[5pt]
&=d\beta^{d-1}(\alpha D(\beta))+\alpha D(c)-D(\alpha)\beta^d-D(\alpha)c\\[5pt]
&=dD(\alpha)\beta^{d}+\alpha D(c)-D(\alpha)\beta^d-D(\alpha)c\\[5pt]
&=(d-1)D(\alpha)\beta^d+\alpha D(c)-D(\alpha)c.
\end{split} 
\end{equation*}   
In other words, if we view $\alpha$ as fixed, then every element $\beta\in\PrePer(\phi,K)$ is a root of the polynomial $F_{\phi,\alpha}\in K[z]$ given by
    \[
        F_{\phi,\alpha}(z) := (d-1)D(\alpha) z^d + \alpha D(c) - cD(\alpha).
    \]
Moreover, since we have assumed that both $d \not\equiv 1 \pmod p$ and $\alpha \notin K^p$, we have that $(d-1)D(\alpha) \ne 0$. Thus, since $\alpha \in \PrePer(\phi,K)$ is also a root of $F_{\phi,\alpha}$, we deduce that 
    \[
        \PrePer(\phi,K) \subseteq \{\text{roots of $F_{\phi,\alpha}$}\} = \{\zeta\alpha : \zeta \in \mu_d\}.
    \]
Therefore, $\phi(\beta) = \phi(\zeta\alpha) = \phi(\alpha)$ for all $\beta \in \PrePer(\phi,K)$ as claimed.\\

\noindent {\bf Case 2.} Now suppose that $\alpha \in K^p$ for all $\alpha \in \PrePer(\phi, K)$ and let $m$ be the largest integer with the property that $\PrePer(\phi,K) \subset K^{p^m}$; note that $m$ is well-defined since we have assumed that $c$ is non-constant, hence the same is true for all elements of $\PrePer(\phi,K)$. From here, we fix $\alpha \in \PrePer(\phi,K)$ such that $\alpha \in K^{p^m} \setminus K^{p^{m+1}}$ and let $A \in K \setminus K^p$ be such that $\alpha = A^{p^m}$. Note also that 
    \[
        c = \phi(\alpha) - \alpha^d \in K^{p^m},
    \]
so we can write $c = C^{p^m}$ for some $C \in K$. Finally, let $\Phi(x) := x^d + C$ and let $\sigma: K\rightarrow K^{p^m}$ be the ring isomorphism given by $\sigma(x)=x^{p^m}$. 

Then, thinking of $\phi$ and $\Phi$ as functions on $K^{p^m}$ and $K$, respectively, we have that 
    \[
        \Phi = \sigma^{-1} \circ \phi \circ \sigma.
    \]
Therefore, $\Phi^k = \sigma^{-1} \circ \phi^k \circ \sigma$ for all $k \ge 0$. In particular, since we have assumed that $\PrePer(\phi,K) \subset K^{p^m}$, it follows that $B\in \PrePer(\Phi, K)$ if and only if $B^{p^m} \in \PrePer(\phi,K)$.

Now, since $A \notin K^p$, it follows from Case 1 above that for every $B \in \PrePer(\Phi, K)$ we must have that $\Phi(A) = \Phi(B)$. However, this fact then implies that for every $\beta \in \PrePer(\phi,K)$, we must have that $\phi(\alpha) = \phi(\beta)$ as claimed.
\end{proof}

We now complete the proof of Theorem~\ref{thm:uniform-bound}.

\begin{proof}[Proof of Theorem~\ref{thm:uniform-bound}]
That statement (2) holds for function fields is precisely Proposition~\ref{prop:function-fields-ubc}; we now show that (2) holds for number fields. Arguing as in Proposition~\ref{prop:function-fields-ubc}, it suffices to show that if $K$ is a number field, then for all $d \ge 42$ and all $c \in K$ with $h(c)\gg_K0$, any preperiodic points $\alpha,\beta \in K$ for $\phi(x) = x^d + c$ must satisfy $\phi(\alpha) = \phi(\beta)$.

Thus, let $K$ be a number field. Let $c \in K$ be such that $\phi$ has preperiodic points $\alpha,\beta$ with $\phi(\alpha) \ne \phi(\beta)$ and let $a=\big(\phi(\alpha)-\phi(\beta)\big)^{-1}$. As in the proof of Proposition~\ref{prop:function-fields-ubc}, $(x,y) = (\alpha,\beta)$ is a solution to the generalized Fermat-Catalan equation $ax^d-ay^d=1$. By Proposition~\ref{prop:fermat-catalan}, we have
	\[
		d\max\{h(\alpha), h(\beta)\} \le 41h(a) + B_2(K)
	\]
for some constant $B_2(K)$ depending only on $K$. Note that
	\[
		h(a) = h(\phi(\alpha) - \phi(\beta)) \le \frac1d h(c) + \log 4
	\]
by Lemma~\ref{lem:diff-bound}, since $\phi(\alpha)$ and $\phi(\beta)$ are preperiodic for $\phi$; applying this inequality as well as the inequality from Corollary~\ref{cor:ht-bound} yields
	\[
		d\left(\frac1dh(c) - \log\rho_d\right) \le 41\left(\frac1dh(c) + \log4\right) + B_2(K).
	\]
Since we have assumed $d \ge 42$, this implies that
	\begin{align*}
		h(c) &\le \frac{d}{d-41}\left(\log\rho_d^d + 41\log4 + B_2(K)\right)\\
                &\le 42\left(\log\rho_{42}^{42} + 41\log4 + B_2(K)\right),
	\end{align*}
where the second inequality follows from Lemma~\ref{lem:gamma-d}. Therefore, if we choose $C_1 = C_1(K)$ so that
    \[
        C_1 > 42\left(\log\rho_{42}^{42} + 41\log4 + B_2(K)\right) > 2434.088 + 42B_2(K)
    \]
and assume that $d>D_1\ge41$ and $h(c) > C_1$, then we must have that $\phi(\alpha) = \phi(\beta)$ as claimed. This completes the proof of part (2) of Theorem~\ref{thm:uniform-bound}.

We now prove part (3). The statement is immediate when $K$ is a function field: If $h(c) \le C_1(K) = 0$, then $c$ is a constant point in $K$, hence all of the preperiodic points for $\phi(x) = x^d + c$ are constant as well. Thus, we assume $K$ is a number field.

By Lemma~\ref{lem:rou}, for any $c \in K$, there is a constant $D_{1,c}' = D_{1,c}'(K)$ such that for all $d > D_{1,c}'$, all $K$-rational preperiodic points for $x^d + c$ have height $0$. Since there are only finitely many elements in $K$ of height at most $C_1(K)$, we can define
    \[
        D_1'(K) := \max\{D_{1,c}' : h(c) \le C_1(K)\},
    \]
and then we can set $D_1(K) := \max\{D_1'(K), 41\}$ to guarantee that both statements (2) and (3) of the theorem hold.

Finally, we prove part (1). By part (2), it suffices to show that if $c \in K^\times$, $d > D_1(K)$, and $h(c) \le C_1(K)$ (and $c$ is non-constant in the case that $K$ is a function field), then $\phi(x) = x^d + c$ cannot have $K$-rational points of period greater than $3$. The statement is vacuously true when $K$ is a function field, since $h(c) \le C_1(K) = 0$ implies $c$ is constant, so we again assume that $K$ is a number field.

Now suppose that $c\in K^\times$, $h(c) \le C_1(K)$, and $d > D_1(K)$. By part (3), $\PrePer(\phi, K)$ consists only of height-zero points; it therefore suffices to show that for all $n \ge 4$, there cannot be an $n$-cycle consisting only of roots of unity and $0$.

We first show there cannot be such a cycle with at least four roots of unity. Indeed, if there were, then there would be distinct roots of unity $\omega_1, \omega_2, \omega_3, \omega_4$ such that $\phi(\omega_i) = \omega_{i+1}$ for $i=1,2,3$, contradicting Corollary~\ref{cor:two-images}.

Thus, the only way that $\phi$ could have a cycle of length at least $4$ consisting only of roots of unity and $0$ is for the cycle to have the form
	\[
		0 \mapsto \omega_3 \mapsto \omega_1 \mapsto \omega_2 \mapsto 0
	\]
with each $\omega_i$ a root of unity.
Suppose there was such a cycle. Since $\omega_1 \ne\omega_2$, it follows from Corollary~\ref{cor:two-images} that
	\[
		\omega_3 = \phi(0) = c = \omega_1 + \omega_2
	\]
and, since $\omega_3 \in \mu_K$, it follows from Lemma~\ref{lem:x+y=c} that
	\[
		\omega_1 = \zeta_6\omega_3 \quad\text{and}\quad \omega_2 = \zeta_6^{-1}\omega_3
	\]
with $\zeta_6$ a primitive $6$th root of unity. Since $\phi(\omega_1) = \omega_2$, we have
		\[
			\zeta_6^d\omega_3^d + \omega_3 = \zeta_6^{-1}\omega_3,\quad\text{hence}\quad \omega_3^{d-1} = \zeta_6^{-d}(\zeta_6^{-1} - 1) = \zeta_6^{-d}\zeta_6^{-2} = \zeta_6^{-d-2}.
		\]
On the other hand, since $\phi(\omega_2) = 0$, we have
		\[
			\zeta_6^{-d}\omega_3^d + \omega_3 = 0,\quad\text{hence}\quad \omega_3^{d-1} = -\zeta_6^d = \zeta_6^{d+3}.
		\]
This implies that $\zeta_6^{-d-2} = \zeta_6^{d+3}$, which is impossible for any integer $d$. Therefore, $\phi$ cannot admit a $4$-cycle consisting only of roots of unity and $0$, completing the proof of part (1).
\end{proof}
Combining Theorem \ref{thm:uniform-bound} for large degrees with prior work in \cite{MR4065068,looper2021dynamical} for small degrees, we obtain a uniform boundedness statement for unicritical polynomials over global fields that is independent of the degree of the map.
\begin{proof}[Proof of Corollary \ref{cor:strong-ubc}]
By \cite[Theorem 1.7]{MR4065068} in the function field case and \cite[Theorem 1.2]{looper2021dynamical} in the number field case, for each $d \ge D_2(K)$ not divisible by $p$, there is a uniform bound $B_{K,d}$ such that
    \[
        \big|\PrePer(x^d + c, K)\big| \le B_{K,d}
    \]
for all $c\in K$ (non-constant, if $K$ is a function field). It then follows from Theorem~\ref{thm:uniform-bound} in the function field case that if $d \not\equiv 0,1\pmod{p}$, then 
\[\big|\PrePer(x^d + c, K)\big| \le \max\big\{|\mu_{K,d}|,\; \max\{B_{K,d} : D_2(K) \le d \le D_1(K) \text{ and } p\nmid d\}\big\}\]
for all $d\geq D_2$ and all non-constant $c\in K$. (Only statement (2) of Theorem \ref{thm:uniform-bound} applies, since $c$ non-constant implies $h(c) > 0 = C_1(K)$.) Likewise, when $K$ is a number field, Theorem \ref{thm:uniform-bound} implies that 
\[\big|\PrePer(x^d + c, K)\big| \le \max\big\{|\mu_{K}|+1,\; \max\{B_{K,d} : D_2(K) \le d \le D_1(K)\}\big\},\]
since both statements (2) and (3) of Theorem \ref{thm:uniform-bound} must be considered. 
In any case, the result follows by setting
    \[
        B(K) := \max\{B_{K,d} : D_2(K) \le d \le D_1(K)\}
    \]
if $K$ is a function field, and
    \[
        B(K) := \max\big\{|\mu_{K}|+1,\; \max\{B_{K,d} : D_2(K) \le d \le D_1(K)\}\big\}
    \]
if $K$ is a number field.
\end{proof}

\section{Classification of preperiodic portraits}\label{sec:portraits}

Let $K$ be an $abc$-field. For a map $\phi \in K[x]$, the ({\bf preperiodic}) {\bf portrait} $\mathscr{P}(\phi, K)$ is the directed graph whose vertices are the elements of $\PrePer(\phi, K)$, with an edge $\alpha \to \beta$ if and only if $\phi(\alpha) = \beta$. Following the convention of \cite{Poonen}, we omit from the portrait $\scrP(\phi,K)$ the point at $\infty$, which is a fixed point for every polynomial map.

Corollary~\ref{cor:strong-ubc} says that for a fixed $abc$-field $K$ for which $\mu_K$ is finite, the number of $K$-rational preperiodic points for $\phi(x) = x^d + c$ is bounded above by a constant depending only on $K$, as long as $c$ is nonzero and, when $K$ is a function field, non-constant; when $K$ has characteristic $p > 0$, we further assume that $d \not\equiv 0,1\pmod p$. This implies there are finitely many portraits (up to isomorphism) that can be realized as $\scrP(x^d + c, K)$ for some (non-constant) $c \in K^\times$ and $d \ge D_2(K)$, with $D_2(K)$ as in Corollary~\ref{cor:strong-ubc}, and again assuming $d \not\equiv 0,1\pmod p$ in the case $p > 0$. In this section, we classify all such portraits for $c \in K^\times$ and for all sufficiently large $d$; specifically, we do so for all $d > D_1(K)$, with $D_1(K)$ as in Theorem~\ref{thm:uniform-bound}.

Unfortunately, ``classify all such portraits" does not translate to ``provide a finite list containing all such portraits" as in, for example, Poonen's classification result for rational preperiodic points for quadratic polynomials; see \cite[Figure 1]{Poonen}. Indeed, there is no finite list that contains all portraits $\scrP(x^d + c, K)$ when we allow both $K$ and $d > D_1(K)$ to vary. The fundamental problem is that if $\alpha$ is a nonzero $K$-rational preperiodic point for $\phi(x) = x^d + c$, then so is $\zeta\alpha$ for all $\zeta\in \mu_{K,d}$, as $\phi(\zeta\alpha) = \phi(\alpha)$. Thus, as the size of $\mu_{K,d}$ increases, so must the number of $K$-rational preperiodic points for $x^d+ c$ (assuming there are any). For example, let $K$ be a number field, let $k := \big|\mu_{K,d}\big|$, and let $\zeta$ be a generator for $\mu_{K,d}$. If we take $c := y - y^d$ with $y\in K^\times$, then $y$ is a $K$-rational fixed point for $x^d + c$, and $\zeta^iy$ is a preimage of $y$ for each $i \in\{1,\ldots,k-1\}$, so the portrait $\scrP(x^d + c, K)$ contains the graph illustrated in Figure~\ref{fig:FP}.

\begin{figure}
\centering
    \begin{tikzpicture}[scale=2]
	\tikzset{vertex/.style = {}}
	\tikzset{every loop/.style={min distance=10mm,in=45,out=-45,->}}
	\tikzset{edge/.style={decoration={markings,mark=at position 1 with %
    {\arrow[scale=1.5,>=stealth]{>}}},postaction={decorate}}}
	%
	%
	\node[vertex] (11a) at  (-1, 1) {$\zeta y$};
	\node[vertex] (11b) at  (-.3, 1) {$\zeta^2 y$};
	\node[vertex] (dots) at (.3, 1) {$\cdots$};
	\node[vertex] (11c) at (1, 1) {$\zeta^{k-1} y$};
	\node[vertex] (1) at (0, 0) {$y$};
	%
	%
	\draw[-{Latex[length=1.5mm,width=2mm]}] (11a) to (1);
	\draw[-{Latex[length=1.5mm,width=2mm]}] (11b) to (1);
	\draw[-{Latex[length=1.5mm,width=2mm]}] (11c) to (1);
	\draw[-{Latex[length=1.5mm,width=2mm]}] (1) to[out=220, in=320, looseness=7] (1);
\end{tikzpicture}
	\caption{A fixed point and its preimages under $x^d + c$}
	\label{fig:FP}
\end{figure}

Based on the previous paragraph, one might say that the full preperiodic portrait has some redundancy: the data of a single (nonzero) preperiodic point $\alpha$ is enough to determine the entire fiber $\{\zeta\alpha : \zeta\in\mu_{K,d}\}$ over $\phi(\alpha)$. (Note that this redundancy is a special property of polynomials of the form $x^d + c$. For a generic polynomial $\phi(x)$ of degree $d > 2$, one does not expect the existence of a preperiodic point $\alpha$ to imply the existence of an additional preimage of $\phi(\alpha)$.) Thus, it makes sense to replace the preperiodic portrait $\scrP(\phi, K)$ with what we will call its {\it skeleton} $\scrS(\phi, K)$, which removes this redundancy. Our main result of this section (Theorem~\ref{thm:prep-classification}) says that, up to isomorphism, there are only finitely many such skeleta $\scrS(x^d + c, K)$ as one ranges over all $abc$-fields $K$, all $c \in K^\times$ (non-constant in the function field case), and all degrees $d \ge D_1(K)$.

\begin{defin}
The {\bf skeleton} of a portrait $\scrP(\phi,K)$ is the directed graph $\mathscr{S}(\phi,K)$ constructed from $\scrP(\phi,K)$ as follows: \vspace{.1cm}
	\begin{itemize}
	\item For all vertices $\alpha$ in $\scrP(\phi,K)$ with at least one $K$-rational preimage under $\phi$, include $\alpha$ as a vertex of $\scrS(\phi,K)$.\vspace{.1cm}
	\item If $\alpha$ has been included in $\scrS(\phi,K)$ but none of its $K$-rational preimages have, then include a single vertex $v$ to represent all of the $K$-rational preimages of $\alpha$ under $\phi$. \vspace{.1cm}
	\item Edge relations among the vertices in $\scrS(\phi,K)$ are inherited from $\scrP(\phi,K)$.
	\end{itemize}
\end{defin}

\begin{exa}
The portrait in Figure~\ref{fig:FP} has as its skeleton the directed graph labeled {\bf(1)a} in Table~\ref{tab:SPs}.    
\end{exa}

We now state our main classification result, which refers to the graphs in Table~\ref{tab:SPs}. We note that the labels assigned to the graphs in Table~\ref{tab:SPs} consist of the cycle structure (i.e., the nonincreasing sequence of cycle lengths appearing in the graph), together with a letter if there are multiple graphs with the same cycle structure.

\begin{thm}\label{thm:prep-classification}
Let $K$ be an $abc$-field, and let $D_1(K)$ be as in Theorem~\ref{thm:uniform-bound}. If $d > D_1(K)$ and $d \not\equiv 0,1\pmod p$, and if $c \in K^\times$ (non-constant, if $K$ is a function field), then the skeleton $\scrS(x^d + c, K)$ is isomorphic to either the empty graph or one of the twelve graphs in Table~\ref{tab:SPs}.
\end{thm}

\begin{table}
\aboverulesep = 0pt
\belowrulesep = 0pt
\begin{tabular}{|C{.48\textwidth}|C{.48\textwidth}|}
\toprule
{\bf (1)a}\hfill\mbox{}
	&
{\bf (1)b}\hfill\mbox{}\\
\begin{tikzpicture}[scale=.9]
	\tikzset{vertex/.style = {}}
	\tikzset{every loop/.style={min distance=10mm,in=45,out=-45,->}}
	\tikzset{edge/.style={decoration={markings,mark=at position 1 with %
    {\arrow[scale=1.5,>=stealth]{>}}},postaction={decorate}}}
	%
	%
	\node[vertex] (1) at (0, 0) {$\bullet$};
	%
	%
	\draw[-{Latex[length=1.5mm,width=2mm]}] (1) to[out=310, in=50, looseness=7] (1);
\end{tikzpicture}
	&
\begin{tikzpicture}[scale=.9]
	\tikzset{vertex/.style = {}}
	\tikzset{every loop/.style={min distance=10mm,in=45,out=-45,->}}
	\tikzset{edge/.style={decoration={markings,mark=at position 1 with %
    {\arrow[scale=1.5,>=stealth]{>}}},postaction={decorate}}}
	%
	%
	\node[vertex] (1-2) at (0, 0) {$\omega_3$};
	\node[vertex] (1-1) at (2, 0) {$\omega_2$};
	\node[vertex] (1) at (4, 0) {$\omega_1$};
	%
	%
	\draw[-{Latex[length=1.5mm,width=2mm]}] (1) to[out=310, in=50, looseness=7] (1);
	\draw[-{Latex[length=1.5mm,width=2mm]}] (1-1) to (1);
	\draw[-{Latex[length=1.5mm,width=2mm]}] (1-2) to (1-1);
\end{tikzpicture}
\\
\midrule
	{\bf(1)c}\hfill\mbox{}
	&
	{\bf(1)d}\hfill\mbox{}
\\
\begin{tikzpicture}[scale=.9]
	\tikzset{vertex/.style = {}}
	\tikzset{every loop/.style={min distance=10mm,in=45,out=-45,->}}
	\tikzset{edge/.style={decoration={markings,mark=at position 1 with %
    {\arrow[scale=1.5,>=stealth]{>}}},postaction={decorate}}}
	%
	%
	\node[vertex] (1-2) at (0, 0) {$0$};
	\node[vertex] (1-1) at (2, 0) {$\omega_3$};
	\node[vertex] (1) at (4, 0) {$\omega_1$};
	%
	%
	\draw[-{Latex[length=1.5mm,width=2mm]}] (1) to[out=310, in=50, looseness=7] (1);
	\draw[-{Latex[length=1.5mm,width=2mm]}] (1-1) to (1);
	\draw[-{Latex[length=1.5mm,width=2mm]}] (1-2) to (1-1);
\end{tikzpicture}
	&
\begin{tikzpicture}[scale=1]
	\tikzset{vertex/.style = {}}
	\tikzset{every loop/.style={min distance=10mm,in=45,out=-45,->}}
	\tikzset{edge/.style={decoration={markings,mark=at position 1 with %
    {\arrow[scale=1.5,>=stealth]{>}}},postaction={decorate}}}
	%
	%
	\node[vertex] (1) at (2, 0) {$\omega_1$};
	\node[vertex] (1-1a) at (.166, .5) {$\omega_3$};
	\node[vertex] (1-1b) at (.166, -.5) {$\omega_2$};
	\node[vertex] (1-2a) at (-1.833, .5) {$0$};
	\node[vertex] (1-2b) at (-1.833, -.5) {$\omega_4$};
	%
	%
	\draw[-{Latex[length=1.5mm,width=2mm]}] (1) to[out=310, in=50, looseness=7] (1);
	\draw[-{Latex[length=1.5mm,width=2mm]}] (1-1a) to (1);
	\draw[-{Latex[length=1.5mm,width=2mm]}] (1-1b) to (1);
	\draw[-{Latex[length=1.5mm,width=2mm]}] (1-2a) to (1-1a);
	\draw[-{Latex[length=1.5mm,width=2mm]}] (1-2b) to (1-1b);
\end{tikzpicture}
	\\
\midrule
	{\bf(1)e}\hfill\mbox{}
	&
	{\bf(1,1)}\hfill\mbox{}\\
\begin{tikzpicture}[scale=1]
	\tikzset{vertex/.style = {}}
	\tikzset{every loop/.style={min distance=10mm,in=45,out=-45,->}}
	\tikzset{edge/.style={decoration={markings,mark=at position 1 with %
    {\arrow[scale=1.5,>=stealth]{>}}},postaction={decorate}}}
	%
	%
	\node[vertex] (1) at (2, 0) {$\omega_1$};
	\node[vertex] (1-1a) at (.166, .5) {$\omega_3$};
	\node[vertex] (1-1b) at (.166, -.5) {$\omega_2$};
	\node[vertex] (1-2a) at (-1.833, .5) {$0$};
	\node[vertex] (1-2b) at (-1.833, -.5) {$\omega_4$};
        \node[vertex] (1-3a) at (-3.833, .5) {$\omega_5$};
	%
	%
	\draw[-{Latex[length=1.5mm,width=2mm]}] (1) to[out=310, in=50, looseness=7] (1);
	\draw[-{Latex[length=1.5mm,width=2mm]}] (1-1a) to (1);
	\draw[-{Latex[length=1.5mm,width=2mm]}] (1-1b) to (1);
	\draw[-{Latex[length=1.5mm,width=2mm]}] (1-2a) to (1-1a);
	\draw[-{Latex[length=1.5mm,width=2mm]}] (1-2b) to (1-1b);
        \draw[-{Latex[length=1.5mm,width=2mm]}] (1-3a) to (1-2a);
\end{tikzpicture}
	&
\begin{tikzpicture}[scale=.9]
	\tikzset{vertex/.style = {}}
	\tikzset{every loop/.style={min distance=10mm,in=45,out=-45,->}}
	\tikzset{edge/.style={decoration={markings,mark=at position 1 with %
    {\arrow[scale=1.5,>=stealth]{>}}},postaction={decorate}}}
	%
	%
	\node[vertex] (1a) at (0, 0) {$\omega_1$};
	\node[vertex] (1b) at (2, 0) {$\omega_2$};
	%
	%
	\draw[-{Latex[length=1.5mm,width=2mm]}] (1a) to[out=310, in=50, looseness=7] (1a);
	\draw[-{Latex[length=1.5mm,width=2mm]}] (1b) to[out=310, in=50, looseness=7] (1b);
\end{tikzpicture}
	\\
\midrule
	{\bf(2)a}\hfill\mbox{}
	&
	{\bf(2)b}\hfill\mbox{}\\
\begin{tikzpicture}[scale=.9]
	\tikzset{vertex/.style = {}}
	\tikzset{every loop/.style={min distance=10mm,in=45,out=-45,->}}
	\tikzset{edge/.style={decoration={markings,mark=at position 1 with %
    {\arrow[scale=1.5,>=stealth]{>}}},postaction={decorate}}}
	%
	%
	\node[vertex] (2c) at (2, 0) {$\omega_1$};
	\node[vertex] (2d) at (4, 0) {$\omega_2$};
	%
	%
	\draw[-{Latex[length=1.5mm,width=2mm]}] (2c) to[bend right=30] (2d);
	\draw[-{Latex[length=1.5mm,width=2mm]}] (2d) to[bend right=30] (2c);
\end{tikzpicture}
	&
\begin{tikzpicture}[scale=.9]
	\tikzset{vertex/.style = {}}
	\tikzset{every loop/.style={min distance=10mm,in=45,out=-45,->}}
	\tikzset{edge/.style={decoration={markings,mark=at position 1 with %
    {\arrow[scale=1.5,>=stealth]{>}}},postaction={decorate}}}
	%
	%
	\node[vertex] (2c) at (2, 0) {$0$};
	\node[vertex] (2d) at (4, 0) {$\omega_3$};
	%
	%
	\draw[-{Latex[length=1.5mm,width=2mm]}] (2c) to[bend right=30] (2d);
	\draw[-{Latex[length=1.5mm,width=2mm]}] (2d) to[bend right=30] (2c);
\end{tikzpicture}
	\\
\midrule
	{\bf(2)c}\hfill\mbox{}
	&
	{\bf(2,1,1)}\hfill\mbox{}\\
\begin{tikzpicture}[scale=.9]
	\tikzset{vertex/.style = {}}
	\tikzset{every loop/.style={min distance=10mm,in=45,out=-45,->}}
	\tikzset{edge/.style={decoration={markings,mark=at position 1 with %
    {\arrow[scale=1.5,>=stealth]{>}}},postaction={decorate}}}
	%
	%
	\node[vertex] (2a) at (2, 0) {$0$};
	\node[vertex] (2b) at (4, 0) {$\omega_3$};
	\node[vertex] (2-1a) at (.166, .5) {$\omega_1$};
	\node[vertex] (2-1b) at (.166, -.5) {$\omega_2$};
	\node[vertex] (2-2a) at (-1.833, .5) {$\omega_4$};
	\node[vertex] (2-2b) at (-1.833, -.5) {$\omega_5$};
	%
	%
	\draw[-{Latex[length=1.5mm,width=2mm]}] (2a) to[bend right=30] (2b);
	\draw[-{Latex[length=1.5mm,width=2mm]}] (2b) to[bend right=30] (2a);
	\draw[-{Latex[length=1.5mm,width=2mm]}] (2-1a) to (2a);
	\draw[-{Latex[length=1.5mm,width=2mm]}] (2-1b) to (2a);
	\draw[-{Latex[length=1.5mm,width=2mm]}] (2-2a) to (2-1a);
	\draw[-{Latex[length=1.5mm,width=2mm]}] (2-2b) to (2-1b);
\end{tikzpicture}
	&
\begin{tikzpicture}[scale=.9]
	\tikzset{vertex/.style = {}}
	\tikzset{every loop/.style={min distance=10mm,in=45,out=-45,->}}
	\tikzset{edge/.style={decoration={markings,mark=at position 1 with %
    {\arrow[scale=1.5,>=stealth]{>}}},postaction={decorate}}}
	%
	%
	\node[vertex] (2a) at  (-2, 0) {$0$};
	\node[vertex] (2b) at  (0, 0) {$\omega_3$};
	\node[vertex] (1a) at (2, 0) {$\omega_1$};
	\node[vertex] (1b) at (4, 0) {$\omega_2$};
	%
	%
	\draw[-{Latex[length=1.5mm,width=2mm]}] (2a) to[bend right=30] (2b);
	\draw[-{Latex[length=1.5mm,width=2mm]}] (2b) to[bend right=30] (2a);
	\draw[-{Latex[length=1.5mm,width=2mm]}] (1a) to[out=310, in=50, looseness=7] (1a);
	\draw[-{Latex[length=1.5mm,width=2mm]}] (1b) to[out=310, in=50, looseness=7] (1b);
\end{tikzpicture}
	\\
\midrule
	{\bf(2,2)}\hfill\mbox{}
	&
	{\bf(3)}\hfill\mbox{}\\
\begin{tikzpicture}[scale=.9]
	\tikzset{vertex/.style = {}}
	\tikzset{every loop/.style={min distance=10mm,in=45,out=-45,->}}
	\tikzset{edge/.style={decoration={markings,mark=at position 1 with %
    {\arrow[scale=1.5,>=stealth]{>}}},postaction={decorate}}}
	%
	%
	\node[vertex] (2a) at  (-2, 0) {$0$};
	\node[vertex] (2b) at  (0, 0) {$\omega_3$};
	\node[vertex] (2c) at (2, 0) {$\omega_1$};
	\node[vertex] (2d) at (4, 0) {$\omega_2$};
	%
	%
	\draw[-{Latex[length=1.5mm,width=2mm]}] (2a) to[bend right=30] (2b);
	\draw[-{Latex[length=1.5mm,width=2mm]}] (2b) to[bend right=30] (2a);
	\draw[-{Latex[length=1.5mm,width=2mm]}] (2c) to[bend right=30] (2d);
	\draw[-{Latex[length=1.5mm,width=2mm]}] (2d) to[bend right=30] (2c);
\end{tikzpicture}
	&
\begin{tikzpicture}[scale=.9]
	\tikzset{vertex/.style = {}}
	\tikzset{every loop/.style={min distance=10mm,in=45,out=-45,->}}
	\tikzset{edge/.style={decoration={markings,mark=at position 1 with %
    {\arrow[scale=1.5,>=stealth]{>}}},postaction={decorate}}}
	%
	%
	\node[vertex] (3a) at  (-1, 0) {$0$};
	\node[vertex] (3b) at  (.5, -0.866) {$\omega_3$};
	\node[vertex] (3c) at (.5, 0.866) {$\omega_1$};
	%
	%
	\draw[-{Latex[length=1.5mm,width=2mm]}] (3a) to[bend right=30] (3b);
	\draw[-{Latex[length=1.5mm,width=2mm]}] (3b) to[bend right=30] (3c);
	\draw[-{Latex[length=1.5mm,width=2mm]}] (3c) to[bend right=30] (3a);
\end{tikzpicture}
	\\
\bottomrule
\end{tabular}
\caption{Nonempty graphs that can be realized as $\scrS(z^d + c, K)$ with $c \in K^\times$ and $d > D_1(K)$. Here $\omega_i$ always represents a root of unity.}
\label{tab:SPs}
\end{table}

\begin{remark}
Theorem~\ref{thm:prep-classification} explicitly excludes $c = 0$, since $0$ and every root of unity are preperiodic for $f_{d,0}(z) = z^d$, so even the skeleta $\scrS(f_{d,0}, K)$ will take infinitely many forms as $K$ varies. Similarly, if $c$ is a constant element of a function field $K$, then the set of preperiodic points could be arbitrarily large---even infinite if, for example, the constant subfield is algebraically closed.
\end{remark}

\begin{remark}
    The $K = \Q$ case of Theorem~\ref{thm:prep-classification} recovers \cite[Theorem 3]{MR4432520}, which says that if the $abc$ conjecture holds, then for all $d \gg 0$ the following statements hold:
        \begin{itemize}
            \item For all $c \in \Q^\times$, the map $x^d + c$ has at most four rational preperiodic points (including the fixed point $\infty$). 
            \item If, moreover, $c \ne -1$, then $x^d + c$ has at most one (necessarily fixed) rational periodic point.
        \end{itemize}
More precisely, assuming the $abc$ conjecture (over $\Q$), the only graphs that can be realized as $\scrS(x^d + c, \Q)$ with $d \gg 0$ and $c \in \Q^\times$ are
    \begin{itemize}
        \item {\bf(1)a}, if $c = y - y^d$ with $y \notin \{0,1,(-1)^d\}$, and
        \item {\bf(2)b}, if $c = -1$ and $d$ is even.
    \end{itemize}
\end{remark}

If $K$ is a function field, then it follows from Theorem~\ref{thm:uniform-bound} that for all $d > D_1(K)$ satisfying $d\not\equiv 0,1\pmod p$, and for all non-constant $c \in K$, the preperiodic points must form a portrait isomorphic to the one illustrated in Figure~\ref{fig:FP}, hence the skeleton must be isomorphic to {\bf(1)a}. For this reason, \textit{we henceforth assume $K$ is a number field}.

By Theorem~\ref{thm:uniform-bound}, if $c \in K^\times$ and $d > D_1(K)$, then $\PrePer(x^d + c, K)$ either consists entirely of height-zero points or has the form $\{\zeta y : \zeta \in \mu_{K,d}\}$ for a fixed point $y \in K$ for $x^d + c$. In the latter case, the full preperiodic portrait is the portrait appearing in Figure~\ref{fig:FP}, hence its skeleton is the one labeled {\bf(1)a} in Table~\ref{tab:SPs}. Thus, it remains to consider only those $c \in K^\times$ and $d > D_1(K)$ for which $\PrePer(x^d + c, K)$ consists only of height-zero points. We get a further restriction from Corollary~\ref{cor:two-images}: there can be at most two roots of unity $\omega \in \PrePer(x^d + c, K)$ with a root of unity preimage. We therefore make the following definition.

\begin{defin}
    A {\bf height-zero skeleton} is a skeleton all of whose vertices correspond to $0$ or roots of unity and for which there are at most two roots of unity with a preimage which is also a root of unity.
\end{defin}

Note that, up to graph isomorphism, there are only finitely many height-zero skeleta. Thus, to prove Theorem~\ref{thm:prep-classification}, one could list all such graphs and then, for each graph $G$ in the list, determine whether there exist a number field $K$, a degree $d > D_1(K)$, and an element $c \in K^\times$ such that $\scrS(x^d + c, K)$ is isomorphic to $G$. This is essentially the approach we take, with some shortcuts included so we can more quickly dispense with some of the graphs.

\begin{notation}
Throughout the rest of this section, the symbols $\omega_i$ are reserved for roots of unity. More specifically, we will use $\omega_1$ and $\omega_2$ for roots of unity which, in a given skeleton, have preimages which are also roots of unity. This explains why, for example, some graphs in Table~\ref{tab:SPs} have $\omega_1$ and $\omega_3$ but not $\omega_2$; in such a graph, $\omega_1$ would be the only root of unity with a root of unity preimage. Likewise, we will often use $\zeta_6$ to represent a choice of a primitive $6$th root of unity, with $\zeta_3 := \zeta_6^2$. For example, if we say ``$\omega_1 = \zeta_6\omega_3$", we mean that the equation is valid {\it for some choice} of a primitive root of unity $\zeta_6$, and all subsequent uses of $\zeta_6^{\pm 1}$ and $\zeta_3^{\pm 1}$ are dependent on that choice.

Note also that we have used $\bullet$ for the sole vertex of graph {\bf(1)a} in Table~\ref{tab:SPs}, rather than $\omega$ or $0$ as in the other graphs in that table. The reason for this is that for any $d \ge 2$, there exists $c \in K$ such that $\scrS(x^d + c, K)$ is isomorphic to {\bf(1)a} but for which the fixed point is neither a root of unity nor $0$.
\end{notation}

Before we begin discussing specific graphs, we record the following elementary lemma that will be useful for ruling out certain dynamical behavior.

\begin{lem}\label{lem:w1w2w3}
Let $d \ge 2$, let $K$ be a number field, and let $\omega_1,\omega_2,\omega_3 \in \mu_K$. Let $\phi(z) := z^d + \omega_3$ and suppose that $\omega_1,\omega_2 \in \mu_K \cap \phi(\mu_K)$ are preperiodic for $\phi$. If $\phi(\omega_3) = \phi(\omega_i)$ for at least one $i \in \{1,2\}$, then $\phi(\omega_1) = \phi(\omega_2) = \phi(\omega_3)$.
\end{lem}

\begin{proof}
By Corollary~\ref{cor:two-images}, the hypotheses imply that $\omega_1 = \zeta_6\omega_3$ and $\omega_2 = \zeta_6^{-1}\omega_3$. Thus, the fact that $\phi(\omega_3) = \phi(\omega_i)$ for some $i\in\{1,2\}$ implies that $6 \mid d$. It then follows that $\phi(\omega_1) = \phi(\omega_3) = \phi(\omega_2)$.
\end{proof}

\subsection{Graphs with $3$-cycles}
We begin with classifying the skeleta that contain a $3$-cycle, since these are the easiest to deal with. 
\begin{prop}\label{prop:3cycle}
Let $d \ge 2$, let $K$ be a number field, let $c \in K^\times$, and let $\phi(x) = x^d + c$. Suppose that $\scrS(\phi, K)$ is a height-zero skeleton. If $\phi$ has a $K$-rational point of period $3$, then the skeleton $\scrS(\phi, K)$ is isomorphic to graph {\bf(3)} from Table~\ref{tab:SPs}. Furthermore, this is the case if and only if the following conditions hold: 
	\begin{itemize}
	\item $d \equiv 1 \pmod 6$,
 	\item $\omega_3^{d-1} = \zeta_3$,
	\item $c = \omega_3$,
	\item $\omega_1 = \zeta_6\omega_3$.
	\end{itemize}
\end{prop}

\begin{proof}
We first show that the four listed conditions are necessary and sufficient for $\scrS(\phi, K)$ to {\it contain} the graph {\bf(3)}, and after that we show that the skeleton cannot properly contain {\bf(3)}. It is straightforward to show that the four conditions imply that $\{0, \omega_3, \omega_1\}$ forms a $3$-cycle for $\phi$, so we focus on showing that these four conditions are necessary.

Thus, suppose that $\phi$ has a $K$-rational point of period $3$. We first note that the $3$-cycle must contain $0$, since otherwise the cycle would consist solely of roots of unity, contradicting Corollary~\ref{cor:two-images}. In other words, the $3$-cycle must be as shown in {\bf(3)}.

We begin with the simple observation that $c = \phi(0) = \omega_3$. Next, since $\phi(\omega_3) = \omega_1$, we have  
	\[
		\omega_3^d + \omega_3 = \omega_1, \quad\text{hence}\quad \omega_3 = \omega_1 - \omega_3^d,
	\]
which implies that $\omega_1 = \zeta_6\omega_3$ and $-\omega_3^d = \zeta_6^{-1}\omega_3$ by Lemma~\ref{lem:x+y=c}; in other words,
	\[
		\omega_1 = \zeta_6\omega_3 \quad\text{and}\quad \omega_3^{d-1} = -\zeta_6^{-1} = \zeta_3.
	\]
Finally, since $\phi(\omega_1) = 0$, we have that 
	\[
		\zeta_6^d\omega_3^d + \omega_3 = 0, \quad\text{hence}\quad \zeta_3 = \omega_3^{d-1} = -\zeta_6^{-d} = \zeta_6^{3-d},
	\]
which implies that $d \equiv 1 \pmod 6$.

It remains to show that if $\scrS(\phi, K)$ is a height-zero skeleton with a $3$-cycle, then $\scrS(\phi, K)$ is precisely {\bf(3)}.
Suppose for contradiction that $\scrS(\phi,K)$ \textit{properly} contains the $3$-cycle from {\bf(3)}. Then there must be an additional root of unity $\omega_2$ appearing in the skeleton, and $\omega_2$ must itself be the image of a root of unity under $\phi$. Since we already have $\omega_1 = \zeta_6\omega_3$, we would then have $\omega_2 = \zeta_6^{-1}\omega_3$ by Corollary~\ref{cor:two-images}, thus
	\begin{align*}
		\phi(\omega_2)
			&= \zeta_6^{-d}\omega_3^d + \omega_3\\
			&= \omega_3(\zeta_6^{-1} \omega_3^{d-1} + 1) \qquad \text{(since $d\equiv1 \pmod6$)}\\
			&= \omega_3(\zeta_6^{-1} \zeta_3 + 1)\\
			&= \omega_3(\zeta_6 + 1).
	\end{align*}
But $\zeta_6 + 1$, and therefore $\omega_3(\zeta_6+1)$, has positive height, contradicting our initial assumption that all $K$-rational preperiodic points have height $0$.
\end{proof}

\subsection{Graphs with $2$-cycles}
A height-zero $2$-cycle will take the form $\{0,\omega_3\}$ or $\{\omega_1,\omega_2\}$, so there are two separate cases to consider.

\begin{lem}\label{lem:2-cycles}
Let $d \ge 2$, let $K$ be a number field, let $c \in K^\times$, and let $\phi(x) = x^d + c$.\vspace{.05cm}  
	\begin{enumerate}
	\item[\textup{(1)}]The map $\phi$ has a $2$-cycle of the form $\{\omega_1, \omega_2\}$ if and only if $c = \omega_1 + \omega_2$, $\omega_1^{d-1} = \omega_2^{d-1} = -1$, and $\omega_1 \ne \omega_2$. \vspace{.2cm} 
	\item[\textup{(2)}]The map $\phi$ has a $2$-cycle of the form $\{0, \omega_3\}$ if and only if $c = \omega_3$ and $\omega_3^{d-1} = -1$.
	\end{enumerate}
\end{lem}

\begin{proof}
For part (1), suppose that $\{\omega_1,\omega_2\}$ is a $2$-cycle for $\phi$. That $c = \omega_1 + \omega_2$ follows from Corollary~\ref{cor:two-images}, and we must certainly have $\omega_1 \ne \omega_2$.
Now, since
	\[
		c = \omega_1 - \omega_2^d = \omega_2 - \omega_1^d,
	\]
it follows from Lemma~\ref{lem:x+y=c} that $\omega_1 = -\omega_1^d$ and $\omega_2 = -\omega_2^d$, hence $\omega_1^{d-1} = \omega_2^{d-1} = -1$. The converse, that these conditions imply that $\omega_1$ and $\omega_2$ form a $2$-cycle for $\phi$, is straightforward to verify.

For part (2), we simply note that $c = \phi(0) = \omega_3$ and
	\[
		0 = \phi(\omega_3) = \omega_3^d + \omega_3,
	\]
hence $\omega_3^{d-1} = -1$. Once again, we omit the straightforward proof of the converse.
\end{proof}

\begin{prop}\label{prop:most2cycles}
Let $K$ be a number field, let $c \in K^\times$, and suppose that $\scrS(x^d + c, K)$ is a height-zero skeleton that contains the graph {\bf(2)b} from Table~\ref{tab:SPs}, so that $c = \omega_3$ and $\omega_3^{d-1} = -1$ for some $\omega_3 \in \mu_K$. Then $\scrS(x^d + c, K)$ is isomorphic to one of the following: \vspace{.1cm}
	\begin{enumerate}
		\item[\textup{(1)}] {\bf(2,2)}, if and only if \vspace{.05cm}
			\begin{itemize}
			\item $d \equiv 1 \pmod 6$,\vspace{.05cm}
			\item $\omega_1 = \zeta_6\omega_3$, and
			\item $\omega_2 = \zeta_6^{-1}\omega_3$;
			\end{itemize} \vspace{.15cm}
		\item[\textup{(2)}] {\bf(2,1,1)}, if and only if
                \vspace{.05cm}
			\begin{itemize}
			\item $d \equiv 5 \pmod 6$,\vspace{.05cm}
			\item $\omega_1 = \zeta_6\omega_3$, and\vspace{.05cm}
			\item $\omega_2 = \zeta_6^{-1}\omega_3$;
			\end{itemize}\vspace{.15cm}
		\item[\textup{(3)}] {\bf(2)c}, if and only if
            \vspace{.05cm}
			\begin{itemize}
			\item $d \equiv 0 \pmod 6$,\vspace{.05cm}
			\item $\omega_1 = \zeta_6\omega_3$,\vspace{.05cm}
                \item $\omega_2 = \zeta_6^{-1}\omega_3$,\vspace{.05cm}
			\item $\omega_4^d = \zeta_3\omega_3$, and\vspace{.05cm}
			\item $\omega_5^d = \zeta_3^{-1}\omega_3$;
			\end{itemize} \vspace{.15cm}
		\item[\textup{(4)}] {\bf(2)b}, otherwise.
	\end{enumerate}
\end{prop}

\begin{proof}
In each case, it is easy to verify that if the conditions in each bulleted list are satisfied, then $\scrS(x^d + c, K)$ {\it contains} the corresponding graph. But {\bf(2,2)}, {\bf(2,1,1)}, and {\bf(2)c} are maximal height-zero skeleta, in the sense that there is no height-zero skeleton properly containing any of them, so it follows that each of the first three bulleted lists of conditions is sufficient for $\scrS(x^d + c, K)$ to be {\it isomorphic to} the corresponding graph. We now prove necessity.
Write $\phi(x) = x^d + c$.

First, suppose $\scrS(\phi, K)$ is isomorphic to {\bf(2,2)}. By Lemma~\ref{lem:2-cycles}, we must have the following:
	\begin{itemize}
	\item $\omega_3 = c = \omega_1 + \omega_2$,
	\item $\omega_1^{d-1} = \omega_2^{d-1} = \omega_3^{d-1} = -1$, and
	\item $\omega_1 \ne \omega_2$.
	\end{itemize}
By Corollary~\ref{cor:two-images}, we have $\omega_1 = \zeta_6\omega_3$ and $\omega_2 = \zeta_6^{-1}\omega_3$,
so it remains only to show that $d \equiv 1 \pmod 6$. Indeed, using the fact that $\omega_1^{d-1} = \omega_3^{d-1}$, we have $\zeta_6^{d-1}\omega_3^{d-1} = \omega_3^{d-1}$, hence $\zeta_6^{d-1} = 1$,
and therefore $d \equiv 1 \pmod 6$.

Next, suppose $\scrS(\phi, K)$ is isomorphic to {\bf(2,1,1)}. Once again, we have $\omega_1 = \zeta_6\omega_3$ and $\omega_2 = \zeta_6^{-1}\omega_3$ by Corollary~\ref{cor:two-images}, and since $\phi(\omega_1) = \omega_1$, we have
	\[
		\zeta_6^d\omega_3^d + \omega_3 = \zeta_6\omega_3, \quad\text{hence}\quad -1 = \omega_3^{d-1} = \zeta_6^{-d}(\zeta_6 - 1) = \zeta_6^{-d}\zeta_3 = \zeta_6^{2-d}.
	\]
Therefore, $d \equiv 5 \pmod 6$.

We now suppose that $\scrS(\phi, K)$ is isomorphic to {\bf(2)c}. Yet again, it follows from Corollary~\ref{cor:two-images} that $\omega_1 = \zeta_6\omega_3$ and $\omega_2 = \zeta_6^{-1}\omega_3$. Since $\omega_1$ and $\omega_3$ are preimages of a common point (i.e., $0$), we have $\zeta_6^d\omega_3^d = \omega_1^d = \omega_3^d$, and hence $6 \mid d$. The fact that $\phi(\omega_4) = \omega_1$ implies that
	\[
		\omega_4^d + \omega_3 = \omega_1 = \zeta_6\omega_3, \quad\text{hence}\quad \omega_4^d = (\zeta_6 - 1)\omega_3 = \zeta_3\omega_3.
	\]
The same argument shows that $\omega_5^d = (\zeta_6^{-1} - 1)\omega_3 = \zeta_3^{-1}\omega_3$.

To complete the proof, it remains to show that the four graphs listed in the proposition---that is, {\bf(2,2)}, {\bf(2,1,1)}, {\bf(2)c}, and {\bf(2)b}---are the only graphs containing {\bf(2)b} that can be realized as $\scrS(x^d + c, K)$.

\begin{figure}
\begin{subfigure}[t]{0.45\textwidth}
\centering
\begin{tikzpicture}[scale=.9]
	\tikzset{vertex/.style = {}}
	\tikzset{every loop/.style={min distance=10mm,in=45,out=-45,->}}
	\tikzset{edge/.style={decoration={markings,mark=at position 1 with %
    {\arrow[scale=1.5,>=stealth]{>}}},postaction={decorate}}}
	%
	%
	\node[vertex] (2a) at  (0, 0) {$0$};
	\node[vertex] (2b) at  (2, 0) {$\omega_3$};
	\node[vertex] (1a) at (4, 0) {$\omega_1$};
	%
	%
	\draw[-{Latex[length=1.5mm,width=2mm]}] (2a) to[bend right=30] (2b);
	\draw[-{Latex[length=1.5mm,width=2mm]}] (2b) to[bend right=30] (2a);
	\draw[-{Latex[length=1.5mm,width=2mm]}] (1a) to[out=310, in=50, looseness=7] (1a);
\end{tikzpicture}
	\caption{}
	\label{fig:period2a}
\end{subfigure}
\vline
\begin{subfigure}[t]{0.45\textwidth}
\centering
\begin{tikzpicture}[scale=.9]
	\tikzset{vertex/.style = {}}
	\tikzset{every loop/.style={min distance=10mm,in=45,out=-45,->}}
	\tikzset{edge/.style={decoration={markings,mark=at position 1 with %
    {\arrow[scale=1.5,>=stealth]{>}}},postaction={decorate}}}
	%
	%
	\node[vertex] (2a) at  (2, 0) {$0$};
	\node[vertex] (2b) at  (4, 0) {$\omega_3$};
	\node[vertex] (2-1) at (0, 0) {$\omega_1$};
	\node[vertex] (2-2) at (-2, 0) {$\omega_4$};
	\node[vertex] (q) at (-2, -1.15) {}; 
	%
	%
	\draw[-{Latex[length=1.5mm,width=2mm]}] (2a) to[bend right=30] (2b);
	\draw[-{Latex[length=1.5mm,width=2mm]}] (2b) to[bend right=30] (2a);
	\draw[-{Latex[length=1.5mm,width=2mm]}] (2-1) to (2a);
	\draw[-{Latex[length=1.5mm,width=2mm]}] (2-2) to (2-1);
\end{tikzpicture}
	\caption{}
	\label{fig:period2b}
\end{subfigure}
	\caption{Two graphs containing {\bf(2)b}}
	\label{fig:period2}
\end{figure}
Any height-zero skeleton that contains {\bf(2)b} but is not isomorphic to any of the four listed in the proposition must contain one of the two appearing in Figure~\ref{fig:period2}. We show that if $\scrS(x^d + c, K)$ contains the graph in Figure~{\sc\ref{fig:period2a}} (resp., Figure~{\sc\ref{fig:period2b}}), then $\scrS(x^d + c, K)$ must actually contain---and therefore be isomorphic to---the graph {\bf (2,1,1)} (resp., {\bf (2)c}).

Let $\phi(x) = x^d + c$, and suppose that $\scrS(\phi, K)$ contains the graph in Figure~{\sc\ref{fig:period2a}}. 
The relation $\phi(\omega_1) = \omega_1$ implies that
	\[
		\omega_1^d + \omega_3 = \omega_1, \quad\text{hence}\quad \omega_3 = \omega_1 - \omega_1^d,
	\]
and therefore $\omega_1 = \zeta_6\omega_3$ by Lemma~\ref{lem:x+y=c}. The same argument as when we considered the graph {\bf(2,1,1)} above shows that $d \equiv 5 \pmod 6$.

We now claim that $\omega_2 := \zeta_6^{-1}\omega_3 \ne \omega_1$ is necessarily a fixed point for $\phi$ as well, hence $\scrS(\phi, K)$ contains {\bf(2,1,1)}. Indeed, we have
	\begin{align*}
		\phi(\omega_2)
			&= \omega_2^d + \omega_3\\
			&= \zeta_6^{-d}\omega_3^d + \omega_3\\
			&= \zeta_6^{-5}(-\omega_3) + \omega_3\\
			&= (-\zeta_6 + 1)\omega_3\\
			&= \zeta_6^{-1}\omega_3 = \omega_2.
	\end{align*}
This shows that if $\scrS(\phi, K)$ contains the graph from Figure~{\sc\ref{fig:period2a}}, then it must actually be isomorphic to {\bf(2,1,1)}.

Now suppose that $\scrS(\phi, K)$ contains the graph from Figure~{\sc\ref{fig:period2b}}. Since $\phi(\omega_4) = \omega_1$, we have
	\[
		\omega_3 = \omega_1 - \omega_4^d,
	\]
so $\omega_1 = \zeta_6\omega_3$ and $\omega_4^d = -\zeta_6^{-1}\omega_3 = \zeta_3\omega_3$. On the other hand, $\omega_1$ and $\omega_3$ have the same image under $\phi$, so $6 \mid d$. If we set $\omega_2 := \zeta_6^{-1}\omega_3 \ne \omega_1$, then because $6 \mid d$ we have $\phi(\omega_2) = \phi(\omega_3) = 0$. Finally, if we set $\omega_5 := \omega_3^2\omega_4^{-1}$, we have
	\begin{align*}
		\phi(\omega_5)
			&= \omega_3^{2d}\omega_4^{-d} + \omega_3\\
			&= \omega_3^{2d}\cdot\zeta_3^{-1}\omega_3^{-1} + \omega_3\\
			&= (\zeta_3^{-1}\omega_3^{2(d-1)} + 1)\omega_3\\
			&= (\zeta_3^{-1} + 1)\omega_3\\
			&= \zeta_6^{-1}\omega_3 = \omega_2.
	\end{align*}
Therefore, $\scrS(\phi, K)$ contains---hence is isomorphic to---the graph {\bf(2)c}.
\end{proof}

To finish our discussion of graphs with $2$-cycles, we need to show that {\bf (2)a} is the only graph with a $2$-cycle not listed in Proposition~\ref{prop:most2cycles} that can be realized as $\scrS(x^d + c, K)$. Since a skeleton with a $2$-cycle must contain either {\bf(2)a} or {\bf(2)b}, it suffices to prove the following:

\begin{prop}\label{prop:(2)a}
Let $K$ be a number field, let $c \in K^\times$, and suppose that $\scrS(x^d + c, K)$ is a height-zero skeleton that contains the graph {\bf(2)a} from Table~\ref{tab:SPs}, so that $c = \omega_1 + \omega_2 \ne 0$ and $\omega_1^{d-1} = \omega_2^{d-1} = -1$ for some distinct $\omega_1,\omega_2 \in \mu_K$. Then $\scrS(x^d + c, K)$ is isomorphic to one of the following:
	\begin{enumerate}
		\item[\textup{(1)}] {\bf(2,2)}, if and only if \vspace{.05cm} 
			\begin{itemize}
			\item $d \equiv 1 \pmod 6$, and \vspace{.05cm} 
			\item $\omega_3 = c$ is a root of unity;
   \vspace{.1cm}  
			\end{itemize}
		\item[\textup{(2)}] {\bf(2)a}, otherwise.
	\end{enumerate}
\end{prop}

\begin{proof}
Based on what was already done in the proof of Proposition~\ref{prop:most2cycles}, all we need to do is show that if $\scrS(x^d + c, K)$ contains {\bf(2)a} but is not isomorphic to {\bf(2,2)}, then actually $\scrS(x^d + c, K)$ is isomorphic to {\bf(2)a}.

Since {\bf(2)a} already has two roots of unity with root of unity preimages, any height-zero skeleton properly containing {\bf(2)a} must contain $0$; moreover, we cannot have $0$ as a fixed point (since $c \ne 0$), so $0$ must appear in its own $2$-cycle, in which case {\bf(2,2)} is the full skeleton, or $0$ must appear in the same component as the $2$-cycle $\{\omega_1,\omega_2\}$, in which case $\scrS(x^d + c, K)$ would contain the graph illustrated in Figure~\ref{fig:2cycle}. However, the latter possibility is ruled out by Lemma~\ref{lem:w1w2w3}.
\end{proof}

\begin{figure}
\centering
\begin{tikzpicture}[scale=1]
	\tikzset{vertex/.style = {}}
	\tikzset{every loop/.style={min distance=10mm,in=45,out=-45,->}}
	\tikzset{edge/.style={decoration={markings,mark=at position 1 with %
    {\arrow[scale=1.5,>=stealth]{>}}},postaction={decorate}}}
	%
	%
	\node[vertex] (2a) at  (2, 0) {$\omega_1$};
	\node[vertex] (2b) at  (4, 0) {$\omega_2$};
	\node[vertex] (2-1) at (0, 0) {$\omega_3$};
	\node[vertex] (2-2) at (-2, 0) {$0$};
	%
	%
	\draw[-{Latex[length=1.5mm,width=2mm]}] (2a) to[bend right=30] (2b);
	\draw[-{Latex[length=1.5mm,width=2mm]}] (2b) to[bend right=30] (2a);
	\draw[-{Latex[length=1.5mm,width=2mm]}] (2-1) to (2a);
	\draw[-{Latex[length=1.5mm,width=2mm]}] (2-2) to (2-1);
\end{tikzpicture}
	\caption{A graph containing {\bf(2)b}}
	\label{fig:2cycle}
\end{figure}

\subsection{Graphs with fixed points}
Finally, we classify the height-zero skeleta with fixed points. To do this, we frequently make use of the following fact: 
\begin{lem}\label{lem:fixed-points}
Let $d \ge 2$, let $K$ be a number field, let $c \in K^\times$, and let $\phi(x) = x^d + c$.
	\begin{enumerate}
	\item[\textup{(1)}] The map $\phi$ has a fixed point $\omega_1\in\mu_K$ if and only if $c = \omega_1 - \omega_1^d$. \vspace{.1cm} 
	\item[\textup{(2)}] The map $\phi$ has a second fixed point $\omega_2\in\mu_K$ if and only if $\omega_2 = -\omega_1^d$, $\omega_1 = -\omega_2^d$, and $\omega_1\ne\omega_2$. \vspace{.1cm} 
        \item[\textup{(3)}] If $\scrS(\phi,K)$ is a height-zero skeleton with exactly one fixed point $\omega_1$, then $\omega_1$ is the only $K$-rational periodic point for $\phi$.
	\end{enumerate}
\end{lem}

\begin{proof}
Part (1) follows immediately from the fact that $\omega_1$ is fixed by $\phi$ if and only if $\omega_1^d + c = \phi(\omega_1) = \omega_1$.

Now suppose that $\phi$ has two fixed points $\omega_1,\omega_2$. That $c = \omega_1 + \omega_2$ follows from Corollary~\ref{cor:two-images}; this then implies that $\omega_2 = -\omega_1^d$. Since $\omega_2$ is a fixed point, we also have $c = \omega_2 - \omega_2^d$, so that $\omega_1 = -\omega_2^d$.
Conversely, if
	\[
		c = \omega_1 - \omega_1^d = \omega_1 + \omega_2 = \omega_2 - \omega_2^d,
	\]
then $\omega_1$ and $\omega_2$ are fixed points for $\phi$.

Part (3) follows from previous statements: if $\scrS(\phi,K)$ is a height-zero skeleton with exactly one fixed point, then $\scrS(\phi, K)$ cannot have points of period greater than $3$ by the proof of part (1) of Theorem~\ref{thm:uniform-bound}, equal to $3$ by Proposition~\ref{prop:3cycle}, or equal to $2$ by Propositions~\ref{prop:most2cycles} and~\ref{prop:(2)a}.
\end{proof}

We now classify all graphs with fixed points that are realized as $\scrS(x^d + c, K)$ for some number field $K$, $d > D_1(K)$, and $c \in K^\times$. We begin by considering those graphs with two fixed points.

\begin{prop}
Let $K$ be a number field, let $c \in K^\times$, and suppose that $\scrS(x^d + c, K)$ is a height-zero skeleton with two distinct fixed points, so that $c = \omega_1 + \omega_2 \ne 0$, $\omega_1^d = -\omega_2$, and $\omega_2^d = -\omega_1$ for some distinct $\omega_1,\omega_2\in\mu_K$. Then $\scrS(x^d + c, K)$ is isomorphic to one of the following:
	\begin{enumerate}
		\item[\textup{(1)}] {\bf(2,1,1)}, if and only if \vspace{.05cm}
			\begin{itemize}
			\item $d \equiv 5 \pmod 6$, and \vspace{.05cm}
			\item $c = \omega_3$ is a root of unity.\vspace{.05cm}
			\end{itemize}
		\item[\textup{(2)}] {\bf(1,1)}, otherwise.
	\end{enumerate}
\end{prop}

\begin{proof}
Any height-zero skeleton properly containing {\bf(1,1)} must either contain (hence be isomorphic to) {\bf(2,1,1)} or contain the graph in Figure~\ref{fig:2FPs}. As the necessary and sufficient conditions for $\scrS(x^d + c, K)$ to be isomorphic to {\bf(2,1,1)} were given in Proposition~\ref{prop:most2cycles}, all that remains is to show is that $\scrS(x^d + c, K)$ cannot contain the graph from Figure~\ref{fig:2FPs}. But this follows immediately from Lemma~\ref{lem:w1w2w3}.
\end{proof}

\begin{figure}
\centering
\begin{tikzpicture}[scale=1]
	\tikzset{vertex/.style = {}}
	\tikzset{every loop/.style={min distance=10mm,in=45,out=-45,->}}
	\tikzset{edge/.style={decoration={markings,mark=at position 1 with %
    {\arrow[scale=1.5,>=stealth]{>}}},postaction={decorate}}}
	%
	%
	\node[vertex] (1a) at (0, 0) {$\omega_1$};
	\node[vertex] (1b) at (2, 0) {$\omega_2$};
	\node[vertex] (1-1) at (-2, 0) {$\omega_3$};
	\node[vertex] (1-2) at (-4, 0) {$0$};
	%
	%
	\draw[-{Latex[length=1.5mm,width=2mm]}] (1a) to[out=310, in=50, looseness=7] (1a);
	\draw[-{Latex[length=1.5mm,width=2mm]}] (1b) to[out=310, in=50, looseness=7] (1b);
	\draw[-{Latex[length=1.5mm,width=2mm]}] (1-1) to (1a);
	\draw[-{Latex[length=1.5mm,width=2mm]}] (1-2) to (1-1);
\end{tikzpicture}
\caption{A graph with two fixed points}
\label{fig:2FPs}
\end{figure}

To complete the proof of Theorem~\ref{thm:prep-classification}, it remains to show that if $\scrS(x^d + c, K)$ is a height-zero skeleton with exactly one fixed point, then $\scrS(x^d + c, K)$ is isomorphic to one of the graphs {\bf(1)a} through {\bf(1)e} in Table~\ref{tab:SPs}. To help organize our proof of this assertion, we break this final case into two parts, depending on whether $\scrS(x^d + c, K)$ contains a graph isomorphic to {\bf(1)c}.

\begin{lem}\label{lem:contains(1)c}
    Let $K$ be a number field, let $c \in K^\times$, and suppose $\scrS(x^d + c, K)$ is a height-zero skeleton. Then $\scrS(x^d + c, K)$ contains a subgraph isomorphic to {\bf(1)c} if and only if
    \begin{itemize}
        \item $6 \mid d$,
        \item $c = \omega_3 = \omega_1 - \omega_1^d$, and
        \item $\omega_1 = \zeta_6\omega_3$.
    \end{itemize}
\end{lem}

\begin{proof}
Sufficiency of the conditions is not difficult, so we only prove necessity. Thus, let $\phi(x) = x^d + c$, and suppose $\scrS(\phi, K)$ contains {\bf(1)c}.

First, we observe that $c = \phi(0) = \omega_3$. Since $\phi(\omega_1) = \omega_1$, we have that
    \[
        \omega_3 = c = \omega_1 - \omega_1^d,\quad\text{hence}\quad \omega_1 = \zeta_6\omega_3.
    \]
Finally, since $\phi(\omega_3) = \omega_1 = \phi(\omega_1)$, we have
    \[
        \omega_3^d = \omega_1^d = \zeta_6^d\omega_3^d,\quad\text{hence}\quad 6\mid d.\qedhere
    \]
\end{proof}
\begin{prop}\label{prop:not(1)c}
Let $K$ be a number field, let $c \in K^\times$, and suppose that $\scrS(x^d + c, K)$ is a height-zero skeleton with a single fixed point $\omega_1\in \mu_K$, so that $c = \omega_1 - \omega_1^d \ne 0$.
Furthermore, suppose that $\scrS(x^d + c, K)$ does \emph{not} contain a subgraph isomorphic to {\bf(1)c}. Then $\scrS(x^d + c, K)$ is isomorphic to one of the following: \vspace{.05cm}
    \begin{enumerate}
		\item[\textup{(1)}] {\bf(1)b}, if and only if \vspace{.05cm}
			\begin{itemize}
			\item $\omega_2^{d-1} = -1$,\vspace{.05cm}
			\item $\omega_1^d = -\omega_2$, \vspace{.05cm}
			\item $\omega_3^d = -\omega_1$, and \vspace{.05cm}
			\item $\omega_1$, $\omega_2$, and $\omega_3$ are distinct. \vspace{.2cm}
			\end{itemize}
		\item[\textup{(2)}] {\bf(1)a}, otherwise.
    \end{enumerate}
\end{prop}

\begin{proof}
    Let $\phi(x) = x^d + c$, and suppose $\scrS(\phi, K)$ satisfies the hypotheses of the proposition. By Lemma~\ref{lem:fixed-points}, since $\scrS(\phi, K)$ has a single fixed point, there can be no other periodic points in $\scrS(\phi,K)$. Combined with the fact that $\scrS(\phi,K)$ does not contain {\bf(1)c}, either $\scrS(\phi,K)$ is isomorphic to {\bf(1)a} or {\bf(1)b}, or $\scrS(\phi,K)$ contains the graph in Figure~{\sc\ref{fig:not(1)c}}. Therefore, it suffices to show that $\scrS(\phi, K)$ \emph{contains} a subgraph isomorphic to {\bf(1)b} if and only if the four conditions described in part (1) are satisfied and that $\scrS(\phi,K)$ cannot contain the graph from Figure~{\sc\ref{fig:not(1)c}}.

\begin{figure}
\begin{subfigure}[t]{\textwidth}
\centering
\begin{tikzpicture}[scale=1]
	\tikzset{vertex/.style = {}}
	\tikzset{every loop/.style={min distance=10mm,in=45,out=-45,->}}
	\tikzset{edge/.style={decoration={markings,mark=at position 1 with %
    {\arrow[scale=1.5,>=stealth]{>}}},postaction={decorate}}}
	%
	%
	\node[vertex] (1a) at (2, 0) {$\omega_1$};
	\node[vertex] (1-1) at (0, 0) {$\omega_2$};
	\node[vertex] (1-2) at (-2, 0) {$\omega_3$};
	\node[vertex] (1-3) at (-4, 0) {$0$};
	%
	%
	\draw[-{Latex[length=1.5mm,width=2mm]}] (1a) to[out=310, in=50, looseness=7] (1a);
	\draw[-{Latex[length=1.5mm,width=2mm]}] (1-1) to (1a);
	\draw[-{Latex[length=1.5mm,width=2mm]}] (1-2) to (1-1);
	\draw[-{Latex[length=1.5mm,width=2mm]}] (1-3) to (1-2);
\end{tikzpicture}
    \caption{}
    \label{fig:not(1)c}
\end{subfigure}
\begin{subfigure}[t]{\textwidth}
\centering
\begin{tikzpicture}[scale=.9]
	\tikzset{vertex/.style = {}}
	\tikzset{every loop/.style={min distance=10mm,in=45,out=-45,->}}
	\tikzset{edge/.style={decoration={markings,mark=at position 1 with %
    {\arrow[scale=1.5,>=stealth]{>}}},postaction={decorate}}}
	%
	%
	\node[vertex] (1-3) at (-2, 0) {$\omega_5$};
	\node[vertex] (1-2) at (0, 0) {$0$};
	\node[vertex] (1-1) at (2, 0) {$\omega_3$};
	\node[vertex] (1) at (4, 0) {$\omega_1$};
	%
	%
	\draw[-{Latex[length=1.5mm,width=2mm]}] (1) to[out=310, in=50, looseness=7] (1);
	\draw[-{Latex[length=1.5mm,width=2mm]}] (1-1) to (1);
	\draw[-{Latex[length=1.5mm,width=2mm]}] (1-2) to (1-1);
	\draw[-{Latex[length=1.5mm,width=2mm]}] (1-3) to (1-2);
\end{tikzpicture}
        \caption{}
	\label{fig:short1FP}
 \end{subfigure}
	\caption{Two graphs with a single fixed point}
\end{figure}

It is straightforward to show that the four conditions in part (1) are, together with the initial assumptions on $\scrS(\phi,K)$, sufficient for $\scrS(\phi,K)$ to contain {\bf(1)b}. We now show those conditions are necessary, so suppose that $\scrS(\phi, K)$ contains {\bf(1)b}. The points $\omega_1$, $\omega_2$, and $\omega_3$ are necessarily distinct. By Corollary~\ref{cor:two-images}, we have $\omega_1 + \omega_2 = c = \omega_1 - \omega_1^d$, so $\omega_1^d = -\omega_2$. Since $\phi(\omega_2) = \omega_1 = \phi(\omega_1)$, we have $\omega_1^d = \omega_2^d$; thus,
    \[
       \omega_2^d = \omega_1^d = -\omega_2, \quad\text{hence}\quad \omega_2^{d-1} = -1.
    \]
Finally, the relation $\omega_3^d + \omega_1 +\omega_2 = \omega_3^d + c = \phi(\omega_3) = \omega_2$ implies that $\omega_3^d = -\omega_1$.

It remains to show that $\scrS(\phi,K)$ cannot contain the graph from Figure~{\sc\ref{fig:not(1)c}}. Suppose for contradiction that $\scrS(\phi,K)$ does contain such a subgraph. Then $\scrS(\phi,K)$ also contains {\bf(1)b}, so the four conditions of part (1) are satisfied, but now with the additional condition that $c = \phi(0) = \omega_3$. From Corollary~\ref{cor:two-images}, we have
    \[
        \omega_1 = \zeta_6\omega_3\quad\text{and}\quad\omega_2 = \zeta_6^{-1}\omega_3 = \zeta_3^{-1}\omega_1.
    \]
Since $\phi(\omega_2) = \omega_1 = \phi(\omega_1)$, we have $3 \mid d$; on the other hand, since $\phi(\omega_3) = \omega_2 \ne \phi(\omega_1)$, we have $6 \nmid d$. In other words, $d \equiv 3 \pmod 6$, which then implies that
    \begin{equation}\label{eq:not(1)c}
        \omega_1^d = \omega_2^d = -\omega_3^d.
    \end{equation}
The relation $\omega_3^d + \omega_3 = \phi(\omega_3) = \omega_2$ implies that
    \[
        \omega_1 + \omega_2 = \omega_3 = \omega_2 - \omega_3^d,
    \]
hence $\omega_3^d = -\omega_1$. Combined with \eqref{eq:not(1)c}, we have $\omega_1^d = \omega_1$. However, this implies that
    \[
        c = \phi(\omega_1) - \omega_1^d = \omega_1 - \omega_1^d = 0,
    \]
and since we specifically required $c \ne 0$, it follows that $\scrS(\phi,K)$ cannot contain the graph from Figure~{\sc\ref{fig:not(1)c}}.
\end{proof}

\begin{prop}\label{prop:contains(1)c}
Let $K$ be a number field, let $c \in K^\times$, and suppose that $\scrS(x^d + c, K)$ is a height-zero skeleton with a single fixed point $\omega_1\in \mu_K$, so that $c = \omega_1 - \omega_1^d \ne 0$.
Furthermore, suppose that $\scrS(x^d + c, K)$ contains a subgraph isomorphic to {\bf(1)c}, so that the conditions of Lemma~\ref{lem:contains(1)c} are satisfied. Then $\scrS(x^d + c, K)$ is isomorphic to one of the following:
	\begin{enumerate}
            \item[\textup{(1)}] {\bf(1)c} if and only if $\zeta_3$ is not a $d$th power in $K$. \vspace{.1cm}
		\item[\textup{(2)}] {\bf(1)d} if and only if $\zeta_3$ is a $d$th power in $K$ but $\zeta_6$ is not. \vspace{.1cm}
		\item[\textup{(3)}] {\bf(1)e} if and only if $\zeta_6$ is a $d$th power in $K$.
	\end{enumerate}
\end{prop}

\begin{proof}
Suppose that $\phi(x) = x^d + c$ satisfies the hypotheses of the proposition. As in the proof of Proposition~\ref{prop:not(1)c}, the fact that $\scrS(\phi, K)$ is a height-zero skeleton with a single fixed point implies that there are no periodic points other than that one fixed point. Combined with the assumption that $\scrS(\phi,K)$ contains {\bf(1)c}, one of the following must be true: $\scrS(\phi,K)$ is isomorphic to {\bf(1)c}, $\scrS(\phi,K)$ contains {\bf(1)d}, or $\scrS(\phi,K)$ contains the graph from Figure~{\sc\ref{fig:short1FP}}.

We begin by showing that if $\scrS(\phi,K)$ contains the graph from Figure~{\sc\ref{fig:short1FP}}, then in fact $\scrS(\phi,K)$ contains {\bf(1)e}, hence is isomorphic to {\bf(1)e}, since {\bf(1)e} is a maximal height-zero skeleton.
Thus, suppose $\scrS(\phi,K)$ contains the graph from Figure~{\sc\ref{fig:short1FP}}. In addition to the conditions
    \begin{itemize}
        \item $6 \mid d$,
        \item $c = \omega_3 = \omega_1 - \omega_1^d$, and
        \item $\omega_1 = \zeta_6\omega_3$
    \end{itemize}
from Lemma~\ref{lem:contains(1)c}, we also have $\phi(\omega_5) = 0$, hence $\omega_5^d = -\omega_3$.

We now claim that if we take
    \[
        \omega_2 := \zeta_6^{-1}\omega_3 \quad\text{and}\quad \omega_4 := \omega_5^2\omega_1^{-1},
    \]
then $\phi(\omega_4) = \omega_2$ and $\phi(\omega_2) = \omega_1$, so that the full skeleton $\scrS(\phi, K)$ contains {\bf(1)e}. (That $\omega_2 \notin\{\omega_1,\omega_3\}$ is clear.)

Since $6 \mid d$ and $\omega_2 = \zeta_6^{-1}\omega_3$, we certainly have $\phi(\omega_2) = \phi(\omega_3) = \omega_1$. 
To show that $\phi(\omega_4) = \omega_2$, we write
    \[
        \phi(\omega_4)
            = \frac{(\omega_5^d)^2}{\omega_1^d} + \omega_3
            = \frac{\omega_3^2}{\omega_1 - \omega_3} + \omega_3
            = \frac{\omega_3^2}{(\zeta_6 - 1)\omega_3} + \omega_3
            = (\zeta_3^{-1} + 1)\omega_3
            = \zeta_6^{-1}\omega_3 = \omega_2.
    \]
Therefore, $\scrS(\phi,K)$ contains---hence is isomorphic to---the graph {\bf(1)e}.

What we have shown so far is that if $\scrS(\phi,K)$ contains {\bf(1)c}, then $\scrS(\phi,K)$ must be isomorphic to one of {\bf(1)c}, {\bf(1)d}, and {\bf(1)e}. It remains to prove the following two statements, assuming that $\scrS(\phi,K)$ contains {\bf(1)c}:

    \begin{enumerate}[label=(\alph*)]
        \item The skeleton $\scrS(\phi,K)$ contains {\bf(1)d} if and only if $\zeta_3$ is a $d$th power in $K$.
        \item The skeleton $\scrS(\phi,K)$ contains {\bf(1)e} if and only if $\zeta_6$ is a $d$th power in $K$.
    \end{enumerate}

We begin with (a). Since we have assumed that $\scrS(\phi,K)$ contains {\bf(1)c}, we again assume the conditions of Lemma~\ref{lem:contains(1)c} hold. Setting $\omega_2 := \zeta_6^{-1}\omega_3$ and using the fact that $6\mid d$, we have $\phi(\omega_2) = \phi(\omega_3) = \omega_1$. Thus, it suffices to show that there exists $\omega_4 \in \mu_K$ with $\phi(\omega_4) = \omega_2$ if and only if $\zeta_3$ is a $d$th power in $K$. Indeed, we have $\phi(\omega_4) = \omega_2$ if and only if
    \[
        \omega_4^d = \omega_2 - \omega_3 = (\zeta_6^{-1} - 1)\omega_3 = \zeta_3^{-1}\omega_3.
    \]
On the other hand, we already know that
    \[
        \omega_1^d = \omega_1 - \omega_3 = (\zeta_6 - 1)\omega_3 = \zeta_3\omega_3,
    \]
so $\phi(\omega_4) = \omega_2$ if and only if $\omega_4^d = \zeta_3\omega_1^d$, which has a solution $\omega_4 \in \mu_K$ if and only if $\zeta_3$ is a $d$th power in $K$.

The proof of (b) is similar: The skeleton $\scrS(\phi,K)$ contains {\bf(1)e} if and only if $\scrS(\phi,K)$ contains {\bf(1)d} and there exists $\omega_5 \in \mu_K$ for which $\phi(\omega_5) = 0$. Now, $\phi(\omega_5) = 0$ if and only if
    \[
        \omega_5^d = -\omega_3 = -\zeta_3^{-1}\omega_1^d = \zeta_6\omega_1^d,
    \]
so there exists such $\omega_5 \in \mu_K$ if and only if $\zeta_6$ is a $d$th power in $K$.
\end{proof}
\section{Irreducible polynomials in semigroups}\label{red:irreducibility}
We briefly recall the setup from the introduction. Fix a field $K$ of characteristic $0$, and let $S = \{\phi_1,\ldots,\phi_s\}$ be a set of polynomials in $K[x]$ with degrees $d_i = \deg \phi_i \ge 2$. We denote by $M_S$ the semigroup generated by $S$ under composition.

We will specifically be interested in the case that $\phi_i(x) = x^{d_i} + c_i$ for some $c_i \in K$. The proof of our main irreducibility result, Theorem \ref{thm:main+irred}, has a number of steps, so we provide here a basic outline of the proof.\\[2pt]
\indent\textbf{Step(1)}: First, since the semigroups $M_S$ we consider are free \cite[Theorem 3.1]{MR4349782}, it suffices to construct an irreducible $g\in M_S$ such that $g\circ f$ is irreducible for all $f\in M_S$. To do this, we take an irreducible $\phi=x^d+c\in S$ and try and prove that $g=\phi^N$ is irreducible for some sufficiently large $N$. This is accomplished by ruling out $m$th powers in the critical orbit $\{\phi(0),\phi^2(0),\dots\}$ by studying the Fermat-Catalan equations $Y^m=X^d+c$ where $m|d$. In particular, we can do this when $d\gg_K0$ and $h(c)\gg_K0$ by using the height bounds in Proposition \ref{prop:fermat-catalan}; see Theorem  \ref{prop:oldstability} and Proposition \ref{prop:uniformstability}.\\[3pt] 
\indent\textbf{Step(2):} Next, we try and prove that $\phi^N\circ f$ is irreducible for all $f\in M_S$. In fact, if this is not the case, then roughly $\phi^N(a)=y^m$ for some $a,y\in K$ and $m\geq2$. But then $(\phi^{N-1}(a),y)$ is a solution to the Fermat-Catalan equation $Y^m=X^d+c$, so that the height of $\phi^{N-1}(a)$ is  bounded. However, by making $N\gg_{K,d}0$ this means that $a$ must be preperiodic for $\phi$; here we use a result of Looper \cite{looper2021dynamical}. But again, by enlarging $N$ if necessary, this means that $y^m=\phi^N(a)$ is a \emph{periodic point} for $\phi$. Moreover, by enlarging $d$ if necessary, any periodic point must be a fixed point by Theorem \ref{thm:uniform-bound}. Therefore, we have succeeded in proving that $\phi^N\circ f$ is irreducible for any $f\in M_S$ unless $\phi\in S$ has a powered fixed point; see Definition \ref{def:poweredfixedpoint}. Moreover, since the choice of $\phi$ was arbitrary as long as $\phi$ is irreducible, we are done unless every irreducible map in $S$ has a powered fixed point.\\[3pt]  
\indent\textbf{Step(3):} Lastly, we deal with the case of maps with powered fixed points separately by examining some additional diophantine equations. For example, if $\phi_1$ and $\phi_2$ both have powered fixed points and $\phi_1(K)\cap \PrePer(\phi_2,K)\neq\varnothing$ and $\phi_2(K)\cap \PrePer(\phi_1,K)\neq\varnothing$, then the ratio of the powered fixed points of $\phi_1$ and $\phi_2$ are roots of unity; see Lemma  \ref{lem:two+special+overlap}, rephrased as a system of diophantine equations. In particular, when this is not the case, either $\phi_1^N\circ\phi_2\circ f$ or $\phi_2^N\circ\phi_1\circ g$ must be irreducible for all $f,g\in M_S$. Moreover, a similar argument can be used when $\phi_2$ is reducible. \\[3pt]
\indent We begin with the following irreducibility test for polynomials having a unicritical polynomial as a compositional factor; compare to \cite[Corollary 2.3]{JonesIMRN}.
\begin{thm}\label{prop:oldstability} Let $K$ be a field of characteristic zero, let $f,g\in K[x]$ with $f(x)=x^d+c$ for some $d\geq2$, and assume that $g$ is monic and irreducible in $K[x]$. Then, if $g\circ f$ is reducible in $K[x]$, we have that  
\[g(f(0))=(-1)^{e_1}4^{e_2} y^m\] 
for some $e_1,e_2\in\{0,1\}$ and some divisor $m\geq2$ of $d$. 
\end{thm}
\begin{remark}\label{rem:stability} However, if $K$ contains a primitive $d$th root of unity, then one may consider only the case when $g(f(0))=\pm{y^p}$ for some prime $p$ dividing $d$ (i.e., $e_2=0$); see \cite[Corollary 2.3]{JonesIMRN}.  
\end{remark}
\begin{proof} Let $f$ and $g$ be as above and assume that $g\circ f$ is reducible over $K$. Then Capelli's Lemma implies that $f(x)-\alpha=x^d+c-\alpha$ is reducible over $K(\alpha)$ for some root $\alpha\in\overline{K}$ of $g$. From here, \cite[Theorem 9.1, p. 297]{MR1878556} implies that $\alpha-c=z^p$ for some $z\in K(\alpha)$ and some prime $p|d$ or $\alpha-c=4z^4$ when $4|d$. On the other hand, since $g\in K[x]$ is irreducible and the norm map is multiplicative, we have that  
\begin{equation*}
\begin{split} 
N_{K(\alpha)/K}(\alpha-c)=(-1)^{\deg(g)}N_{K(\alpha)/K}(c-\alpha)=(-1)^{\deg(g)}g(c)=(-1)^{\deg(g)}g(f(0)).
\end{split} 
\end{equation*}
In particular, when $\alpha-c=z^p$, then $(-1)^{\deg(g)}y^p=g(f(0))$ for $y=N_{K(\alpha)/K}(z)\in K$ and the claim follows. Likewise, assume that $\alpha-c=4z^4$ in the case when $4|d$. If $\deg(g)=2m$ is even, then the norm calculation above implies that
\[\Big(2^mN_{K(\alpha)/K}(z)\Big)^4=4^{\deg(g)}\Big(N_{K(\alpha)/K}(z)\Big)^4=g(f(0))\]
is a fourth power as claimed. Likewise, when $\deg(g)=2m+1$ is odd, then 
\[-4\Big(2^mN_{K(\alpha)/K}(z)\Big)^4=g(f(0))\]
and the result follows.  
\end{proof}
Thus, to produce many irreducible polynomials in $M_S$ it suffices to construct an irreducible polynomial $g\in M_S$ with no perfect powered images (up to a small multiple). With this in mind, we first prove that an irreducible polynomial $x^d+c$ for $c\in K$ is \emph{stable} over $K$ (meaning that all iterates of $\phi$ are irreducible over $K$) whenever $d$ is sufficiently large and $c$ has sufficiently large height.
\begin{prop}\label{prop:uniformstability} Let $K$ be an $abc$-field of characteristic zero. Then there is an absolute constant $C_4$ and a constant $D_4=D_4(K)$ depending on $K$ such that if $h(c)>C_4$ and $d\geq D_4(K)$, then $\phi(x)=x^d+c$ is irreducible over $K$ if and only if $\phi$ is stable over $K$. Moreover, one can take $C_4=0$ and $D_4=\max\{14,4g_K+10\}$ when $K$ is a function field and $C_4=\frac{3}{5}\log2$ when $K/\mathbb{Q}$ is a number field.   
\end{prop}
Before we give the proof of this result, we recall some basic properties of dynamical canonical heights. The canonical height function attached to a rational function $\phi(x) \in K(x)$ is the unique function $\hat{h}_\phi : \P^1(\Kbar) \to \R$ satisfying
    \[
        \hat{h}_\phi(\alpha) = h(\alpha) + O(1) \quad\text{and}\quad \hat{h}_\phi(\phi(\alpha)) = \deg\phi \cdot \hat{h}(\alpha)
    \]
for all $\alpha \in \P^1(\Kbar)$; see \cite[\textsection 3.4]{SilvDyn}. We will require the following basic facts about canonical heights for unicritical polynomials; see, for instance, \cite[Lemma 2.2]{MR4127857}.
\begin{lem}\label{lem:basic} Let $K$ be a number field or a function field of characteristic zero and let $\phi(x)=x^d+c$ for some $d\geq2$ and $c\in K$. Then we have that 
\[|h(\alpha)-\hat{h}_\phi(\alpha)|\leq \frac{1}{(d-1)}(h(c)+\log 2)
\;\;\;\;\text{and}\;\;\;\;
\hat{h}_{\phi}(\phi^m(\alpha))=d^m\hat{h}_\phi(\alpha)
\]
for all $\alpha\in \overline{K}$ and $m\geq0$.
\end{lem}
\begin{proof}[Proof of Proposition \ref{prop:uniformstability}] 
Suppose that $c\in K$ is such that $h(c)>\frac{3}{5}\log2$. (Note that when $K$ is a function field, this condition on the height is equivalent to saying that $c$ is non-constant.) Moreover, assume that $\phi^n(0)=ry^m$ for some $y\in K$, some $n\geq2$, some $r$ in the set  
\begin{equation}\label{eq:r's}
R_K:=\big\{(-1)^{e_1}4^{e_2}\zeta\;:\; e_1,e_2\in\{0,1\}\;\text{and}\; \zeta\in\mu_{K}\big\}, 
\end{equation} 
and some divisor $m|d$ with $ m\geq2$. (Technically speaking, we only need to consider the  case when $r\in\{\pm{1},\pm{4}\}$ for this result. However, the more general case will be useful for later estimates.) Then $(X,Y)=(\phi^{n-1}(0),y)$ is a solution to the generalized Fermat-Catalan equation $(-c^{-1})X^d+(rc^{-1})Y^m=1$. Hence, Proposition \ref{prop:fermat-catalan} implies that
\begin{equation}\label{eq:uniformstability1}
d h(\phi^{n-1}(0))\leq B_1 h(c)+B_2, 
\end{equation}
for some constants $B_1$ and $B_2$, with $B_2$ depending on $K$.
On the other hand, Lemma \ref{lem:basic} applied to the point $\alpha=\phi^{n-1}(0)$ implies that
\begin{equation*}
d^{n-1}\hat{h}_\phi(0)-\frac{1}{(d-1)}(h(c)+\log2)=\hat{h}_\phi(\phi^{n-1}(0))-\frac{1}{(d-1)}(h(c)+\log2)\leq h(\phi^{n-1}(0)).
\end{equation*}
In particular, multiplying the above inequality by $d$ and combining with \eqref{eq:uniformstability1} implies that 
\begin{equation}\label{eq:uniformstability2}
\begin{split} 
d^n\hat{h}_\phi(0)&\leq\frac{d}{d-1}(h(c)+\log2)+dh(\phi^{n-1}(0))\\[4pt] 
&\leq\frac{d}{d-1}(h(c)+\log2)+B_1h(c)+B_2\\[4pt]
&\leq B_3 h(c)+B_4\\[2pt]   
\end{split}
\end{equation}
for some constants $B_3$ and $B_4$, with $B_4$ depending on $K$. From here, applying Lemma \ref{lem:basic} to the point $\alpha=c$, increasing $D_4$ if necessary so that $d \ge 4$, and using the fact that $h(c)\geq \frac{3}{5}\log2$, we see that  
\begin{equation}\label{eq:uniformstability3}
\frac{1}{9}h(c)\leq\frac{2}{3}h(c)-\frac{1}{3}\log2\leq\frac{d-2}{d-1}h(c)-\frac{\log2}{d-1}\leq \hat{h}_\phi(c)=\hat{h}(\phi(0))=d\hat{h}_\phi(0). 
\end{equation} 
In particular, by combining \eqref{eq:uniformstability2} and \eqref{eq:uniformstability3} with the fact that $n\geq2$, we deduce that 
\[d\Big(\frac{1}{9}h(c)\Big)\leq d^{n-1}\Big(\frac{1}{9}h(c)\Big)\le d^{n-1}(d\hat{h}_\phi(0))=d^n\hat{h}_\phi(0)\leq B_3 h(c)+ B_4.\]
Hence, $d\leq 9B_3+15B_4/\log2$, so that $\phi^n(0)\neq ry^m$ for all $m\geq2$, all $y\in K$, all $r\in R_K$, and all $n\geq2$ whenever $d\gg_K0$. From here, repeated application of Theorem \ref{prop:oldstability} then implies that $\phi^n$ is irreducible for all $n\geq2$ whenever $\phi$ is irreducible as desired. 
\end{proof}
\begin{remark} In the function field case, a simpler argument yields the explicit constants mentioned in Proposition \ref{prop:uniformstability}. Namely, if $c$ is non-constant, then it is straightforward to verify that $h(\phi^{t}(0))=h(\phi^{t-1}(c))=d^{t-1}h(c)$ for all $t\geq1$. Hence, using \cite[Theorem 2]{MR0664038} instead of Proposition \ref{prop:fermat-catalan}, the bound in \eqref{eq:uniformstability1} may be replaced by
\[\frac{1}{2}(d-4)h(c)\leq\frac{1}{2}(d-4)d^{n-2}h(c)\leq\Big(d-2-\frac{d}{m}\Big)d^{n-2}h(c)\leq5h(c)+2g_K-2\]
for all $n \ge 2$.
In particular, the claim that $C_4=0$ and that $D_4=\max\{14,4g_K+10\}$ follows. 
\end{remark}
\begin{remark}\label{rem:criticalorbitpowers} Since it will be useful for future arguments, we note that in the course of proving Proposition \ref{prop:uniformstability} we established the following fact regarding the set $R_K$ defined in \eqref{eq:r's}: for all $h(c)\gg_K0$ and $d\gg_K0$, if $\phi^n(0)=ry^m$ for some $y\in K$, some $n\geq1$, some $r\in R_K$, and some $m\geq2$, then necessarily $n=1$. Moreover, we note that $R_K$ is a set of bounded height (independent of $K$) and is closed under multiplication by roots of unity and their negatives.     
\end{remark}
Next (and with the goal of avoiding powers in mind), we have the following new result, which implies that if some large iterate of $\phi$ is a perfect power, then $\phi$ has a periodic point that is a perfect power; compare to similar results in \cite{Mathworks2, Mathworks1} over $\mathbb{Z}$ and $\mathbb{Q}$.
\begin{prop}\label{prop:uniformpower} Let $K$ be an $abc$-field of characteristic zero, let $d\geq5$ and $m\geq2$, and let $\phi(x)=x^d+c$ for some $c\in K^\times$. Moreover, if $K$ is a function field, we assume that $c$ is non-constant. Then there is a constant $N=N(K,d)$ such that if $\phi^n(a)=ry^m$ for some $a,y\in K$, some $r\in R_{K}$, and some $n\geq N$, then $ry^m$ is a periodic point for $\phi$.\end{prop}
\begin{proof}
The first steps in the proof are similar to the proof of Proposition \ref{prop:uniformstability}. Suppose that $\phi^n(a)=ry^m$ for some $a,y\in K$, some $r\in R_{K}$, some $d\geq5$ and $m\geq2$. Moreover, assume that $a$ is not preperiodic for $\phi$. Then $\hat{h}_\phi(a)>0$ by \cite[Theorem 3.22]{SilvDyn} when $K$ is a number field and \cite[Theorem A]{MR2202175} when $K$ is a function field; here we use that $c$ is non-constant in the function field case---see Remark~\ref{rem:isotrivial} below. Thus $(X,Y)=(\phi^{n-1}(a),y)$ is a solution to the generalized Catalan equation $(-c^{-1})X^d+(rc^{-1})Y^m=1$. Hence, Proposition \ref{prop:fermat-catalan} and Lemma \ref{lem:basic} imply that
\begin{equation}\label{eq:uniformpower1}
d^n \hat{h}_\phi(a)\leq B_3h(c)+B_4, 
\end{equation}
for some constants $B_3$ and $B_4$ with $B_4$ depending on $K$; see \eqref{eq:uniformstability2} and simply replace $0$ with $a$. On the other hand, Lemma \ref{lem:basic} and the fact that $c=\phi(0)$ imply that  
\begin{equation}\label{eq:uniformpower2}
h(c)\leq \frac{d(d-1)}{(d-2)} \max\big\{\hat{h}_\phi(0),1\big\}+\frac{\log2}{d-2}.   
\end{equation}
In particular, by combining \eqref{eq:uniformpower1} and \eqref{eq:uniformpower2}, we deduce that
\begin{equation}\label{eq:uniformpower3}
d^{n}\hat{h}_\phi(a)\leq B_{5}\max\big\{\hat{h}_\phi(0),1\big\}+B_{6}
\end{equation} 
for some constants $B_5$ and $B_{6}$ depending on $K$ and $d$. But, \cite[Theorem 1.4]{looper2021uniform} implies that
\begin{equation}\label{eq:uniformpower5}
\hat{h}_\phi(a)\geq \kappa \max\big\{\hat{h}_\phi(0),1\big\}
\end{equation} 
for some positive constant $\kappa$ depending on $K$ and $d$. Therefore,
\[
d^{n}\leq B_7 := (B_5+B_6)/\kappa,
\]
where $B_7$ is now a positive constant depending on $K$ and $d$. Thus, we see that there is a constant $N_1=N_1(K,d)$ such that if $\phi^n(a)=ry^m$ for some $n\geq N_1$, $y\in K$, $r\in R_{K}$, and $m\geq2$, then $a$ must be \emph{preperiodic} for $\phi$.

On the other hand, by the Dynamical Uniform Boundedness Theorem for the family $x^d+c$ (see \cite[Theorem 1.7]{MR4065068} and \cite[Theorem 1.2]{looper2021dynamical}), there is an $N_2=N_2(K,d)$ such that if $a\in K$ is a preperiodic point for $\phi(x)=x^d+c$, then $\phi^{n}(a)$ must be a \emph{periodic} point for $\phi$ for all $n\geq N_2$ (that is, the tail lengths of preperiodic points are uniformly bounded). In particular, if we define $N(K,d):=\max\{N_1(K,d),N_2(K,d)\}$ and assume that $\phi^n(a)=ry^m$ for some $n\geq N$, $m\geq2$, and some $y\in K$, then $a$ is preperiodic by definition of $N_1$ and $\phi^{n}(a)=ry^m$ must be periodic by definition of $N_2$. The claim follows.  
\end{proof}

\begin{remark}\label{rem:isotrivial}
In the proof Proposition~\ref{prop:uniformpower}, we apply \cite[Theorem A]{MR2202175} in the case that $K/k(t)$ is a function field and $\phi(x) = x^d + c$ with $c \in K$ non-constant. However, to apply \cite[Theorem A]{MR2202175}, we require $\phi$ to be \emph{nonisotrivial}; that is, $\phi$ cannot be conjugate to a rational function defined over $\bar k$. Here we sketch a proof of the standard fact that $\phi(x) = x^d + c$ is isotrivial if and only if $c \in \bar k$.

Certainly if $c \in \bar k$, then $\phi$ itself is in $\bar k[x]$, so $\phi$ is isotrivial. Now suppose that $\phi$ is isotrivial, in which case $\phi$ is conjugate to a map $\psi(x) \in k(x)$. Since $\phi$ has only two ramification points (namely $0$ and $\infty$), so must $\psi$; moreover, since $\psi \in \bar k(x)$, the ramification points for $\psi$ are elements of $\P^1(\bar k)$, so there is a change of coordinates over $\bar k$ that moves those ramification points to $0$ and $\infty$. Thus, $\phi$ has a conjugate of the form $Ax^d + B$ with $A,B \in \bar k$, and therefore, after further conjugating by the map $A^{1/(1-d)}x$, we have that $\phi$ is conjugate to a map of the form $x^d + C$ with $C \in \bar k$. Finally, the two maps $x^d + c$ and $x^d + C$ are conjugate if and only if $c$ and $C$ differ by a factor of a $(d-1)$th root of unity, so $c \in \bar k$ as well.
\end{remark}

In particular, if $\phi^N$ is irreducible and $\phi^N\circ f$ is reducible for some $\phi\in S$ and $f\in M_S$, then it must be the case that $\phi$ has an $m$th powered periodic point (up to a small multiple) for some $m\geq2$. However, for sufficiently large degrees $d$, all periodic points for $x^d + c$ are fixed points, outside of a set of $c$ of bounded height. With this in mind, we make the following definition.  
\begin{defin}\label{def:poweredfixedpoint} Let $K$ be a field of characteristic $p\geq0$, let $d\geq2$, and let $c\in K$. Then we say that $\phi(x)=x^d+c$ has a \emph{powered fixed point} if there exist $y\in K$, $m\geq2$, and $r\in\{\pm{1},\pm{4}\}$ such that $\phi(ry^m)=ry^m$. Equivalently, $\phi$ has a powered fixed point if it is of the form $\phi(x)=x^d+ry^m-(ry^m)^d$ for the stipulated $y$, $m$, and $r$. 
\end{defin}
We immediately deduce the following irreducibility result for semigroups containing an irreducible map with no powered fixed points. 
\begin{prop}\label{prop:nopoweredfixedpoints} Let $K$ be an $abc$-field of characteristic zero and let $S=\{x^{d_1}+c_1,\dots x^{d_s}+c_s\}$ for some $c_1,\dots,c_s\in K$ and some $d_1,\dots,d_s\geq2$. Moreover, assume that $S$ contains $\phi(x)=x^d+c$ that is irreducible and has no powered fixed points. Then there exist constants $C_5=C_5(K)$, $D_5=D_5(K)$ and $N=N(K,d)$ such that if $h(c)>C_5$ and $d>D_5$, then  
\[\{\phi^N\circ f\,:\, f\in M_S\}\]
is a set of irreducible polynomials in $K[x]$. Moreover, one can take $C_5=0$ and $D_5=\max\{14, 4g_K+10\}$ in the function field case.  
\end{prop}
\begin{proof} Assume that $\phi(x)=x^d+c\in S$ is irreducible and without a powered fixed point. Let $C_5=\max\{C_1,C_4\}$ and $D_5=\max\{D_1,D_4\}$ where $C_1$ and $D_1$ are the constants from Theorem \ref{thm:uniform-bound} and $C_4$ and $D_4$ are the constants from Proposition \ref{prop:uniformstability}. Moreover, let $N = N(K,d)$ be as in Proposition \ref{prop:uniformpower} and assume that $h(c)>C_5$ and $d>D_5$. Then Proposition \ref{prop:uniformstability} implies that $\phi^N$ is irreducible over $K$. Now suppose for a contradiction that $\phi^N\circ f$ is reducible for some $f\in M_S$. Since $f$ is a composition of maps of the form $x^{d_i} + c_i$, repeated application of Theorem \ref{prop:oldstability} implies that $\phi^N(a)=ry^m$ for some $a,y\in K$, some $r\in\{\pm{1},\pm{4}\}$, and some $m\geq2$. In particular, it follows from Proposition \ref{prop:uniformpower} that $ry^m$ is a periodic point for $\phi$. But then $ry^m$ must be a fixed point for $\phi$ by Theorem \ref{thm:uniform-bound}, and we reach a contradiction of our assumption that $\phi$ has no powered fixed points.          
\end{proof}
Thus we have succeeded in constructing a large set of irreducible polynomials in $M_S$ unless \emph{every irreducible polynomial in $S$ (of sufficiently large height) has a powered fixed point}. To handle the case when $S$ contains at least two such maps, we use the following result:
\begin{lem}\label{lem:two+special+overlap} Let $K$ be an $abc$-field of characteristic zero, let $a,b,y,z\in K$, let $d_1,d_2\geq2$ and let $\zeta_i\in\mu_{K,d_i}$ for $i=1,2$. Moreover, assume that 
\begin{equation}\label{eq1:two-irre-overlap}
    a^{d_1}+y-y^{d_1}=\zeta_2z\;\;\;\,\text{and}\,\;\;\; b^{d_2}+z-z^{d_2}=\zeta_1y.\vspace{.1cm}  
\end{equation}
Then there exists $D_6=D_6(K)$, such that if $\min\{d_1,d_2\}>D_6$ and $\min\{h(y),h(z)\}>0$, then $z=\zeta_1 y$. In particular, we can take $D_6=\max\{13,2g_K+11\}$ in the function field case.  
\end{lem}
\begin{proof}
Assume that $\min\{h(y),h(z)\}>0$ and that $z\neq \zeta_1 y$. Now suppose first that also $y\neq \zeta_2z$. Then in particular, we may define $c_1:=(\zeta_2z-y)^{-1}$ and $c_2:=(\zeta_1y-z)^{-1}$ so that 
\begin{equation}\label{overlap}
c_1a^{d_1}-c_1y^{d_1}=1\;\;\;\text{and}\;\;\;c_2a^{d_2}-c_2z^{d_2}=1.
\end{equation}
But then Proposition \ref{prop:fermat-catalan} and the fact that $h(\pm{c_i})\leq h(y)+h(z)+\log2$ together imply 
\begin{equation*}
\begin{split} 
\max\{d_1h(y),d_2h(z)\}\leq B_1'(h(y)+h(z))+B_2'
\end{split} 
\end{equation*}
for some constants $B_1'$ and $B_2'$, with $B_2'$ depending on $K$. In particular, we deduce that  
\begin{equation}\label{eq:overlap+bd}
\begin{split}
\min\{d_1,d_2\}(h(y)+h(z))\leq d_1h(y)+d_2h(z)\leq 2B_1'(h(y)+h(z))+ 2B_2'. 
\end{split}
\end{equation}
On the other hand, the minimal nonzero height $h^{\min}_K$ of an element of $K$ is strictly positive; see also \eqref{eq:heightmin}. Hence, \eqref{eq:overlap+bd} implies that $\min\{d_1,d_2\}\leq 2B_1'+B_2'/h^{\min}_K$. Therefore, for $\min\{d_1,d_2\}\gg_K0$ it must be the case that $y=\zeta_2z$. But then $ b^{d_2}+(\zeta_2^{-1}y)-y^{d_2}=\zeta_1y$. Hence, if $\zeta_1-\zeta_2^{-1}\neq0$, then we see that
\[\Big(\frac{1}{(\zeta_1-\zeta_2^{-1})y}\Big)b^{d_2}+\Big(\frac{-1}{(\zeta_1-\zeta_2^{-1})y}\Big)y^{d_2}=1.\]
Therefore, Proposition \ref{prop:fermat-catalan} implies that $d_2h(y)\leq B_1'h(y)+B_2'$, and thus $\min\{d_1,d_2\}\leq B_1'+B_2'/h^{\min}_K$. In particular, for $\min\{d_1,d_2\}\gg_K0$ we have that $\zeta_2=\zeta_1^{-1}$. But then, since we already know that $y=\zeta_2z$, we deduce that $z=\zeta_1y$ as claimed.        
\end{proof} 
\begin{remark} As before, a simpler and more explicit argument may be given in the function field case. Namely, \cite[Theorem 2]{MR0664038} and \eqref{overlap} together imply that \[\big(\min\{d_1,d_2\}-3\big)\max\{h(y),h(z)\}\leq 10\max\{h(y),h(z)\}+2g_K-2.\]
In particular, we reach a contradiction for all $\min\{d_1,d_2\}>\max\{13,2g_K+11\}$ as claimed. 
\end{remark}
As a consequence, we deduce the following irreducibility result for semigroups containing two irreducible maps with powered fixed points; compare to \cite[Theorem 1.1]{Mathworks1}.
\begin{prop}\label{prop:2specialirred,unrelated} Let $K$ be an $abc$-field of characteristic zero, let $S=\{x^{d_1}+c_1,\dots x^{d_s}+c_s\}$ for some $c_1,\dots,c_s\in K$ and some $d_1,\dots,d_s\geq2$, and suppose that $S$ contains irreducible polynomials $\phi_1=x^{d_1}+c_1$ and $\phi_2=x^{d_2}+c_2$ with powered fixed points $P_1$ and $P_2$, respectively. Moreover, assume that $P_1/P_2$ is a not in $\mu_{K,d_1}$. Then there exist constants $C_7=C_7(K)$, $D_7=D_7(K)$ and $N_i=N(K,d_i)$ such that if $\min\{h(c_1),h(c_2)\}>C_7$ and $\min\{d_1,d_2\}>D_7$, then either  
\[\{\phi_1^{N_1}\circ \phi_2\circ f\;:\; f\in M_S\}\;\;\;\;\text{or}\;\;\;\; \{\phi_2^{N_2}\circ \phi_1\circ f\;:\; f\in M_S\}\]
is a set of irreducible polynomials in $K[x]$. In particular, one can take $C_7=0$ and $D_7=\max\{14, 4g_K+10\}$ in the function field case.  
\end{prop}
\begin{proof} Let $\phi_1,\phi_2\in S$ be irreducible over $K$ with powered fixed points $P_1,P_2\in K$ respectively and let $C_7=\max\{C_1,C_4,\log2\}$ and $D_7=\max\{D_1,D_4,D_6\}$; here $C_1$ and $D_1$ are the constants from Theorem \ref{thm:uniform-bound}, $C_4$ and $D_4$ are the constants from Proposition \ref{prop:uniformstability}, and $D_6$ is the constant from Lemma \ref{lem:two+special+overlap}. Moreover, let $N_i=N(K,d_i)$ be as in Proposition \ref{prop:uniformpower}, and assume that $\min\{h(c_1),h(c_2)\}>C_7$ and $\min\{d_1,d_2\}>D_7$. Then $\phi_1^{N_1}$ and $\phi_2^{N_2}$ are irreducible by Proposition \ref{prop:uniformstability}. Hence, if there exist polynomials of the form $\phi_1^{N_1}\circ\phi_2\circ f_1$ and $\phi_2^{N_2}\circ\phi_1\circ f_2$ for $f_1,f_2\in M_S\cup\{\text{Id}\}$ that are \emph{both} reducible, then repeated application of Theorem  \ref{prop:oldstability} implies that there exist some $a_1,a_2,b_1,b_2\in K$, some $r_1,r_2\in\{\pm{1},\pm{4}\}$, and  some $m_1,m_2\geq2$ such that $\phi_2^{N_2}(\phi_1(a_1))=r_1a_2^{m_1}$ and $\phi_1^{N_1}(\phi_2(b_1))=r_2b_2^{m_2}$. But then Proposition \ref{prop:uniformpower} implies that $\phi_1(a_1)\in \PrePer(\phi_2,K)$ and that $\phi_2(b_1)\in \PrePer(\phi_1,K)$. On the other hand, if follows from Theorem \ref{thm:uniform-bound} that 
\[\PrePer(\phi_1,K)=\big\{\zeta P_1\,:\, \zeta\in\mu_{K,d_1}\big\}
\;\;\; \text{and}\;\;\;
\PrePer(\phi_2,K)=\big\{\zeta P_2\,:\, \zeta\in\mu_{K,d_2}\big\}.\]
Hence, since $c_1=P_1-P_1^{d_1}$ and $c_2=P_2-P_2^{d_2}$ by assumption, it must be the case that 
\[a_1^{d_1}+P_1-P_1^{d_1}=\phi_1(a_1)=\zeta_2 P_2\;\;\;\text{and}\;\;\; b_1^{d_2}+P_2-P_2^{d_2}=\phi_2(b_1)=\zeta_1 P_1\] 
for some $\zeta_i\in \mu_{K,d_i}$. But then Lemma \ref{lem:two+special+overlap} implies that either $h(P_1)=0$ or $h(P_2)=0$ or $P_2=\zeta_1 P_1$. However, in the former case we have that 
\[h(c_i)=h(P_i-P_i^{d_i})\leq h(P_i)+d_ih(P_i)+\log2\leq \log2\]
for some $i$, a contradiction. Therefore, $P_1/P_2\in \mu_{K,d_1}$ as claimed.    
\end{proof}
Lastly, we handle the case when $S$ contains an irreducible map with a powered fixed point and another reducible polynomial. To do this, we need the following lemma; see \eqref{eq:r's} for the definition of the set $R_{K}$.  
\begin{lem}\label{lem:special+reducible} Let $K$ be an $abc$-field of characteristic zero and let $P,z\in K$ satisfy 
\begin{equation}\label{eq:lem:special+reducible}
rz^{m}=\zeta_1 P-\zeta_2 P^d
\end{equation} 
for some $r\in R_{K}$, some $\zeta_1,\zeta_2\in \mu_{K}$, and some $m\geq2$. Then there exists a constant $D_8(K)$ such that if $d\geq D_8$, then $h(P)=0$. Moreover, $D_8=\max\{14, 4g_K+10\}$ suffices in the function field case.  
\end{lem}
\begin{proof} Assuming that $P$ is non-zero we may rewrite \eqref{eq:lem:special+reducible} as \[\Big(\frac{\zeta_2}{\zeta_1 P}\Big)P^{d}+\Big(\frac{r}{\zeta_1 P}\Big)z^{m}=1.\] 
But then Proposition \ref{prop:fermat-catalan} and the fact that $R_K$ is a set of bounded height together imply that $dh(P)\leq B_3' h(P)+B_4'$ for some constants $B_3'$ and $B_4'$ depending on $K$; the result follows.   
\end{proof}
In particular, we have the following irreducibility result for semigroups containing an irreducible map with a powered fixed point and a second reducible polynomial.  
\begin{prop}\label{prop:oneirred+onered} Let $K$ be an $abc$-field of characteristic zero and let $S=\{x^{d_1}+c_1,\dots x^{d_s}+c_s\}$ for some $c_1,\dots,c_s\in K$ and $d_1,\dots,d_s\geq2$, and suppose that $S$ contains an irreducible map $\phi_1(x)=x^{d_1}+c_1$ with a powered fixed point $P$ (in the sense of Definition \ref{def:poweredfixedpoint}) and a reducible map $\phi_2(x)=x^{d_2}+c_2$. Moreover, assume that $P/c_2$ is not in $\mu_{K,d_1}$. Then there exist constants $C_9=C_9(K)$, $D_9=D_9(K)$ and $N_i=N(K,d_i)$ such that if $\min\{h(c_1),h(c_2)\}>C_9$ and $\min\{d_1,d_2\}>D_9$, then
\[\{\phi_1^{N_1}\circ \phi_2^{N_2}\circ f\;:\; f\in M_S\}\]
is a set of irreducible polynomials in $K[x]$. In particular, one can take $C_9=0$ and $D_9=\max\{14, 4g_K+10\}$ in the function field case.     
\end{prop}
\begin{proof} 
Let $\phi_1(x)=x^{d_1}+c_1$, $\phi_2(x)=x^{d_2}+c_2$, and $P$ be as above and let $N_i=N(K,d_i)$ be as in Proposition \ref{prop:uniformpower}. Moreover, assume that $P/c_2$ is not a $d$th root of unity in $K$, that $\min\{h(c_1),h(c_2)\}>\max\{C_1,C_4,\log2\}$, and that $\min\{d_1,d_2\}>\max\{D_1,D_4,D_8\}$; here $C_1$ and $D_1$ are the constants from Theorem \ref{thm:uniform-bound}, $C_4$ and $D_4$ are the constants from Proposition \ref{prop:uniformstability}, and $D_8$ is the constant in Lemma \ref{lem:special+reducible}. From here, we note first that
\begin{equation}\label{eq:special+red1}
P=r_1y^{m_1} \;\;\;\text{and}\;\;\;\; c_2=r_2z^{m_2}
\end{equation}
for some $y,z\in K$, some divisors $m_1,m_2\geq2$ of $d_1$ and $d_2$ respectively, and some $r_1,r_2\in R_{K}$ defined \eqref{eq:r's}; here we use that $P$ is a powered fixed point (see Definition \ref{def:poweredfixedpoint}) and that $\phi_2$ is reducible so that $\phi_2(0)=r_2z^{m_2}$ by Theorem \ref{prop:oldstability}. Moreover, $\phi_1$ is stable by Proposition \ref{prop:uniformstability}.

From here, we show that $\phi_1^{N_1}\circ\phi_2^{N_2}$ is irreducible in $K$. If not, then repeated application of Theorem \ref{prop:oldstability} implies that $\phi_1^{N_1}(\phi_2^t(0))=r_3w^{m_3}$ for some $w\in K$, some $r_3\in R_{K}$, some divisor $m_3\geq 2$ of $d_2$, and some iterate $1\leq t\leq N_2$. But then Proposition \ref{prop:uniformpower} implies that $a=\phi_2^t(0)$ must be preperiodic for $\phi_1$. On the other hand, Theorem \ref{thm:uniform-bound} implies that 
\begin{equation}\label{eq:special+red2}
\PrePer(\phi_1,K)=\{\zeta P:\zeta\in\mu_{K,d_1}\},
\end{equation}
so we can write $\phi_2^t(0)=\zeta_1 P$ for some $\zeta_1\in\mu_{K,d_1}$. Next set $r_4=\zeta_1 r_1$ and note that $r_4\in R_{K}$ since $r_1\in R_{K}$ and $R_{K}$ is closed under multiplying by roots of unity in $K$. Then, by combining \eqref{eq:special+red1} with \eqref{eq:special+red2}, we see that
\[\phi_2^t(0)=\zeta_1 P=\zeta_1 r_1y^{m_1}=r_4 y^{m_1}.\]
Hence, it follows from Remark \ref{rem:criticalorbitpowers} (see also the proof of Proposition \ref{prop:uniformstability}) that $t=1$. But then by construction $\zeta_1 P=\phi_2^t(0)=\phi_2(0)=c_2$, so that $P/c_2$ is a $d_1$-th root of unity, a contradiction. In particular, we must have that $\phi_1^{N_1}\circ \phi_2^{N_2}$ is irreducible over $K$.

Similarly, if $\phi_1^{N_1}\circ \phi_2^{N_2}\circ f$ is reducible over $K$ for some $f\in M_S$, then repeated application of Theorem \ref{prop:oldstability} applied to the map $g=\phi_1^{N_1}\circ \phi_2^{N_2}$ implies that 
\[r_5u^{m_4}=\phi_1^{N_1}(\phi_2^{N_2}(b))\] 
for some $r_5\in R_{K}$, some $b,u\in K$, and some $m_4\geq2$. Thus Proposition \ref{prop:uniformpower}, this time applied to the map $\phi_1$ and the point $a=\phi_2^{N_2}(b)$, implies that $\phi_2^{N_2}(b)$ must be preperiodic for $\phi_1$. But this means 
\[\phi_2^{N_2}(b)=\zeta_2 P=\zeta_2 r_1 y^{m_1}\]
for some $\zeta_2\in\mu_{K,d_1}$ by \eqref{eq:special+red1} and \eqref{eq:special+red2}. In particular, setting $r_6=\zeta_2 r_1\in R_{K}$, we deduce that 
\[\phi_2^{N_2}(b)=r_6y^{m_1}.\]
But then Proposition \ref{prop:uniformpower}, now applied to the map $\phi_2$ and the basepoint $b$, implies that $r_6y^{m_1}$ must be a periodic point for $\phi_2$. Hence, the point $r_6y^{m_1}=\zeta_2 r_1 y^{m_1}=\zeta_2 P$ is fixed by $\phi_2$ by Theorem \ref{thm:uniform-bound}. However, from here we see that 
\[r_2z^{m_2}=c_2=(\zeta_2 P)-(\zeta_2 P)^{d_2}=\zeta_2 P-\zeta_2^{d_2} P^{d_2}=\zeta_2 P-\zeta_3 P^{d_2}\]
for some $\zeta_3\in\mu_{K,d_1}$. In particular, we obtain a solution to a diophantine equation in Lemma \ref{lem:special+reducible}, so that $h(P)=0$. But then 
\[h(c_1)=h(P-P^{d_1})\leq h(P)+d_1h(P)+\log2=\log2,\]
and we reach a contradiction. The claim follows. 
\end{proof}
We now have the tools in place to prove our main irreducibility result from the introduction. 
\begin{proof}[Proof of Theorem \ref{thm:main+irred}] 
Let $C_3(K)$ and $D_3(K)$ denote the maximum of the constants $C_i(K)$ and $D_i(K)$ in this section respectively (so that all of the results in this section hold). Next decompose $S=S_1\cup S_2\cup S_3$ into the following pieces:\vspace{.1cm}  
\begin{equation}\label{eq:S+decomp}
\begin{split}
S_1&:=\{\phi(x)=x^d+c\in S\,:\, h(c)\leq C_3\;\text{or}\; d\leq D_3\},\\[5pt] 
S_2&:=\{\phi(x)=x^d+c\in S\,:\, h(c)>C_3,\; d>D_3, \;\text{and}\; \text{$\phi$ is irreducible over $K$}\},\\[5pt] 
S_3&:=\{\phi(x)=x^d+c\in S\,:\, h(c)>C_3,\; d>D_3,\;\text{and}\; \text{$\phi$ is reducible over $K$}\}.\\[2pt] 
\end{split}
\end{equation}
Now suppose that there is some $\phi(x)=x^{d_1}+c_1\in S_2$. We first show that there is some $g\in M_S$ such that $\{g\circ f: f\in M_S\}$ is a set of irreducible polynomials over $K$, unless $S$ is of a special form. Thus, assume no such $g$ exists. Proposition \ref{prop:nopoweredfixedpoints} implies that \emph{every map} in $S_2$ has a powered fixed point. Let $P$ denote the powered fixed point of $\phi$. From here, Proposition \ref{prop:2specialirred,unrelated} applied separately to the pairs $(\phi,\psi)$ for each $\psi\in S_2$ implies that
\begin{equation}\label{eq:irred-decomp}
S_2\subseteq \Big\{x^d+(\zeta P)-(\zeta P)^{d}\,:\, \zeta\in\mu_{K,d_1},\; d\in\{d_1,\dots,d_s\}\Big\}.
\end{equation}
Likewise, Proposition \ref{prop:oneirred+onered} applied separately to the pairs $(\phi,\psi)$ for each $\psi\in S_3$ implies that    
\begin{equation}\label{eq:red-decomp}  
S_3\subseteq \Big\{x^d+\zeta P\,:\, \zeta\in\mu_{K,d_1},\; d\in\{d_1,\dots,d_s\}\Big\}.
\end{equation}
In particular, if there is no $g\in M_S$ such that $\{g\circ f: f\in M_S\}$ is a set of irreducible polynomials over $K$, then $S$ is necessarily of the special form 
\begin{equation}\label{eq:total-decomp}
S\subseteq S_1\cup \Big\{x^d+(\zeta_1 P)-(\zeta_1 P)^d,\; x^d+\zeta_2 P\;:\;\text{$\zeta_1,\zeta_2\in\mu_{K,d_1}$\;\text{and}\;$d\in\{d_1,\dots,d_s\}$}\Big\}. 
\vspace{.15cm}
\end{equation} 
Moreover, since $P$ is a powered fixed point, we have that $P=ry^m$ for some $y\in K$, some $r\in\{\pm{1},\pm{4}\}$, and some $m\geq2$; see Definition \ref{def:poweredfixedpoint}. Equivalently, 
\begin{equation}\label{eq:specialform}
S\subseteq \mathcal{S}(C_3,D_3)\cup\mathcal{I}(ry^m,d_1,\mathfrak{d})\cup \mathcal{R}(ry^m,d_1,\mathfrak{d})
\end{equation} 
for some such $r$, $y$, and $m$ and with $\mathfrak{d}:=\{d_1,\dots,d_s\}$; see \eqref{eq:smallhtordeg}, \eqref{eq:oneparameter+irre}, and \eqref{eq:oneparameter+red} for the definitions of these sets. 

Finally, we prove the claim about the lower bound on the proportion of irreducible polynomials in $M_S$ when $S$ is not of the form in \eqref{eq:specialform}. First we note that $M_S$ is a free semigroup by \cite[Theorem 3.1]{MR4349782}. Next, we note that if $|S|=1$, then $M_S=\{\phi^n\}_{n\geq1}$ consists entirely of irreducible polynomials by Proposition \ref{prop:uniformstability}. Therefore, we may assume that $|S|\geq2$. In this case, let $\rho$ be the unique positive real number satisfying $(1/d_1)^{\rho}+\dots+ (1/d_s)^{\rho}=1$. Then \cite[Proposition 2.2]{bell2023counting} implies that there are positive constants $b_1$ and $b_2$ depending only on $d_1,\dots,d_s$ such that
\begin{equation}\label{eq:free+count}
b_1B^\rho\leq\big|\{F\in M_S\,: \deg(F)\leq B\}\big|\leq b_2 B^{\rho} 
\end{equation}
for all $B$ sufficiently large. In particular, when $S$ is not of the form in \eqref{eq:specialform}, it follows from \eqref{eq:free+count} that
\begin{equation*}
\begin{split}
\scalemath{.95}{
\frac{\big|\{F\in M_S\,: \deg(F)\leq B\;\text{and $F$ is irreducible over $K$} \}\big|}{\big|\{F\in M_S\,: \deg(F)\leq B\}\big|}}\,&
\scalemath{.95}{
\geq \frac{\big|\{g\circ f\, : \, f\in M_S\; \text{and}\; \deg(g\circ f)\leq B\}\big|}{\big|\{F\in M_S\,: \deg(F)\leq B\}\big|}}\\[8pt] 
&=\frac{\big|\{f\in M_S\,:\, \deg(f)\leq \frac{B}{\deg(g)}\}\big|}{\big|\{F\in M_S\,: \deg(F)\leq B\}\big|}\\[8pt] 
&\geq \frac{b_1}{b_2\deg(g)^{\rho}} 
\end{split} 
\end{equation*}
for all large $B$. Therefore, $M_S$ contains a positive proportion of irreducible polynomials over $K$ as claimed. Moreover, since $g$ may be taken to be of the form $g=\phi_i^{N_i}$, or $g=\phi_i^{N_i}\circ \phi_j$, or $g=\phi_i^{N_i}\circ \phi_j^{N_j}$ for some $\phi_i,\phi_j\in S$ and some iterates $N_i=N(K,d_i)$ and $N_j=N(K,d_j)$ depending only on $K$ and $d_i=\deg(\phi_i)$ and $d_j=\deg(\phi_j)$, it follows that there is a positive lower bound on the proportion of irreducibles in $M_S$ that depends only on $K$ and $d_1,\dots,d_s$. 
\end{proof}
\begin{remark} In fact, by symmetry in the proof of Lemma \ref{lem:two+special+overlap}, it follows that the condition on the ratio of the powered fixed points $P_1$ and $P_2$ in Proposition \ref{prop:2specialirred,unrelated} can be strengthened: it is enough to assume that $P_1/P_2$ is not an element of $\mu_{K,d_1}\cap\mu_{K,d_2}$. In particular, the parameter $\zeta$ in the definition of \eqref{eq:oneparameter+irre} can be assumed to come from $\mu_{K,n}\cap\mu_{K,d}$ instead of just $\mu_{K,n}$.      
\end{remark}
In particular, when $K/\mathbb{Q}$ is a number field (so that $\mu_K$ is finite), Theorem \ref{thm:main+irred} implies that $M_S$ contains a positive proportion of irreducibles whenever the polynomials in the generating set $S$ have sufficiently large degree and height, the cardinality of $S$ is large enough, and $S$ contains at least one irreducible polynomial.    
\begin{proof}[Proof of Corollary \ref{cor:numberfield}] Suppose that $\min_i\{h(c_i)\}>C_3$ and that $\min_i\{d_i\}>D_3$ where $C_3$ and $D_3$ are as in Theorem \ref{thm:main+irred}. Moreover, assume that there is some irreducible $\phi(x)=x^{d_1}+c_1\in S$. In particular, $\phi$ is necessarily in the set $S_2$ defined in \eqref{eq:irred-decomp}. Hence, if $M_S$ does not contain a positive proportion of irreducible polynomials over $K$, the proof of Theorem \ref{thm:main+irred} implies that $S$ takes the special form in \eqref{eq:total-decomp}. On the other hand, $S_1=\varnothing$ by assumption, so that  
\begin{equation}\label{eq:cor+decomp}
S\subseteq \Big\{x^d+(\zeta_1 P)-(\zeta_1 P)^d,\; x^d+\zeta_2 P\;:\;\text{$\zeta_1,\zeta_2\in\mu_{K,d_1}$\;\text{and}\; $d\in\{d_1,\dots,d_s\}$}\Big\}.
\end{equation} 
However, for each fixed $d=d_i$ there are at most $2|\mu_{K,d_1}|\leq 2|\mu_K|$ such elements in the set above. Hence, if $S$ contains more than $2|\mu_K|\cdot \big|\{d_1,\dots,d_s\}\big|$ polynomials, then we reach a contradiction; thus $M_S$ contains a positive proportion of irreducibles over $K$ in this case. 

Similarly, suppose that there is a place $v$ on $K$ and some indices $i$ and $j$ such that $v(c_i)<0$ and $v(c_j)\geq0$. Since we know that either $c_i=\zeta P-(\zeta P)^d$ or $c_i=\zeta P$ for some $\zeta\in\mu_{K}$ and some $d\geq 2$, it follows that $v(P)<0$. On the other hand, $c_j$ is also of the form $c_j=\zeta' P-(\zeta' P)^{d'}$ or $c_j=\zeta' P$ for some $\zeta'\in\mu_{K}$ and some $d'\geq2$. But in the latter case, 
\[0\leq v(c_j)=v(\zeta' P)=v(\zeta')+v(P)=v(P)<0,\] 
and we reach a contradiction. Likewise, since the valuations $v(\zeta' P)=v(P)$ and $v((\zeta' P)^{d'})=d'v(P)$ are distinct, we deduce that
\[v(c_j)=v\Big(\zeta' P-(\zeta' P)^{d'}\Big)=\min\Big\{v(\zeta' P),\; v\big((\zeta' P)^{d'}\big)\Big\}=d'v(P)<0,\]
in the first case, and we again reach a contradiction. Therefore, $S$ cannot be of the special form in \eqref{eq:cor+decomp}; thus $M_S$ contains a positive proportion of irreducibles in this case as well.   
\end{proof}
We are able to prove more in the function field setting. As an example, we consider the case when every map in $S$ is of the excluded special form \eqref{eq:total-decomp} and prove the following result:  
\begin{prop}\label{prop:allpthpowerfixedpts+related} 
Let $K/k(t)$ be a function field of characteristic zero, and let 
\[
S\subseteq\big\{x^d+(\zeta P)-(\zeta P)^d\,:\, \zeta\in\mu_{K}\;\text{and}\; d\geq 2\big\}
\]
for some non-constant $P\in K$. Moreover, assume that $\phi(x)=x^{d}+P-P^{d}$ is in $S$, that $\phi$ is irreducible over $K$, and that $d>\max\{14, 4g_K+10\}$. Then there exists $N=N(K,d)$ such that $\{\phi^N\circ f : f\in M_S\}$ is a set of irreducible polynomials in $K[x]$.   
\end{prop}
\begin{proof} Let $S$ and $\phi$ be as in the statement above and let $N$ be as in Proposition \ref{prop:uniformpower}. Then $\phi^N$ is irreducible by Proposition \ref{prop:uniformstability}. Now let $F\in M_S$ be an arbitrary non-identity element so that $F=\theta_{1}\circ\theta_{2}\circ\dots\circ\theta_{n}$ for some $\theta_i\in S$. In particular, we may write $\theta_{i}(x)=x^{d_i}+(\zeta_iP)-(\zeta_iP)^{d_i}$ for some $\zeta_i\in\mu_{K}$ and $d_i\geq2$. 
A quick induction argument shows that
\begin{equation}\label{eq:F(0)}
F(0)= \zeta_F P^{d_1\cdots d_n}+G_F(P) \vspace{.1cm}
\end{equation}
for some $\zeta_F\in\mu_K$ and some polynomial $G_F\in k[x]$ of degree strictly less than $d_1\cdots d_n$. 

On the other hand, suppose that $\phi^N\circ f$ is reducible in $K[x]$ for some $f\in M_S$. Then repeated application of Theorem \ref{prop:oldstability} implies that there exists some $F\in M_S$ (not necessarily the same as $f$), some $y\in K$, some $r\in\{\pm{1},\pm{4}\}$, and some $m\geq2$ such that $\phi^N(F(0))=ry^m$. But then Proposition \ref{prop:uniformpower} and Theorem \ref{thm:uniform-bound} together imply that $F(0)=\zeta' P$ for some $\zeta'\in \mu_{K}$. In particular, \eqref{eq:F(0)} implies that $\zeta_F P^{d_1\cdots d_n}+G_F(P)-\zeta' P=0$, and we obtain a polynomial with coefficients in $k$ of degree $d_1\cdots d_n$ for some $n\geq1$ that vanishes at $P$, which contradicts our assumption that $P$ is non-constant. The claim follows.     
\end{proof}
We may now deduce from Theorem \ref{thm:main+irred} and Proposition \ref{prop:allpthpowerfixedpts+related} several simple conditions for ensuring that $M_S$ contains a positive proportion of irreducible polynomials in the function field setting.
\begin{cor}\label{cor:functionfield} Let $K/k(t)$ be a function field of characteristic zero, and let $S=\{x^{d_1}+c_1,\dots, x^{d_s}+c_s\}$ for some non-constant $c_1,\dots,c_s\in K$. Moreover, assume that $\min_i\{d_1,\dots,d_s\}>\max\{14,4g_K+10\}$, that $S$ contains $\phi(x)=x^{d_1}+c_1$ that is irreducible over $K$, and that at least one of the following additional conditions is satisfied: \vspace{.1cm} 
 \begin{enumerate}
    \item[\textup{(1)}] There is a place $v$ of $K$ and indices $i$ and $j$ such that $v(c_i)<0$ and $v(c_j)\geq0$,\vspace{.2cm} 
    \item[\textup{(2)}] The number of polynomials in $S$ is greater than $2\,\big|\mu_{K,d_1}\big|\cdot \big|\{d_1,\dots,d_s\}\big|$,\vspace{.2cm} 
    \item[\textup{(3)}] The number of irreducible polynomials in $S$ is greater than $\big|\mu_{K,d_1}\big|\cdot \big|\{d_1,\dots,d_s\}\big|$, \vspace{.2cm} 
    \item[\textup{(4)}] The number of reducible polynomials in $S$ is greater than $\big|\mu_{K,d_1}\big|\cdot \big|\{d_1,\dots,d_s\}\big|$, \vspace{.2cm}  
\item[\textup{(5)}] The extension $K(c_1,\dots,c_s)/K$ is not simple, \vspace{.2cm}
\item[\textup{(6)}] Every polynomial in $S$ is irreducible. \vspace{.2cm} 
\end{enumerate}
Then $M_S$ contains a positive proportion of irreducible polynomials over $K$.       
 \end{cor} 
\begin{proof} Statement (1) has the same proof as in the number field case; see statement (2) of Corollary \ref{cor:numberfield}. From here, let $K/k(t)$ be a function field and let $S=\{x^{d_1}+c_1,\dots, x^{d_s}+c_s\}$ for some non-constant $c_1,\dots,c_s\in K$ and some $\min\{d_1,\dots,d_s\}>\max\{14,4g_K+10\}$. Moreover, assume that at least one polynomial $\phi(x)=x^{d_1}+c_1$ in $S$ is irreducible over $K$ and let $S_2$ and $S_3$ denote the irreducible and reducible polynomials in $S$ respectively; note that $C(K)=0$ in this case, so that $S_1$ is empty. Then, if $M_S$ fails to contain a positive proportion of irreducible polynomials, the proof of Theorem \ref{thm:main+irred} above implies that $S_2$ and $S_3$ are of the form in \eqref{eq:irred-decomp} and \eqref{eq:red-decomp}. In particular, 
\[\max\{\big|S_2\big|, \big|S_3\big|\}\leq \big|\mu_{K,d_1}\big|\cdot\big|\{d_1,\dots,d_s\}\big|,\]
and we obtain statements (2), (3) and (4) of the corollary. Similarly, the field extension $K(c_1,\dots,c_s)/k$ is simple since $K(c_1,\dots,c_s)\subseteq k(P)$ and so L\"{u}roth's theorem implies that $K(c_1,\dots,c_s)=k(u)$ for some non-constant $u\in k(P)$. Thus, we obtain statement (5). Finally, it must also be the case that $S_3\neq\varnothing$, for otherwise $S$ is of the form in Proposition \ref{prop:allpthpowerfixedpts+related}. Moreover, $M_S$ contains a positive proportion of irreducibles in this case. In particular, statement (6) of the corollary follows.           
\end{proof}
\begin{remark} Assuming that $\min_i\{h(c_i)\}\gg_K0$ and that $\min_i\{d_i\}\gg_K0$, statements (2), (3) and (4) of Corollary \ref{cor:functionfield} also hold over number fields. Moreover, these statements are an improvement over statement (1) of Corollary \ref{cor:numberfield} from the introduction.  
\end{remark}

\section{Dynamical Galois groups}\label{sec:Galois}
We begin this section with notation related to Galois groups generated dynamically. Let $K$ be a field, let $S$ be a fixed set of polynomials over $K$, and let $\gamma=(\theta_1,\theta_2,\dots)$ be an infinite sequence of elements $\theta_i\in S$. Then we are interested in the tower of field extensions \[K_n(\gamma):=K(\theta_1\circ\theta_2\circ\dots\circ\theta_n),\] 
where $K(f)$ denotes the splitting field of $f\in K[x]$ in a fixed algebraic closure $\overline{K}$. Since the aforementioned extensions are nested, $K\subseteq K_{1}(\gamma)\subseteq K_{2}(\gamma)\subseteq\dots K_{n}(\gamma)$, we may form the inverse limit of associated Galois groups, 
\[G_{K}(\gamma):=\lim_{\longleftarrow}\Gal(K_{n}(\gamma)/K),\]
called the \emph{dynamical Galois group} associated to $\gamma$. Based on the analogy with constant sequences \cite{cubic:abc,Ferraguti:quad,Riccati, Jones}, we expect that the group $G_K(\gamma)$ is quite large for ``many'' (or perhaps ``most'') sequences $\gamma$. To make this precise, we fix a probability measure $\nu$ on $S$ and let $\bar{\nu}=\nu^{\mathbb{N}}$ be the product measure on $\Phi_S=S^{\mathbb{N}}$, the set of all infinite sequences of elements of $S$. We say that a property $\mathcal P$ holds for ``many'' sequences in $S$ if it holds with positive probability: \[\bar{\nu}\big(\{\gamma\in\Phi_S\,:\, \text{$\gamma$ has property $\mathcal P$}\}\big)>0.\] 
As a test case, we study the Galois groups generated by sequences of unicritical polynomials and prove that many such sequences produce big dynamical Galois groups. In what follows, we say that a sequence $\gamma=(\theta_n)_{n=1}^\infty\in \Phi_S$ is \emph{stable} over $K$, if  $\gamma_n:=\theta_1\circ\dots\circ\theta_n$ is irreducible over $K$ for every $n\geq1$. Moreover, when $S=\{x^d+c_1,\dots,x^d+c_s\}$ is a set of unicritical polynomials (of the \emph{same degree}) and $\gamma\in\Phi_S$, then we say that the subextension $K_n(\gamma)/K_{n-1}(\gamma)$ is \emph{maximal} if $[K(\gamma_n):K(\gamma_{n-1})]=d^{d^{n-1}}$; see Remark \ref{rem:maximal} for a justification of this terminology. Following \cite[\S1]{MR4680482}, we say that $\gamma \in \Phi_S$ furnishes a \emph{big dynamical Galois group} over $K$ if $\gamma$ is stable over $K$ and $K_n(\gamma)/K_{n-1}(\gamma)$ is maximal infinitely often.

As an application of our irreducibility theorem, we prove the following result; in what follows $[C_d]^\infty$ denotes the infinite, iterated wreath products of the cyclic group $C_d$ of order $d$, a natural overgroup for the dynamical Galois groups $G_K(\gamma)$ in the unicritical case (which may be seen by repeated application of \cite[Lemma 1.1]{MR0962740}; see also \cite{MR4188198} for related results). 
\begin{thm}\label{thm:BigGalois} 
Let $K$ be an $abc$-field of characteristic zero, let $S=\{x^d+c_1,\dots,x^d+c_s\}$ for some $c_1,\dots,c_s\in K$, and let $C_3=C_3(K)$ and $D_3=D_3(K)$ be as in Theorem \ref{thm:main+irred}. Moreover, assume that $K$ contains a primitive $d$th root of unity and that the following conditions hold:
\vspace{.05cm} 
\begin{enumerate}
\item[\textup{(1)}] $d\geq \max\{12,D_3\}$, \vspace{.1cm}   
\item[\textup{(2)}] $S$ contains an irreducible $\phi(x)=x^d+c$ with $h(c)>C_3$, and\vspace{.1cm}  
\item[\textup{(3)}] $S$ is not of the special form 
\[S\subseteq \Big\{x^d+c\,: h(c)\leq C_3\Big\}\cup \Big\{x^d+(\zeta_1 ry^p)-(\zeta_1 ry^p)^{d},\; x^d+\zeta_2 ry^p\;:\;\text{$\zeta_1,\zeta_2\in\mu_{K,d}$}\Big\}\]
for some $y\in K$, some $r\in\{1,-1\}$, and some prime divisor $p|d$. \vspace{.2cm} 
\end{enumerate}
Then there exists a sequence $\gamma\in\Phi_S$ such that $G_K(\gamma)$ is a finite index subgroup of $[C_d]^{\infty}$; in fact, there are infinitely many such $\gamma$ when $|S|\geq2$. Moreover, the probability that a random $\gamma\in \Phi_S$ furnishes a big dynamical Galois group over $K$ is positive:
\[\bar{\nu}\Big(\{\gamma\in \Phi_S\,: \text{$\gamma$ is stable over $K$ and $K_{n}(\gamma)/K_{n-1}(\gamma)$ is maximal i.o.} \}\Big)>0.\]
Furthermore, one can take $C(K)=0$ and $D(K)=\max\{14,4g_K+10\}$ in the function field case. Moreover, there is a lower bound on this probability that depends only on $K$, $d$, and $|S|$.  
\end{thm}

\begin{remark} We note that Theorem \ref{thm:BigGalois} implies Theorem \ref{cor:BigGalois} from the introduction. To see this we note that, as in the proof of Corollary \ref{cor:numberfield}, the conditions in Theorem \ref{cor:BigGalois} imply that condition (3) of Theorem \ref{thm:BigGalois} holds.      
\end{remark}

To prove Theorem \ref{thm:BigGalois}, we need the following maximality test; see \cite{JonesIMRN,MR4292933,Stoll} for analogous statements in both the constant and random sequence settings. We remind the reader that given a sequence $\gamma=(\theta_n)_{n=1}^\infty\in \Phi_S$ and an index $n\geq1$, the polynomial $\gamma_n$ is defined to be $\gamma_n:=\theta_1\circ\dots\circ\theta_n$.   
\begin{thm}\label{thm:GaloisMax} Let $K$ be a number field or a function field of characteristic zero and let $S=\{x^{d}+c_1,\dots,x^{d}+c_s\}$ for some $c_1,\dots,c_s\in K$ and $d\geq2$. Moreover, assume that $K$ contains a primitive $d$th root of unity, that $\gamma_n\in M_S$ is irreducible over $K$ for some $n\geq2$, and that there is a prime $\mathfrak{p}$ of $\mathcal O_K$ with the following properties:\vspace{.1cm}  
\begin{enumerate}
\item[\textup{(1)}] $v_{\mathfrak{p}}(d)=0$,\vspace{.1cm}
\item[\textup{(2)}] $v_{\mathfrak{p}}(c_i)\geq 0$ for all $1\leq i\leq s$, \vspace{.1cm} 
\item[\textup{(3)}] $v_{\mathfrak{p}}(\gamma_n(0))>0$ and $\gcd(v_{\mathfrak{p}}(\gamma_n(0)),d)=1$, and\vspace{.1cm} 
\item[\textup{(4)}] $v_{\mathfrak{p}}(\gamma_m(0))=0$ for all $1\leq m\leq n-1$. \vspace{.1cm}   
\end{enumerate}
Then the extension $K(\gamma_n)/K(\gamma_{n-1})$ is maximal, i.e., $[K(\gamma_n):K(\gamma_{n-1})]=d^{d^{n-1}}$.       
\end{thm}
\begin{remark}\label{rem:maximal} Since $K_n(\gamma)/K_{n-1}(\gamma)$ is the compositum of at most $d^{n-1}$ cyclic extensions of $K_{n-1}(\gamma)$, one for each root of $\gamma_{n-1}$, we see that $[K_n(\gamma):K_{n-1}(\gamma)]\leq d^{d^{n-1}}$. For this reason, we say that the extension $K_n(\gamma)/K_{n-1}(\gamma)$ is \emph{maximal} if the inequality above is an equality.  
\end{remark}
With Theorem \ref{thm:GaloisMax} in mind, we make the following definition. 
\begin{defin} A prime $\mathfrak{p}$ of $K$ satisfying conditions (1)-(4) in Theorem \ref{thm:GaloisMax} is called a \emph{good primitive prime divisor} of $\gamma_n(0)$ in $K$.     
\end{defin}
\begin{proof}[Proof of Theorem~\ref{thm:GaloisMax}] The proof is essentially identical to \cite[Theorem 4.3]{JonesIMRN}: we may replace $g\circ f^{n-1}$ with $\gamma_{n-1}$ and $f$ with $\theta_n$ in the proof of \cite[Lemma 4.1]{JonesIMRN} to deduce that the extension $K(\gamma_n)/K(\gamma_{n-1})$ is maximal if and only if $\gamma_n(0)$ is not a $p$th power in $K_{n-1}(\gamma)$ for all primes $p|d$; here, as in \cite{JonesIMRN}, the key step is to use \cite[Lemma 4.2]{JonesIMRN} with the group $G=\Gal(K_{n-1}(\gamma)/K)$, which applies since $G$ is solvable (as $g$ is solvable by radicals \cite[Theorem 7.2]{MR1878556}) and has order dividing a power of $d$ (since $G$ is naturally a subgroup of $[C_d]^{n-1}$, an iterated wreath product of cyclic groups of order $d$, seen by repeated applying \cite[Lemma 1.1]{MR0962740}).  

Next, we note that the discriminant formula for $\gamma_{n-1}$ in  \cite[Proposition 6.2]{MR4292933} together with standard facts relating discriminants, differents, and ramification \cite[III, \S1-2]{MR1282723} imply that the primes $\mathfrak{p}$ in $K$ that ramify in $K_{n-1}(\gamma)$ must satisfy: $v_{\mathfrak{p}}(d)>0$, or $v(c_i)<0$ for some $1\leq i\leq s$, or $v_{\mathfrak{p}}(\gamma_m(0))>0$ for some $1\leq m\leq n-1$; alternatively, one may apply the argument, mutatis mutandis, used to prove \cite[Proposition 3.1]{MR3703943}. Hence, if we assume that $\mathfrak{p}$ is a good primitive prime divisor of $\gamma_n(0)$, then $\mathfrak{p}$ is unramified in $K_{n-1}(\gamma)$. On the other hand, since we also have that $\gcd(v_{\mathfrak{p}}(\gamma_n(0)),d)=1$, it follows that $\mathfrak{p}$ must ramify in $K(\sqrt[p]{\gamma_n(0)})$ for every prime $p|d$. Putting these two facts together, we deduce that $\gamma_n(0)$ cannot be a $p$th power in $K_{n-1}(\gamma)$ and the extension $K(\gamma_n)/K(\gamma_{n-1})$ is maximal as desired.                                
\end{proof}
\begin{remark}\label{rem:problem+different d's} The step that breaks down in the proof of Theorem \ref{thm:GaloisMax} when considering sets $S=\{x^{d_1}+c_1,\dots,x^{d_s}+c_s\}$ with (possibly) \emph{distinct degrees} is the use of \cite[Lemma 4.2] {JonesIMRN}: the Galois group $G=\Gal(K_{n-1}(\gamma)/K)$, which is still solvable, may not have order dividing a power of $d=\deg(\theta_n)$. In particular, a non-trivial $(\mathbb{Z}/d\mathbb{Z})[G]$-module may only have a trivial set of $G$-invariant elements in this case, a property that is used to prove that the extension $K(\gamma_n)/K(\gamma_{n-1})$ is maximal when $\gamma_n(0)$ is not a $p$th power in $K_{n-1}(\gamma)$. However, it is possible that the proof above works, for example, when all of the $d_i$ are powers of a single integer.             
\end{remark}
With this in mind, we show that for $d$ sufficiently large there are many $\gamma_n\in M_S$ with good primitive prime divisors in $K$, assuming the $abc$ conjecture. In fact, we can do this for semigroups generated by sets with possibly different degrees.  
\begin{prop}\label{prop:newprime} 
Let $K$ be an $abc$-field of characteristic zero and let $S=\{x^{d_1}+c_1,\dots,x^{d_s}+c_s\}$ for some $c_1,\dots,c_s\in K$ and integers $d_1,\ldots,d_s$ such that $\max\{h(c_1),\ldots,h(c_s)\} \ge \frac12\log2$ and $\min\{d_1,\dots,d_s\}\geq12$. 
Then there exists $N'=N'(K,S)$ such that if the following conditions are satisfied:\vspace{.1cm} 
\begin{enumerate} 
\item[\textup{(1)}] $\gamma_n=\theta_1\circ\dots\circ\theta_n$ for some $\theta_1,\dots,\theta_n\in S$ and $n\geq N'$, \vspace{.2cm} 
\item[\textup{(2)}] 
$\theta_n(x)=x^{d_i}+c_i$ with $h(c_i)=\max\{h(c_1),\dots,h(c_s)\}$, and
\vspace{.2cm} 
\item[\textup{(3)}] $\theta_1$ has no roots in $K$,\vspace{.05cm}  
\end{enumerate}
then there is a prime $\mathfrak{p}$ of $K$ with all of the following properties:\vspace{.1cm}  
\begin{enumerate}
\item[\textup{(a)}] $v_{\mathfrak{p}}(d_r)=0$ for all $1\leq r\leq s$,\vspace{.1cm}
\item[\textup{(b)}] $v_{\mathfrak{p}}(c_r)\geq 0$ for all $1\leq r\leq s$,
\vspace{.1cm} 
\item[\textup{(c)}] $v_{\mathfrak{p}}(\gamma_n(0))=1$, and
\vspace{.1cm} 
\item[\textup{(d)}] $v_{\mathfrak{p}}(\gamma_m(0))=0$ for all $1\leq m\leq n-1$. \vspace{.1cm}   
\end{enumerate}
In particular, $\gamma_n(0)$ has a good primitive prime divisor in $K$ when $d_1=d_2\dots=d_s$.    
\end{prop}
Before giving the proof of this result, we need the following height estimate. 
\begin{lem}\label{lem:heightratio} Let $K$ be a number field or a function field of characteristic zero and let $S=\{x^{d_1}+c_1,\dots,x^{d_s}+c_s\}$ for some $c_1,\dots,c_s\in K$ and some $d_1,\dots,d_s\geq2$. Then 
\[\bigg|\frac{h(f(P))}{\deg(f)}-h(P)\bigg|\leq \frac{1}{\min\{d_1,\dots,d_s\}-1}\Big(\max\{h(c_1),\dots,h(c_s)\}+\log2\Big)\]
for all $f\in M_S$ and all $P\in\mathbb{P}^1(\overline{K})$. 
\end{lem}
\begin{proof} Combine \cite[Lemma 4.1]{MR4292933} with the simple fact that $|h(P^d+c)-dh(P)|\leq h(c)+\log2$.      
\end{proof}
We also need \cite[Proposition 2.2]{MR3703943}, sometimes called the \emph{Roth-$abc$ property}; for a proof, see Propositions 3.4 and 4.2 in \cite{MR3138487}. It is worth pointing out that in the case where $K$ is a function field, the result below does not follow directly from the usual $abc$ conjecture but rather Yamanoi’s proof \cite{MR2096455} of the Vojta conjecture for algebraic points on curves over function fields of characteristic zero. 
\begin{prop}\label{prop:abc+polyrad} Let $K$ be an $abc$-field of characteristic zero, let $\epsilon>0$, and let $F\in K[x]$  have degree $d\ge3$ and no repeated factors. Then there is a constant $C(F,\epsilon)$ such that
\[\sum_{v_\mathfrak{p}(F(a))>0} N_{\mathfrak{p}}\geq (d-2-\epsilon)h(a)+C(F,\epsilon)\]
holds for all $a\in K$.  
\end{prop}
We now have the tools in place to produce functions $\gamma_n\in M_S$ with good primitive prime divisors in $K$ subject to some constraints; however, we do not need to assume that the degrees of the maps in $S$ are identical to prove this result.  
\begin{proof}[Proof of Proposition \ref{prop:newprime}] We begin with some notation. Let $h(c_i)=\max\{h(c_1),\dots,h(c_s)\}$ and $d_j=\min\{d_1,\dots,d_s\}$. Next, for $\gamma_n=\theta_1\circ\dots\circ\theta_n$ with $\theta_1,\dots,\theta_n\in S$ and $1\leq m\leq n$, write $d_{\theta_m} = \deg(\theta_m)$. Moreover, assume that $\theta_1$ and $\theta_n$ satisfy conditions (2) and (3) of Proposition \ref{prop:newprime}. In particular, we note that   
\[\gamma_m(0)=\theta_1\circ\dots\circ\theta_m(0)\neq0\;\;\; \text{for all $m\geq1$},\] 
since $\theta_1$ has no roots in $K$, and that $\theta_1$ has no repeated factors in $\overline{K}$, since $\theta_1(0)\neq0$ and $K$ has characteristic zero. 

Now let $a=\theta_2\circ\dots\circ\theta_n(0)$ so that Lemma \ref{lem:heightratio} applied to the map $f=\theta_2\circ\dots\circ\theta_{n-1}$ and the point $P=\theta_n(0)=c_i$  implies that
\begin{equation}\label{eq:newprime1}
\begin{split}
h(a)=h(\theta_2\circ\dots\circ\theta_{n}(0))&\geq (d_{\theta_2}\cdots d_{\theta_{n-1}})\Big(h(c_i)-\frac{1}{d_j-1}(h(c_i)+\log2)\Big)\\[5pt] 
&\geq(d_{\theta_2}\cdots d_{\theta_{n-1}})\frac{(d_j-4)}{(d_j-1)}h(c_i).
\end{split} 
\end{equation}
Here we use that $\theta_n(x)=x^{d_i}+c_i$ and that $h(c_i)\geq\frac{1}{2}\log2$. Moreover, Proposition \ref{prop:abc+polyrad} applied to the polynomial $F=\theta_1$ with $\epsilon=1$ implies that
\begin{equation}\label{eq:newprime2}  
\sum_{v_\mathfrak{p}(\gamma_n(0))>0} N_{\mathfrak{p}}=\sum_{v_\mathfrak{p}(\theta_1(a))>0} N_{\mathfrak{p}}\geq (d_{\theta_1}-3)h(a)+C_{10}
\end{equation}
for some $C_{10}$ depending on $\theta_1\in S$. On the other hand, \eqref{bd:productformula} and the standard fact that $h\circ\theta_1=d_{\theta_1}h+O(1)$ together imply that \vspace{.1cm}    
\begin{equation}\label{eq:newprime3}  
\sum_{v_\mathfrak{p}(\gamma_n(0))>0}v_\mathfrak{p}(\gamma_n(0))N_{\mathfrak{p}}=\sum_{v_\mathfrak{p}(\theta_1(a))>0} v_\mathfrak{p}(\theta_1(a))N_{\mathfrak{p}}\leq h(\theta_1(a))\leq d_{\theta_1}h(a)+C_{11} 
\end{equation}
for some $C_{11}$ depending on $\theta_1\in S$. Now consider the sets  $Y_{1,n}(\gamma):=\{\mathfrak{p}:v_{\mathfrak{p}}(\gamma_n(0))=1\}$ and $Y_{2,n}(\gamma):=\{\mathfrak{p}:v_{\mathfrak{p}}(\gamma_n(0))\geq2\}$. Then combining \eqref{eq:newprime2} and \eqref{eq:newprime3} implies that \vspace{.05cm}     
\begin{equation}\label{eq:newprime4}
 \begin{split}
2(d_{\theta_1}-3)h(a)+2C_{10}&\scalemath{.89}{\leq 2 \bigg(\sum_{v_\mathfrak{p}(\gamma_n(0))>0} N_{\mathfrak{p}}\bigg)=\sum_{\mathfrak{p}\in Y_{1,n}(\gamma)}N_{\mathfrak{p}}\;+\;\bigg( \sum_{\mathfrak{p}\in Y_{1,n}(\gamma)} N_{\mathfrak{p}}+\sum_{\mathfrak{p}\in Y_{2,n}(\gamma)} 2N_{\mathfrak{p}}\bigg)}\\[7pt]
 &\scalemath{.93}{\leq\sum_{\mathfrak{p}\in Y_{1,n}(\gamma)}N_{\mathfrak{p}}\;+\bigg(
 \sum_{\mathfrak{p}\in Y_{1,n}(\gamma)} v_\mathfrak{p}(\gamma_n(0))N_{\mathfrak{p}}+
 \sum_{\mathfrak{p}\in Y_{2,n}(\gamma)} v_\mathfrak{p}(\gamma_n(0))N_{\mathfrak{p}}
 \bigg)}
 \\[7pt] 
 &=\sum_{\mathfrak{p}\in Y_{1,n}(\gamma)}N_{\mathfrak{p}}\;+\; \sum_{v_\mathfrak{p}(\gamma_n(0))>0} v_\mathfrak{p}(\gamma_n(0))N_{\mathfrak{p}}\\[7pt] 
&\leq \sum_{\mathfrak{p}\in Y_{1,n}(\gamma)}N_{\mathfrak{p}}\;+\;d_{\theta_1}h(a)+C_{11}.
     \end{split} 
\end{equation}
Therefore, we deduce from \eqref{eq:newprime1} and \eqref{eq:newprime4} that 
\begin{equation}\label{eq:newprime5}
\sum_{\mathfrak{p}\in Y_{1,n}(\gamma)} N_{\mathfrak{p}}\geq (d_{\theta_1}-6)h(a)+C_{12}\geq(d_{\theta_1}-6)(d_{\theta_2}\cdots d_{\theta_{n-1}})\frac{(d_j-4)}{(d_j-1)}h(c_i) \;+\;C_{12}  
\end{equation}
for some $C_{12}$ depending on $\theta_1\in S$. Now for the crucial step: if we assume that for every prime $\mathfrak{p}\in Y_{1,n}(\gamma)$ either $v_{\mathfrak{p}}(d_r)>0$ for some $1\leq r\leq s$, or $v_{\mathfrak{p}}(c_r)<0$ for some $1\leq r\leq s$, or  $v_{\mathfrak{p}}(\gamma_m(0))>0$ for some $1\leq m\leq n-1$, then we see that 
\begin{equation}\label{eq:newprime6}
 \begin{split}
 \sum_{\mathfrak{p}\in Y_{1,n}(\gamma)} N_{\mathfrak{p}}&\leq\;
 \sum_{m=1}^{n-1}\Big(\sum_{v_{\mathfrak{p}}(\gamma_m(0))>0}N_{\mathfrak{p}}\Big)\;+\;
 \sum_{r=1}^s\Big(\sum_{v_{\mathfrak{p}}(d_r)>0} N_{\mathfrak{p}}\Big)\;+\;
 \sum_{r=1}^s\Big(\sum_{v_{\mathfrak{p}}(c_r)<0}N_{\mathfrak{p}}\Big)\\[6pt]
 &\leq \;\sum_{m=1}^{n-1}h(\gamma_m(0))\;+\; C_{13}\\[6pt]
 &\leq \sum_{m=1}^{n-1}(d_{\theta_1}\cdots d_{\theta_m})\frac{1}{d_j-1}(h(c_i)+\log2)\;+\; C_{13}\\[6pt] 
&=\frac{1}{d_j-1}(h(c_i)+\log2)\sum_{m=1}^{n-1}(d_{\theta_1}\cdots d_{\theta_m})\;+\;C_{13} 
     \end{split}       
\end{equation}
for some $C_{13}$ depending on $K$ and $S$. Here we use Lemma \ref{lem:heightratio} applied to the maps $f=\gamma_m$ and $P=0$ to give an upper bound for $h(\gamma_m(0))$. Hence, combining \eqref{eq:newprime5} and  \eqref{eq:newprime6} with the fact that $h(c_i)\geq \frac{1}{2}\log2$ and the fact that $d_{\theta_1}-6\geq \frac{1}{2}d_{\theta_1}$ since $d_{\theta_1}\geq d_j\geq12$, we see that
\begin{equation}\label{eq:newprime7}
\begin{split} 
\frac{1}{2}(d_{\theta_1}\cdots d_{\theta_{n-1}})\frac{(d_j-4)}{(d_j-1)}
&\leq(d_{\theta_1}-6)(d_{\theta_{2}}\cdots d_{\theta_{n-1}})\frac{(d_j-4)}{(d_j-1)}\\[5pt] 
&\leq\frac{3}{(d_j-1)}\sum_{m=1}^{n-1}(d_{\theta_1}\cdots d_{\theta_m})\;+\;C_{14}
\end{split}
\end{equation}
for some $C_{14}$ depending on $K$ and $S$. Therefore, we deduce from \eqref{eq:newprime7} above that \vspace{.05cm}      
\begin{equation}
\begin{split} 
\label{eq:newprime8}
\frac{(d_j-4)}{(d_j-1)}&\leq \frac{6}{d_j-1}\Big(1+\frac{1}{d_{\theta_{n-1}}}+\frac{1}{d_{\theta_{n-1}}d_{\theta_{n-2}}}+\dots+\frac{1}{d_{\theta_{n-1}}\cdots d_{\theta_2}}\Big)\;+\; \frac{2C_{14}}{d_{\theta_1}\cdots d_{\theta_{n-1}}}\\[6pt] 
&\leq\frac{6}{(d_j-1)}\Big(1+\frac{1}{d_j}+\frac{1}{d_j^2}+\dots \Big)+\frac{2C_{14}}{d_j^{n-1}}\\[6pt]
&=\frac{6d_j}{(d_j-1)^2}+\frac{2C_{14}}{d_j^{n-1}}.  
\end{split} 
\end{equation}
On the other hand, it is straightforward to verify that
\vspace{.1cm} 
\begin{equation}
\label{eq:newprime9}
0<\frac{16}{121}\le\frac{(d-4)}{(d-1)}-\frac{6d}{(d-1)^2}\;\;\text{ for all $d\geq12$.}\vspace{.1cm} 
\end{equation}
In particular, combining \eqref{eq:newprime8} and \eqref{eq:newprime9} we see that $n$ is bounded: there exists a constant $N'(K,S)$ such that if $n\geq N'(K,S)$ and $\gamma_n$ satisfies conditions (1)-(3), then $Y_{1,n}(\gamma)$ contains a prime $\mathfrak{p}$ satisfying $v_\mathfrak{p}(d_r)=0$ for all $1\leq r\leq s$, $v(c_r)\geq 0$ for all $1\leq r\leq s$, and $v_\mathfrak{p}(\gamma_{m}(0))=0$ for all $1\leq m\leq n-1$. In particular, since $v_\mathfrak{p}(\gamma_{n}(0))=1$ by definition of $Y_{1,n}(\gamma)$, it follows that $\mathfrak{p}$ is a good primitive prime divisor of $\gamma_n(0)$ in the case of equal degrees as desired.        
\end{proof}
We now combine Theorem \ref{thm:main+irred}, Theorem \ref{thm:GaloisMax}, and Proposition \ref{prop:newprime} to prove Theorem~\ref{thm:BigGalois}.

\begin{proof}[Proof of Theorem \ref{thm:BigGalois}] Assume that $d_1=d_2=\dots=d_s$ and that conditions (1)-(3) hold for $S$. Then the proof of Theorem \ref{thm:main+irred} implies that there exists $g\in M_S$ such that $\{g\circ f:f\in M_S\}$ is a set of irreducible polynomials in $K[x]$; in fact, there exist $N(K,d)$ and $\phi_1$ and $\phi_2$ in $S$ such that one of the polynomials $\phi_1^N$, $\phi_1^N\circ\phi_2$, or $\phi_1^N\circ\phi_2^N$ works for $g$. From here, write $g=\tau_1\circ\tau_2\dots\tau_{n_0}$ for some $\tau_i\in S$ and $n_0\geq1$. Moreover, fix $\theta\in S$ such that $\theta(x)=x^d+c$ with $h(c)=\max\{h(c_1),\dots,h(c_s)\}$ and note that $h(c)\geq C_3\geq C_4=\frac{3}{5}\log2$ by definition of $C_3(K)$ and $C_4(K)$. Finally, consider the sequences 
\[\mathcal{G}(S):=\Big\{\gamma=(\theta_1,\theta_2,\dots)\in\Phi_S\,:\, 
\text{$\theta_i=\tau_i$ for all $1\leq i\leq n_0$ and $\theta_n=\theta$ i.o.} 
\Big\}.\]
Then for all $\gamma\in\mathcal{G}(S)$ and all (but finitely many) of the infinitely many $n$ satisfying $\theta_n=\theta$, the polynomial $\gamma_n=\theta_1\circ\dots\circ\theta_n$ satisfies the conditions of Proposition \ref{prop:newprime}; note that $\theta_1$ has no roots in $K$ as $g$ is irreducible. Hence, $\gamma_n(0)$ has a good primitive prime divisor in $K$ for such $n$ and $\gamma_n$ is irreducible by construction of $g$. In particular, the extensions $K(\gamma_n)/K(\gamma_{n-1})$ are maximal by Theorem \ref{thm:GaloisMax} for infinitely many $n$. Moreover, the measure $\bar{\nu}(\mathcal{G}(S))$ is at least $|S|^{-2N(K,d)}$, since the set of sequences with $\theta$ appearing infinitely many times has full measure in $\Phi_S$ by the Borel-Cantelli Theorem (more specifically, the Monkey and Typewriter Problem). Likewise, we note that any sequence of the form 
\[\gamma=(\tau_1,\dots,\tau_{n_0},\omega_1,\omega_2\dots,\omega_m,\theta,\theta,\dots),\]for some arbitrary $m\geq0$ and arbitrary $\omega_1,\dots,\omega_m\in S$, furnishes a dynamical Galois group that is a finite index subgroup of $[C_d]^{\infty}$; here we use Proposition \ref{prop:newprime} and Theorem \ref{thm:GaloisMax} to ensure that the subextensions $K(\gamma_n)/K(\gamma_{n-1})$ are maximal for all $n>N'$. In particular, there are infinitely many such sequences when $S$ contains at least two distinct maps as claimed.     
\end{proof}

\begin{remark}
    All of the infinitely many $\gamma \in \Phi_S$ constructed in the proof to have dynamical Galois groups with finite index in $[C_d]^\infty$ are eventually constant sequences. Since $S$ is a finite set, this implies that we have constructed only countably many such $\gamma$; since $\Phi_S$ is uncountable, this further implies that the infinite set of $\gamma$ we have produced to have this property still has measure zero.
\end{remark}

\begin{remark} It is likely that when $\gamma=(\theta_1,\theta_2,\dots)$ is stable and generates maximal subextensions infinitely often (i.e., $\gamma$ furnishes a \emph{big dynamical Galois group} over $K$ in the sense of \cite{MR4680482}), then the Dirichlet density of the set of prime divisors of the sequence \[\theta_1(a),\ \theta_1(\theta_2(a)),\ \theta_1(\theta_2(\theta_3(a))),\ \dots\] 
is zero for all $a\in K$ (and infinite for appropriate $a$); see \cite{JonesIMRN} for the case of iterating a single function for general $d$ and \cite[Theorem 1.5]{MR4680482} for the case of multiple maps when $d=2$.       \end{remark}

On the other hand, even in the case of unequal degrees, the primes $\mathfrak{p}$ in Proposition \ref{prop:newprime} are useful: they must ramify in $K_n(\gamma)$ and not in $K_{m}(\gamma)$ for all $m<n$; c.f. \cite[Theorem 1.1]{MR3703943}. 
\begin{cor}\label{cor:newramifiedprime} Let $K$ be an $abc$-field of characteristic zero, let $S=\{x^{d_1}+c_1,\dots,x^{d_s}+c_s\}$ for some $c_1,\dots,c_s\in K$ and some $\min\{d_1,\dots,d_s\}\geq12$. Moreover, assume that $\gamma_n$ satisfies conditions \textup{(1)-(3)} of Proposition \ref{prop:newprime}. Then the associated prime $\mathfrak{p}$ satisfying \textup{(a)-(d)} ramifies in $K_n(\gamma)$ and not in $K_{m}(\gamma)$ for all $m<n$.   
\end{cor}
\begin{proof} The discriminant formula for $\gamma_m$ in \cite[Proposition 6.2]{MR4292933} implies that $\mathfrak{p}$ is unramified in $K_{m}(\gamma)$ for all $m<n$ (here, we may also apply \cite[II.4, Proposotion 7]{MR1282723} after passing to completions $\mathfrak{p}$); compare to \cite[Proposition 3.1]{MR3703943}. 

On the other hand, a simplified version of the argument given in \cite[Proposition 3.2]{MR3703943} implies that $\mathfrak{p}$ ramifies in $K_n(\gamma)$. To see this, suppose that $\mathfrak{p}$ is unramified in $K_n(\gamma)$ and choose a prime $\mathfrak{P}|\mathfrak{p}$ in $K_n(\gamma)$. In particular, we have that $v_{\mathfrak{p}}(a)=v_{\mathfrak{P}}(a)$ for all $a\in K$ by assumption. Now write $\gamma_n(x)=x^N+\dots+a_dx^d+a_0$ for some $a_i\in K$ and some $d>1$; here we use that $\gamma_n$ is a composition of unicritical polynomials and so has no linear term. Moreover, we have that $v_{\mathfrak{P}}(a_i)\geq0$ for all $i\geq d$ and $v_{\mathfrak{P}}(a_0)=1$ since $\min\{v_{\mathfrak{P}}(c_1),\dots,v_{\mathfrak{P}}(c_1)\}\geq0$ and $a_0=\gamma_n(0)$ respectively. Therefore, the first segment of the $\mathfrak{P}$-adic Newton polygon of $\gamma_n$ is the line from $(0,1)$ to $(\ell,0)$ for some $\ell\geq d>1$. Hence, $\gamma$ has a root $r\in K_n(\gamma)$ with valuation $v_{\mathfrak{P}}(r)=1/\ell$; see, for instance, \cite[Theorem 7.4.7]{MR4175370} or \cite[IV.3]{MR0754003}. But $v_\mathfrak{P}(r)\in\mathbb{Z}$, and so we reach a contradiction. Thus $\mathfrak{p}$ ramifies in $K_n(\gamma)$ as claimed.      
\end{proof}
In particular, we obtain the claim made in Remark~\ref{rem:newramifiedprime} from the introduction.

\bibliographystyle{plain}
\bibliography{FunctionFields}
\bigskip 

\end{document}